\documentclass[12pt]{article}
\usepackage{mathptmx}
\usepackage{anyfontsize}
\usepackage{graphicx,psfrag,amssymb,amsmath,amsthm,comment}
\usepackage{latexsym}
\usepackage{graphicx,amssymb,amsmath,amsthm}
 \usepackage{color}
 \usepackage{psfrag}
\usepackage[
bookmarks=true,
bookmarksnumbered=false, 
bookmarksopen=false, 
colorlinks=true,
linkcolor=blue, citecolor=blue]{hyperref}

\usepackage{nomencl}
\makenomenclature

 \usepackage{makeidx}
\usepackage{pdfsync}
\usepackage{hyperref}
\usepackage{color}
\usepackage{framed}

\usepackage{eucal}
\def\beq{\begin{equation}}
\def\eeq{\end{equation}}

\def\dist{{\rm dist}}

\def\supp{{\rm supp}}

\def\ba{\textbf{a}}

\def\bh{\textbf{h}}
\def\bH{\textbf{H}}
\def\bI{\mathbf{I}}

\def\ba{\mathbf{a}}

\def\bc{\mathbf{c}}
\def\bd{\mathbf{d}}

\def\bh{\mathbf{h}}

\def\br{\mathbf{r}}
\def\bs{\mathbf{s}}

\def\bx{\mathbf{x}}

\def\b1{{\boldsymbol{1}}}

\def\cA{\mathcal{A}}
\def\cB{\mathcal{B}}

\def\cD{\mathcal{D}}
\def\cE{\mathcal{E}}
\def\cF{\mathcal{F}}
\def\cG{\mathcal{G}}
\def\cH{\mathcal{H}}

\def\cJ{\mathcal{J}}\def\eps{\varepsilon}
\def\cK{\mathcal{K}}

\def\cM{\mathcal{M}}
\def\cN{\mathcal{N}}
\def\cO{\mathcal{O}}

\def\cR{\mathcal{R}}
\def\cS{\mathcal{S}}
\def\cT{\mathcal{T}}
\def\cU{\mathcal{U}}
\def\cV{\mathcal{V}}
\def\cW{\mathcal{W}}

\def\cZ{\mathcal{Z}}

\def\IE{{\mathbb E}}

\def\IZ{{\mathbb Z}}

\def\fF{\mathfrak{F}}
\def\fG{\mathfrak{G}}

\def\tS{\tilde{S}}

\def\hmu{\hat{\mu}}

\def\eps{\varepsilon}
\def\hcW{\hat{\mathcal{W}}}

\def\dist{\text{\rm dist}}

\def\Cov{\text{\rm Cov}}

\newtheorem{theorem}{Theorem}

\newtheorem{lemma}[theorem]{Lemma}
\newtheorem{proposition}[theorem]{Proposition}
\newtheorem{corollary}[theorem]{Corollary}
\newtheorem{remark}{Remark}
\newtheorem{defn}{Definition}

\def\pQ{\partial Q}

\numberwithin{theorem}{section}
\numberwithin{equation}{section}

\begin{document}
\title{Optimal bounds for decay of correlations  and $\alpha$-mixing for nonuniformly hyperbolic dynamical systems
}

\author{Sandro Vaienti\thanks{Aix Marseille Universit\'e, Universit\'e de Toulon, CNRS, CPT, 13009 Marseille, France. vaienti@cpt.univ-mrs.fr}
\and Hong-Kun Zhang \thanks{Department of Math. \& Stat.,
University of Massachusetts Amherst, MA 01003,
hongkun@math.umass.edu}\\}
\date{Dedicated to the memory of Nikolai Chernov\thanks{A first version of this paper was prepared together with N. Chernov  three years ago and a few preliminary results were presented in some conferences. The paper was deeply modified after N. Chernov passed away and we now dedicated it to his memory. We are greatly indebted to the ideas and the suggestions he gave us.}}\maketitle

\begin{abstract}
We investigate the decay rates of correlations for
nonuniformly hyperbolic systems with or without singularities,  on piecewise H\"older observables. By constructing a new scheme of coupling methods using the probability renewal theory, we obtain the optimal  bounds for  decay rates of correlations for a large class of such observables. We also establish the alpha-mixing property for  time series generated by these systems, which leads to a vast ranges of limiting theorems.  Our results apply to rather general  hyperbolic systems with singularities, including Bunimovich flower billiards, semidispersing billiards on a rectangle and billiards with cusps, and other   nonuniformly hyperbolic maps.\end{abstract}

\centerline{AMS classification numbers: 37D50, 37A25}

\tableofcontents
\printindex
\section{Introduction and the main results}

\subsection{Motivation and relevant works}

The studies of the statistical properties of 2-dimensional hyperbolic
systems with singularities are motivated in large part by mathematical
billiards with chaotic behavior, introduced by Sinai in \cite{Sin} and
since then  studied extensively by many authors \cite{BSC90, BSC91, Y98,Y99,CM}.

Statistical properties of chaotic dynamical systems are described by
the decay of correlations and by various limiting theorems. Let
$(\cM,\cF,\mu)$ be a \emph{dynamical system}, i.e. a measurable transformation
$\cF\colon \cM\to\cM$ preserving a probability measure $\mu$ on the Borel sigma algebra of $\cM$.
\nomenclature{$\cF$}{The nonuniformly hyperbolic map}
For any real-valued functions $f$ and $g$ on $\cM$ (often called
\emph{observables}) the {\em correlations} are defined by
\beq
    C_n(f,g,\cF) = \int_{\cM} (f\circ \cF^n)\, g\, d\mu -
    \int_{\cM} f\, d\mu    \int_{\cM} g\, d\mu.
       \label{Cn}
\eeq
Note that (\ref{Cn}) is well defined for all $f,g\in L^2_{\mu}(\cM)$.
It is a standard fact that $(\cF,\mu)$ is {\em mixing} if and only if
\beq \label{Cto0}
   \lim_{n\to\infty} C_n(f,g,\cF) = 0,
   \qquad \forall  f,g\in L^2_{\mu}(\cM).
\eeq
The statistical properties of the system $(\cM,\cF,\mu)$ are characterized by
the \emph{rate of decay of correlations}, i.e., by the speed of
convergence in (\ref{Cto0}) for ``good enough'' functions $f$ and $g$.
If $\cM$ is a manifold and $\cF$ is a smooth (or piecewise smooth) map,
then ``good enough'' usually means bounded and (piecewise) H\"older
continuous.

Generally, mixing dynamical systems (even very strongly mixing ones,
such as Bernoulli systems) may exhibit quite different  statistical properties, depending on the rate of the
decay of correlations. If correlations decay exponentially fast (i.e.,
$|C_n| =\cO( e^{-an})$ with $ a>0$), usually the classical Central Limit
Theorem (CLT) holds, as well as many other probabilistic limit laws,
such as Weak Invariance Principle (convergence to Brownian motion),
which play a crucial role in applications to statistical mechanics; we
refer the reader to the surveys in \cite{Ch95,CD,Den89,D04,Y98} and
\cite[Chapter~7]{CM}. Such strongly chaotic dynamical systems behave
very much like sequences of i.i.d.\ (independent identically
distributed) random variables in probability theory.
However, some other mixing and Bernoulli systems have slow rates of the
decay of correlations, such as $|C_n|=\cO( n^{-a})$. Their
statistical properties are usually weak, they exhibit \emph{intermittent} behavior \cite{PM80}: intervals of chaotic motion
are followed by long periods of regular oscillations, etc. Such systems can help to understand the transition from regular  to chaotic motion, and for that  they have long attracted considerable interest in physics community \cite{FM88,Ma83,VCG}.


 A challenging question to ask is ``What are the main reasons that have slowed down the decay rates of correlations for nonuniformly hyperbolic systems, especially for hyperbolic systems with singularities including chaotic billiards?"
 The main difficulty to solve this question is caused by
singularities and the resulting fragmentation of phase space during the evolution of the
dynamics, which slows down the global expansion of unstable manifolds.
Moreover, the differential of the billiard map is unbounded and has
unbounded distortion near the singularities, which aggravates the
analysis of correlations: one has to subdivide the vicinity of singularities into countably many ``shells'' in which distortions can be effectively controlled.

Even for strongly chaotic billiards, exponential upper bounds on
correlations were proven only in 1998 when Young \cite{Y98} introduced
her tower construction as a universal tool for the description of
nonuniformly hyperbolic maps; see also \cite{C99}. Young also
sharpened her estimates on correlations by combining her tower
construction with a coupling technique borrowed from probability
\cite{Y99}. The coupling method was further developed by Bressaud and
Liverani in \cite{BL02} and then reformulated in pure dynamical terms
(without explicit tower construction) by  Dolgopyat \cite{CD,D01} using standard
pairs; see also \cite[Chapter~7]{CM}. Dolgopyat's technique was proven
to be efficient in handling various types of strongly chaotic systems with singularities
\cite{CD,CDgalton,CDlorentz,CM,CZ09,D01,D04b}.

For nonuniformly hyperbolic  billiards, the first rigorous upper bounds on
correlations  were obtained in the mid-2000s
\cite{M04,CZ05a}. The  results on upper bounds were not optimal, as they included an
extra logarithmic factor, which was later removed in \cite{CZ08} by a
finer analysis of return time statistics. Besides billiards, the same
general scheme for bounding correlations has been applied to
linked-twist maps \cite{SS} and generalized baker's transformations
\cite{BM}, intermittent symplectic maps \cite{LM}, and solenoids \cite{AP}.

Surprisingly little progress has been made in obtaining \emph{lower}
bounds on correlations for hyperbolic systems, including billiards.
Among rare results in this direction are those for Bunimovich stadia
\cite{GB} and billiards with cusps \cite{BCD}, where a rough lower bound was obtained as
a byproduct of a non-classical CLT that forces correlations to be at
least of order $\cO(n^{-1})$. For one dimensional non-uniformly expanding maps and for Markov maps,
lower bounds on correlations have been derived via the renewal methods
by Sarig \cite{Sr}, later improved by Gou\"{e}zel \cite{GO}. The
renewal techniques were then extended to more general non-uniformly
expanding maps and certain nonuniformly hyperbolic systems; see \cite{Hu,HV13,MT14,LT} and references therein. The main
scheme of the renewal methods relies on the construction of an induced
map, for which the corresponding transfer operator has a spectral gap
on a certain functional space. Actually the main reason why it is so
difficult to apply operator technique to billiards and related hyperbolic systems with singularities, is the lack of a suitable
functional space on which the transfer operator for the induced map
would have a spectral gap (and would be aperiodic).

For chaotic billiards
and their perturbations, a suitable Banach space of functions was
constructed in \cite{DZ11,DZ13}, and the spectral gap for
transfer operators was proven to exist. But it is still difficult to apply the renewal operator
methods on these systems because in order to take care of the unbounded differential of
the billiard map, the norms defined in \cite{DZ11,DZ13}  cannot directly produce
the necessary estimates needed for the renewal technique \cite{Sr}. We
should however stress that similar difficulties are also encountered in
the studies of non-uniformly expanding (non-invertible) maps, as it
is pointed out in \cite{HV13}. Recently, some progress has been made in this direction for some hyperbolic maps, see \cite{LT,MT14}.

In this paper we are able to identify the main factors that affect the decay rates of correlations for rather general nonuniformly hyperbolic systems, in terms of the tail distribution of the return time function (used in the inducing scheme).
  Since the singularities of the systems make it difficult to apply all existing  methods, we  revisit  Dolgopyat's coupling method and the ideas of standard pairs \cite{CD,D01} for systems with exponential decay rates of correlations, see also \cite{Y99}, and develop a new coupling scheme for nonuniformly hyperbolic systems.  Combining with the elegant probability Renewal Theory, originated from Kolmogorov \cite{Kol}, we are able to obtain an optimal bound for the decay rates of correlations for general 2-dimensional hyperbolic systems.
Our formulas give a precise asymptotic, rather than upper or lower
bounds. To our knowledge, this is the first result of that type in the
context of nonuniformly hyperbolic billiards. And the results have greatly improved all existing  results on decay rates of correlations for hyperbolic systems with slow decay rates of correlations. Moreover, our new method is more flexible, comparing to all other existing methods in this direction, as it  can be applied to dynamical systems under small deterministic or random perturbations. In Section 10, we will describe several classes of billiards to which our results can be applied.  Other nonuniform hyperbolic maps, some of which exhibiting attractors, are currently under  investigations.

We conclude this introduction by quoting forthcoming contributions and  addressing new questions. One of the main difficulty  of the inducing approach is to control the rate of expansion and contraction along the invariant manifolds for the original map, being those rates exponential for the first return map.  We deal with that issue in the present paper by a careful control of the sets $C_{n,b}$, see equation (\ref{Cnb}), whose tail distribution produces very often the (optimal) rate of decay for the lower bound of the correlation function (see our Theorem 3). It turns out that there are a few class of attractors for which the original map exhibit  polynomial contraction/expansion along the invariant manifold in the induced region. We will present in a future article an application of our coupling technique to those systems: the advantage with respect the examples presented in this paper, is that the condition  (\textbf{H3}), involving the set $C_{n,b}$ is now automatically satisfied, and this allows us to get easily lower bound for the correlation decay and even to improve previous results on upper bounds.
Another extension of our theory will be to consider higher dimensional systems, and this needs a  corresponding generalization of the standard pair technique.
We finally quote the problem of lifting a few  probabilistic limit theorems from the induced system to the original one;  it is well known that problems arise for the induced observable which are not anymore essentially bounded and this requires the use of suitable functional spaces. It is worth mentioning that the alpha-mixing property we obtained in this paper can lead to   a vast development of limit theory --- for example, CLTs, weak invariance principles, laws of the iterated logarithm, almost sure invariance principles, and rates of convergence in the strong law of large numbers, that have been studied in literatures for alpha-mixing sequences under certain conditions.  
  We describe the structure of our work in more details below.
\subsection{Assumptions}

Let $\cM$ be a 2-dimensional compact Riemannian manifold, possibly with
boundary.
We assume that the map $\cF$ is (nonuniformly) hyperbolic with singularity set $\cS_1$, as defined by Katok and Strelcyn \cite{KS}. This means that $\cF$ preserves
a probability measure $\mu$ such that $\mu$-a.e.\ point $x\in \cM$ has
two non-zero Lyapunov exponents: one positive and one negative. Also,
the first and second derivatives of the maps $\cF$ and $\cF^{-1}$ do
not grow too rapidly near their singularity sets $\cS_1$ and $\cS_{-1}$,
respectively; and the $\varepsilon$-neighborhood of the singularity set has measure $\cO(\varepsilon^{q_0})$ for some $q_0>0$. This is to ensure the existence and absolute continuity of stable
and unstable manifolds at $\mu$-a.e.\ point.

In order to investigate statistical properties of the nonuniformly hyperbolic system $(\cF, \mu)$, we use the inducing scheme that were used by Markarian \cite{M04} for Bunimovich Stadia, Chernov and Zhang \cite{CZ05a} for general hyperbolic systems with slow decay of correlations.
  We first construct a subset $\hat M\subset \cM$, with positive  measure $\mu(\hat M)>0$; and such that the restricted measure  of $\mu$ on $\hat M$ is mixing under $\cF|_{\hat M}$, and $\cF|_{\hat M}$ is uniformly hyperbolic. Let $R:\cM\to\mathbb{N}$ be the first hitting time of $\hat M$. Then $\cM$ can be split into two invariant sets $\cM=(R<\infty)\cup (R=\infty)$.  It is possible that the set $(R=\infty)$ has positive  Lebesgue measure, in which case we will only concentrate on the set $(R<\infty)$.    We now define a set $\hat \cM\subset \cM$ to be the closure of $(R<\infty)$. More precisely,
\beq\label{defnhM}
\hat \cM=\bigcup_{n=1}^{\infty} \bigcup_{k=0}^{n-1} \cF^k (R=n),\,\,\,\,\,\,\,\,\mu-a.s.
\eeq

By Poincar\'e Recurrent Theory, and using the fact that $\mu(\hat M)>0$, we know that $\mu(\hat\cM)=1$ is a set of full $\mu$-measure, and $\mu(R=\infty)=0$. We define  $$M:=\hat\cM\cap\hat M,\,\,\,\,\,\,\text{ and }\,\,\,\,\,\,\,\,\,\,\,Fx=\cF^{R(x)},\,\,\,\,\,\, \,\forall x\in M,$$ with $\mu_M:=\mu|_{M}/\mu(M)$. Now $(F, M, \mu_M)$ is the induced system for the original $(\cF, \cM, \mu)$. We assume the induced system is time reversible. 

Let $\Omega \subset M $ be an open subset and let $F\colon
\Omega \to M$ be a $C^{2}$ diffeomorphism of $\Omega$ onto
$F(\Omega)$.  We assume that $S_1 =
M\setminus\Omega$ is a finite or countable union of smooth, connected compact
curves.  If $M$ has boundary
$\partial M$, it must be a subset of  $S_1$. We
call $S_1$  the \emph{singularity sets} for the maps $F$. Denote $M_n$ as the closure of the level set  $(R=n)\cap M$. Clearly, the boundary $\partial M_n \subset S_1$.  We assume that 
the restriction of the map $F$ to any
connected component $\Omega_i$ of $\Omega$ can be extended by continuity to its boundary
$\partial \bar{\Omega}_i$, though the extensions to $\partial \bar{\Omega}_i \cap
\partial \bar{\Omega}_j$ for $i\neq j$ need not agree. Similarly, one can define $S_{-1} = M\setminus F(\Omega)$ to be the singular set for $F^{-1}$, which is  a finite or
countable union of smooth compact curves. Moreover, for each $i$ the
restriction of $F^{-1}$ to any connected component $F(\Omega_i)$
can be extended by continuity to its boundary $\partial F(\bar{\Omega}_i)$.

Let $$\cW^{u}_F=\cap_{n\geq 0} F^n (M\setminus S_1),\,\,\,\,\,\,\cW^{s}_F=\cap_{n\geq 0} F^{-n} (M\setminus S_{-1}).$$ One can check that
$\cW^{u/s}_F$ is  $\mu$-almost surely the union of all unstable (stable) manifolds (see also \cite{CM}~Section 4.11. Moreover,
 the partition  $\cW^{u/s}_F$ of $M$ into unstable (stable) manifolds is measurable.
Thus $\mu_M$ induces conditional distributions on $\mu_M$-almost all
unstable manifolds (see the definition and basic properties of
conditional measures in \cite[Appendix~A]{CM}). Most importantly, we
assume that $\mu_M$ is an
Sinai-Ruelle-Bowen (SRB) measure; i.e. the conditional distribution of $\mu_M$ on any unstable
manifold $W\subset \cW^{u}_F$ is absolutely continuous with respect to
the Lebesgue measure on $W$.  SRB measures are known to be the only physically observable measures, in the sense that
their basins of attraction have positive Lebesgue volume; see
\cite{Y02} and \cite[Sect.~5.9]{H02}. By our choice of $\hat M$, or equivalently the set $M$, the induced system $(F, M, \mu_M)$ is ergodic and mixing. Otherwise, we can reduce the set $M$ to a smaller invariant set. 
It is easy to show \cite{CK} that the SRB
measure $\mu_M$ cannot be concentrated on curves, i.e., $\mu_M(W)=0$ for
any smooth curve $W\subset M$. Thus all our singularity sets
$S_{\pm 1}$ and their images under $F^n$, $n\in\IZ$, are
null sets.

Indeed one obtains an ergodic component $\hat \mu$ of  the measure $\mu$ on $\cM$ using $\mu_M$,  such that for any Borel set $A\subset \cM$,
\beq\label{defnmu}\hat \mu(A)=\frac{1}{\mu_M(R)}\cdot \sum_{n=1}^{\infty}\sum_{k=0}^{n-1} \cF^k_*\mu_M(A\cap \cF^k M_n).\eeq

It follows from our assumption on $\mu_M$, that the measure $\hat\mu$ is  also a  mixing SRB measure. Let $\cN_M=\{n\geq 1\,:\, \mu_M(R=n)>0\}$ be the index set for the nontrivial level sets $(R=n)$.  Then the mixing property of $\hat\mu$ is equivalent to $\gcd(\cN_M)=1$. For simplicity of notions, we assume $\hat\mu=\mu$, and the original system $(\cF, \cM, \mu)$ is  a mixing system.

Note that every unstable manifold $W_F^u$ for $F$ is a (part of) an unstable manifold $W^u$ for $\cF$, more precisely, $W^u_F = W^u\cap \Omega_i$ for some $\Omega_i \subset M$.
Now  $\cW_F^u$ can be extended to the whole space $\cM$ by pushing it forward under the original map $\cF$:
\beq\label{def:cWu}
  \cW^{u}_{\ast}=\cup_{n=1}^{\infty}\cup_{k=0}^{n-1}\cF^{-k} \{W^u\in \cW_F^{u}\colon W^u\subset F M_n\}.
\eeq

On the other hand,  the collection $\cW^s$ of stable manifolds for $\cF$ can be constructed similarly:
\beq\label{def:cWs}
  \cW^{s}_*=\cup_{n=1}^{\infty}\cup_{k=0}^{n-1}\cF^{k} \{W^s\in \cW_F^{s}\colon W^s\subset M_n\}.
\eeq

Notice that  $\cW^{u/s}_*$ will not be exactly the collection $\cW^{u/s}$ of unstable/stable manifolds for $\cF$, instead the latter would be obtained by concatenation of some curves from $\cW^{u/s}_{\ast}$. Since the collection of curves in $\cW^{u/s}\setminus \cW^{u/s}_*$ is a $\mu$-null set, we will not make distinction between these two sets below. More precisely, we will identify  $\cW^{s/u}=\cW^{s/u}_*$ almost surely. \bigskip

 In
chaotic billiards, all the above assumptions are satisfied and are
usually easy to check. In particular, the invariant measure $\mu$ for billiards  is absolutely continuous  and has a smooth
 density on all of $\Omega$. In physics terms, this invariant measure $\mu$ is an
equilibrium state. Another important class of systems consists of small
perturbations of chaotic billiards (usually induced by external forces
or special boundary conditions) \cite{Ch01,Ch08}. Those systems model
electrical current \cite{CDlorentz,CELS}, heat conduction and viscous
flows \cite{BSp,CL95}, the motion under gravitation on the Galton board
\cite{CDgalton}, etc. For perturbed billiards all the above assumptions
are satisfied, too, but the measure $\mu$ is no longer absolutely
continuous: it is singular with respect to the Lebesgue measure on
$\cM$ (though every open subset $U\subset \cM$ still has a positive
$\mu$-measure). In physics, such a measure $\mu$ (for billiards under small perturbations) is called a nonequilibrium
steady state (NESS). Our results also apply to rather general hyperbolic systems with singularities, including  certain perturbed billiards that satisfying our assumptions.  We next make more specific assumptions on the induced system $(M,F,\mu_M)$.\\

\noindent (\bH1) {\textbf{Sufficient conditions for exponential decay of correlations of the mixing system $(F,M,\mu_M)$.}}

\begin{itemize}
\item[(\textbf{h1})] \textbf{Hyperbolicity\footnote{We have already assumed that Lyapunov exponents are not zero $\mu$-a.e., but our methods also use stable and unstable cones for the map $F$.} of $F$}. There exist two families of cones
$C^u_x$ (unstable) and $C^s_x$ (stable) in the tangent spaces
${\cal T}_x M$, for all $x\in M\setminus S_1$, and there exists a
constant $\Lambda>1$, with the following properties:
\begin{itemize}
\item[(1)] $D_x F (C^u_x)\subset C^u_{ F x}$ and $D_x F
    (C^s_x)\supset C^s_{ F x}$, wherever $D_x F $ exists.
    \item[(2)] $\|D_x F(v)\|\geq \Lambda \|v\|, \forall
    v\in C_x^u; \quad\text{and}\quad   \|D_xF^{-1}(v)\|\geq
    \Lambda \|v\|, \forall v\in C_x^s$. \item[(3)] These
    families of cones are   continuous on $M$
 and the angle between $C^u_x$ and $C^s_x$ is uniformly
 bounded away from
zero.
 \end{itemize}

We say that a smooth curve $W\subset M$ is an unstable (stable) \emph{curve} if at every point $x \in W$ the tangent line
$\cT_x W$ belongs in the unstable (stable) cone $C^u_x$ ($C^s_x$), and  points on $F^{-1}(W)$ (resp. $FW$) have the same property.
Furthermore, a curve $W\subset  M$ is an unstable (stable) \emph{manifold} if $F^{-n}(W)$ is an unstable (stable) curve for all $n \geq 0$ (resp. $\leq 0$).

\item[(\textbf{h2})] \textbf{Singularities.} The boundary $\partial M$ is transversal to both stable and
    unstable cones. Every other smooth curve $W\subset S_1\setminus \partial M$ (resp.
    $W\subset S_{-1}\setminus \partial M$ ) is a
 stable (resp.  unstable) curve. Every curve in $ S_1$
 terminates either inside another curve of $ S_1$ or on
 the boundary $\partial M$. A similar assumption is made for $S_{-1}$. Moreover, there exist $s_1, s_0\in (0,1]$ and $C>0$ such
    that for any $x\in M\setminus  S_1$ \beq\label{upper}
   \|D_xF \|\leq C\, \dist(x,  S_1)^{-s_0},\eeq
   and for any $\eps>0$,
   \beq\label{epscs11}\mu_M\bigl(x\in M\colon \dist(x, S_1)<\eps\bigr)<C\eps^{s_1}.\eeq

Note that (\ref{epscs11}) implies that for $\mu_M$-a.e. $x\in M$, there exists a stable manifold $W^s(x)$ and an unstable manifold $W^{u}(x)$, such that $F^n W^s(x)$ and $F^{-n}W^u(x)$ never hit $ S_1$, for any $n\geq 0$.

\begin{defn}
For every $x, y \in M$, define $\bs_+(x,y)$, the forward
\emph{separation time} of $x, y$, to be the smallest integer
$n\geq 0$ such that $x$ and $y$ belong to distinct elements of
$F^n (M\setminus S_1)$.
 Fix $\beta\in(0,1)$, then $d(x,y)=\beta^{\bs_+(x,y)}$
 defines a metric on $ M$.
Similarly we define the backward separation time $\bs_-(x,y)$.
\end{defn}

\item[(\textbf{h3})] \textbf{Regularity of stable/unstable
curves}. We assume that the following families of stable/unstable curves, denoted by $\cW^{s/u}_F$ are invariant under $F^{-1}$ (resp. $F$) and include all stable/unstable manifolds:

\begin{enumerate}
\item[(1)]  \textbf{Bounded curvature.} There exist $B>0$ and $c_M>0$, such that the  curvature
    of any $W\in \cW^{s,u}_F$ is uniformly bounded from above by $B$, and the length of the curve $|W|<c_M$.
 \item[(2)] \textbf{Distortion bounds.} There exist $\gamma_0\in (0,1)$ and   $C_{\br}>1$ such
    that for any unstable curve $W\in \cW^{u}_F$ and
    any $x, y\in W$,
 \beq
      \left|\ln\cJ_W (F^{-1}x)-\ln \cJ_W (F^{-1}y)\right|       \leq C_{\br}\, d (x, y)^{\gamma_0}, \label{distor10}
 \eeq
where  $$\cJ_W (F^{-1}x)=dm_{F^{-1}W}(F^{-1}x)/dm_W(x),$$  denotes  the Jacobian
of $F^{-1}$ at $x\in W$ with respect to the Lebesgue measure $m_W$ on the unstable curve $W$.

\item[(3)] { \textbf{Absolute continuity.}}
 Let $W_1,W_2\in \cW^{u}_F$ be two unstable curves close to each other. Denote
 \begin{equation*}
 W_i'=\{x\in W_i\colon
W^s(x)\cap W_{3-i}\neq\emptyset\}, \hspace{.5cm} i=1,2.
 \end{equation*} The map
$\bh\colon W_1'\rightarrow W_2'$ defined by sliding along stable
manifolds
 is called the \textit{holonomy} map. We assume $\bh_*m_{W_1'}$ is absolutely continuous with respect to $m_{W_2'}$, i.e. $ \bh_*m_{W_1'}\prec m_{W_2'}$; and furthermore, there exist uniform constants $C_{\br}>0$ and $\vartheta_0\in (0,1)$, such that  the Jacobian of $\bh $ satisfies
 \beq\label{Jh}
 |\ln\cJ\bh(y)-\ln \cJ\bh(x)| \leq C_{\br} \dist(W_1', W_2')^{\gamma_0}\,
 \vartheta_0^{\bs_+(x,y)},
\hspace{1cm}\forall x, y\in W_1' \eeq
where $\dist(W_1', W_2'):=\sup_{x\in W_1', y\in W_2'} d(x,y)$.
Similarly, for any $n\geq 1$ we can define the holonomy map
$$
   \bh_n=F^n\circ \bh \circ F^{-n}:F^n W_1\to F^n W_2,
$$
and then (\ref{Jh}) and the uniform hyperbolicity (\textbf{h1}) imply
\beq \label{cJhn}
    \ln \cJ\bh_n(F^n \bx)\leq C_{\br} \dist(W_1',W_2')^{\gamma_0} \, \vartheta_0^{n}.
\eeq
\end{enumerate}

\item[(\textbf{h4})] {\textbf{ One-step expansion.}}
We have \beq
  \liminf_{\delta\to 0}\
   \sup_{W\colon |W|<\delta}\sum_{n: V_n\in FW\setminus S_{1}}
   \left(\frac{|W|}{|V_{n
   }|}\right)^{q_0}\cdot \frac{|F^{-1}V_{n}|}{|W|}<1,
      \label{step1} \eeq where the supreme is taken over regular unstable curves $W\subset M$, $|W|$ denotes the length of $W$.  \\\end{itemize}

\noindent{\textbf{Remark.}} Indeed \textbf{(h1)-(h4)} are
sufficient conditions for  exponential decay rates of correlations, as well as for the coupling lemma, for the induced map. These assumptions are quite standard and have been made in many references  \cite{C99,CD,CM,CZ09}.
Note that we may assume that the lengths of
unstable/stable manifolds $W\in\cW^{u/s}$ are uniformly bounded by a small constant
$\bc_M\in(0,1)$. That is, $|W|<\bc_M$ for any unstable/stable manifold $W$. One can
always guarantee this by adding a finite number of grid lines to $\cS_{\pm 1}$ satisfying
\textbf{(h2)}.

\vskip.1cm

 Next we make assumptions on the original mixing system $(\cF,\cM,\mu)$.\\
 \\

\noindent (\bH2) \textbf{Distribution of the first hitting time $R$.} We assume that there exist $C>0$, ${\alpha_0}\geq 2$,  such that  the distribution of $R$  satisfies for any $n\geq 1$: \beq\label{R(x)=n}\mu(M_n) =n^{-1-\alpha_0}L(n)\eeq where  $M_n$ is the closure of the level set of $R$ restricted on $M$ which is $\{x\in M\,:\, R(x)=n\}$, and $L(n)$ is a slowly vary function at infinity, Moreover, we assume that there exists $N_0\geq 1$ such that every level set $M_n$ contains at most $N_0$ connected components.  We also assume that for any unstable manifold $W\subset M_n$, its length
\beq\label{length WinMn}
|W|\leq C n^{-a}
\eeq For some $a\geq \alpha_0$.\\


\noindent{\textbf{Remark.}}
Note that (\ref{R(x)=n}) implies that  $$  \lim_{n\to\infty} \frac{\ln \mu_M(R>n)}{\ln n} =-\alpha_0,\,\,\,\,\,\,\,\,\lim_{n\to\infty} \frac{\ln \mu(R>n)}{\ln n} =1-\alpha_0.$$
Indeed for our application to billiards, one could show that the function $L(n)$ is uniformly bounded.
\medskip
For a large $b$ (whose precise value will be chosen according to (\ref{condonb1}) and (\ref{dpbchi})), we denote
$\psi(n):=(b\ln n)^2$, and
define the set
\beq\label{Cnb}
    C_{n,b}=\{x\in\cM \,|\, \#_{1 \leq i\leq n} \{\cF^i(x) \in M\}\in [0,\psi(n)]\}.
\eeq
Clearly, $C_{n,b}$ contains those points in $(R\geq n)$, as well as those in $(R< n)$, whose forward iterations return to $M$ at most $\psi(n)$ times.  \\

\noindent (\bH3)  {{\textbf{Measure of $C_{n,b}$.}} We assume that $$\mu(C_{n,b}) \leq C \mu(R>n);
  \,\,\,\,\,\mu_M(C_{n,b}\cap M) \leq
C\mu_M(R>n)\leq C_1\mu(R>n)/n.$$
\noindent Remark: Note that by time reversibility, we also have $ \mu(C_{n,b}\cap \cF^{-n}M) \leq
  C\mu_M(R>n)$. Moreover, by (\textbf{H2}), one chan check that $$\mu(C_{n,b}) \leq \frac{C}{\alpha_0(\alpha_0-1)}n^{1-\alpha_0}L(n),\,\,\,\,\,  \mu(C_{n,b}\cap M) \leq
\frac{C}{\alpha_0}n^{-\alpha_0}L(n).$$\\

}


\subsection{Main results}
For any $\gamma\in (0,1)$, we consider those	   real-valued  functions $f\in L_{\infty}(\cM,\mu)$ such that, for any $x$ and $y$ lying on one stable manifold $W\in \cW^s$,
\beq \label{sDHC-} 	|f(x) - f(y)| \leq \|f\|^-_{\gamma} d (x,y)^{\gamma},\eeq
with
$$\|f\|^-_{\gamma}\colon = \sup_{W\in \cW^s,}\sup_{ x, y\in W}\frac{|f(x)-f(y)|}{ d (x,y)^{\gamma}}<\infty.$$
Let $\cH^-(\gamma)$ be the collection of all such observables $f$, such that $\|f\|^-_{\gamma}<\infty$.

Next we define $\cH^{+}(\gamma)$ as the set of all bounded, real-valued  functions $g\in L_{\infty}(\cM, m)$  such that    for any $g\in \cH^+(\gamma)$, and any $x, y$ belong to one unstable manifolds $W\in \cW^u$, and we have
\beq \label{DHC+} 	|g(x) - g(y)| \leq \|g\|^+_{\gamma}  d (x,y)^{\gamma}, \eeq
with
$$\|g\|^+_{\gamma}\colon = \sup_{W\in \cW^u,}\sup_{ x, y\in W}\frac{|g(x)-g(y)|}{ d (x,y)^{\gamma}}<\infty.$$

According to the hyperbolically of $F$ and $\cF$, one can check that $\cH^+(\gamma)$ (resp. $\cH^-(\gamma)$ ) is invariant under $F$ and $\cF$ (resp.  $F^{-1}$ and $\cF^{-1}$).

For every $f\in \cH^{\pm}(\gamma)$ we define
\beq \label{defCgamma}
   \|f\|^{\pm}_{C^{\gamma}}\colon=\|f\|_{\infty}+\|f\|^{\pm}_{\gamma}.
\eeq
In particular, if $f$ is H\"{o}lder continuous on every component of $\cM\setminus \cS_k$, for some integer $k\in \mathbb{Z}$, with H\"{o}lder exponent $\gamma$, then one can check that $f\in \cH^{+}(\gamma)\cap \cH^{-}(\gamma)$. Indeed our results can be easily extended to larger class of functions  $f\in \cH^{-}(\gamma)$ (resp. $f\in \cH^{+}(\gamma)$)  to be H\"{o}lder continuous only on  stable (resp. unstable) curves.

By using the coupling methods, we obtain the following upper bounds for the rate of decay of correlations.
\begin{theorem}\label{main2}
For system$(\cF,\cM,\mu)$ satisfying (\bH1)-(\bH3), for any
observables $f\in \cH^-(\gamma_f)$ and $g\in\cH^+(\gamma_g)$  on $\cM$, with $\gamma_f,\gamma_g\geq \gamma_0$,  for any $n\geq 1$,
\begin{align*}
|\mu(f\circ \cF^n\cdot g)-\mu(f)\mu(g)|
  &\leq C_2\|g\|^+_{_{C^{\gamma_g}}} \|f\|^-_{_{C^{\gamma_f}}} n^{1-{\alpha_0}},\end{align*}
 with $\alpha_0>1$ was given in (\textbf{H2}), and $C_i=C_i(\gamma_0)>0$ is a constant, for $i=1,2$. \end{theorem}

 In the proof of Theorem \ref{main2}, we obtain a rather explicit  upper bound for the decay rates, see Lemma \ref{Lemma91}, which is determined by both the measure of $(R>n)$ and that of $C_{n,b}\cap\supp(g)\cap \supp(f\circ F^{n})$.   If one chooses $f$ and $g$, such that
 $$\mu(C_{n,b}\cap\supp(g)\cap \supp(f\circ F^{n}))\leq C\mu(R>n).$$ Then we get a rate
 $$|\mu(f\circ \cF^n\cdot g)-\mu(f)\mu(g)|\leq  C_1C\|g\|^+_{_{C^{\gamma_g}}} \|f\|^-_{_{C^{\gamma_f}}}\mu(R>n).$$

For any $x\in \cM$, let $X_0=R(x)$. Note that $\cF^{R(x)} x\in M$. Thus for any $n\geq 1$,  let $X_n(x)=R\circ F^n(\cF^{R(x)} x)$. Let a probability density function $g=d\nu/d\mu\in \cH^+(\gamma_0)$ supported on $\cW$, which is  a measurable collection of  unstable curves. Assume Let  $$A_n(\cW)=\{x\in \cW\,:\, X_0(x)+\cdots+X_{k-1}(x)=n, \text{ for some } k=1,\cdots,n\},$$ as all points in $\cW$ that will return to $M$  after $n$-iterations.
\begin{proposition} $$|\nu(A_n(\cW))-\mu(M)|\leq Cn^{1-\alpha_0}.$$
In particular, for any $n\geq 1$, $$\mu(A_n(\cM))=\mu(\bI_M\circ \cF^n)=\mu(M).$$\end{proposition}
\begin{proof} Thus we have
  \beq\label{AntaunWbb1}A_n(\cW)=\bigcup_{k=1}^n \left(\cF^{-k}A_{n-k}(M)\cap \cW_k\right). \eeq
 This implies that
 \begin{align}\label{nuAnpnR1}
\mu(A_n(\cW))&=\sum_{k=1}^{n}\mu(\cF^{-k} A_{n-k}(M)\cap  \cW_k)\nonumber\\
&=\sum_{k=1}^{n}\mu(A_{n-k}(M)\cap \cF^k \cW_k)\nonumber\\
&=\sum_{k=1}^{n}\mu(A_{n-k} (M)|\cF^k\cW_k)\mu (\cW_k).\end{align}

 On the other hand, it can be written as the union of $n$ disjoint $s$-subsets
 \beq\label{AntaunW1}A_n(\cW)=(R=n)\cup \left(\cF^{-1}(\cW_{n-1})\cap A_1(M)\right)\cup\cdots \cup \left(\cF^{-(n-1)}(\cW_{1})\cap A_{n-1}(M)\right).\eeq

Using observables, we also have
$$
\nu(A_n(\cW))=\mu(\bI_M\circ \cF^n\cdot g).
$$
Let $f=\bI_M$,
by Theorem \ref{main2}, using the estimation on   decay rates of correlations, we have
$$|\nu(A_n(\cW))-\mu(M)|\leq Cn^{1-\alpha_0}.$$
In particular, for any $n\geq 1$, $$\mu(A_n(\cM))=\mu(\bI_M\circ \cF^n)=\mu(M).$$
\\
\end{proof}

 Note that by assumption, $\mu(C_{n,b}\cap M)$ has smaller measure comparing to that of $ \mu(R>n)$. This also gives us a hint that by choosing observables $f,g$ that are only supported on $M$, one may make the measure of the second set to be of much smaller order than that of $(R>n)$.  Next we indeed prove that $\mu(R>n)$ characterizes the optimal bound for the decay rates of correlations for observables on $M$.

\begin{theorem} \label{TmMain}
Under conditions of Theorem \ref{main2}, if we further assume both $f$ and $g$ are supported in  $ M$, then
correlations   decay as:
\beq\label{mainh2a}
\mu(f\circ \cF^n\cdot g) -
   \mu(f)\mu(g)  =  \mu(R>n)\mu(f)\mu(g)+ E(f,g,n),\eeq
   for any $n\geq N$, with $$|E(f,g,n)|\leq  C \|f\|^-_{C^{\gamma_f}}\|g\|^+_{C^{\gamma_g}}\left(\cO(\mu(R>n)/n)+\cO(\mu(R>n)^2)\right),$$ and $C=C(\gamma_f,\gamma_g)>0$ is a constant.
\end{theorem}
For the case where $\mu(f)\mu(g)=0$ we also have a better bound $o(\mu(R>n))$ specified in the above Theorem \ref{TmMain}. For example, for semi-dispersing billiards on a rectangle, we have ${\alpha_0}=2$, so for general observables, the correlations decay as $ O(n^{-1})$; see also \cite{CZ05a}. But for  observables supported on $M$ and such that $\mu(f)\mu(g)=0$, the above theorem implies $ C_n(f,g,\cF)= O(n^{-2})$.

Next we will prove that for  dynamical systems with slow decay rates of correlations, the class of H\"{o}lder observables $f$ with support on $M$ will satisfy the classical Central Limit Theorem.

\begin{theorem}\label{MCLXn0}   Let $f\in \cH^-(\gamma)\cap\cH^+(\gamma)$ with $\gamma>0$, $\supp(f)\subset M$ and $\mu(f)=0$. Assume $f$ is not a coboundary, i.e. there is no function $h$ such that $f=h-h\circ \cF$. Then the following sequence converges:
\beq\label{ClT}\frac{f+\cdots+f\circ \cF^n}{\sigma\sqrt{ n }}\xrightarrow{d} Z,\eeq
in distribution, as $n\to\infty$.
Here $$ \sigma^2 = \mu(f^2)+2\sum_{n=0}^{\infty}\mu(f\circ \cF^n\cdot f)<\infty,$$ and $Z$ is a  standard normal variable.
\end{theorem}

According to Theorem \ref{MCLXn0}, for billiards with cusps, even though the correlation for general variables usually decay at order $\cO(n^{-1})$, if we pick a H\"{o}lder continuous function  $f$ supported on $M$, with $f\in H^+(\gamma_0)\cap H^-(\gamma_0)$ and $\mu(f)=0$, then we still be able to get a classical Central limit theorem, instead of the abnormal Central limit theorem.

 Throughout the paper we will use the following conventions: Positive and finite global constants whose value is unimportant, will be denoted by $c, c_1,c_2,\cdots$ or $C, C_1, C_2, \cdots$. These letters may denote different values in different equations throughout the paper. In Appendix,  we also list some notations that we use throughout the paper.

\section{Standard families and the Growth Lemma.}

 In this paper, we will use the coupling method to prove our main theorems, which depends heavily on the concept of standard pairs proposed firstly by Dolgopyat in \cite{D01}, as well as the $\cZ$ function by Chernov and Dolgopyat in \cite{CD,CM}.

\subsection{Standard pairs and standard families}

 For any unstable manifold $W\in \cW^{u}$, let $\mu_{W}$ be the probability measure on $W$ determined by  the unique probability density $\rho_{W}$ (with respect to the Lebesgue measure $m_W$) satisfying
 \beq\label{dens}\frac{\rho_{W}(y)}{ \rho_{W}(x)} =
\lim_{n\rightarrow \infty}\frac{\cJ_{W}(\cF^{-n}y)} {\cJ_{W}(\cF^{-n}x)}.\eeq
 Here $\rho_W=d\mu_W/d m_W$ is called {\it{the u-SRB
 density on $W$}}, and the corresponding probability measure $\mu_{W}$
 on ${W}$ is
called {\it{the u-SRB measure of $F$}}.
\begin{lemma} The conditional measure $\mu_W$ is invariant under both $F$ and $\cF$, i.e. $F_*\mu_W=\mu_{_{FW}}$ and $\cF_*\mu_W=\mu_{_{\cF W}}$.\end{lemma}
\begin{proof}
The formula (\ref{dens}) is  standard  in ergodic
theory, see \cite{CM} page 105.
(\ref{dens}) implies that for any $m\geq 1$, and $x,y\in W$,
$$\frac{\rho_W(y) }{\rho_W(x)}=\frac{\rho_{F^{-m}W}(F^{-m}y)\cJ_W(F^{-m} y)}{\rho_{F^{-m}W}(F^{-m}x)\cJ_W(F^{-m}x)},$$
which is equivalent to
$$\frac{\rho_W(y) }{\rho_{F^{-m}W}(F^{-m}y)\cJ_W(\cF^{-m} y)}=\frac{\rho_W(x)}{\rho_{F^{-m}W}(F^{-m}x)\cJ_W(F^{-m}x)}=c^{-1}(m,W,F),$$
for some constant $c=c(m,W,F)>0$. This implies that for any $x\in W$,
\beq\label{cmW}\rho_{F^{-m}W}(F^{-m}x)\cJ_W(F^{-m} x)=c(m,W,F)\rho_W(x). \eeq
To determine $c(m,W,F)$, we use the fact that both $\rho_W$ and $\rho_{F^{-m}W}$ are probability densities.
\begin{align*}
&c(m,W,F)=c(m,W,F)\int_W\rho_W(y) dm_W(y)\\
&=\int_{W}\rho_{F^{-m}W}(F^{-m}y) \cJ_W(F^{-m} y)\, dm_{W}( y)\\
&=\int_{F^{-m}W}\rho_{F^{-m}W}(x)\cJ_W(x)\cdot \frac{d m_{W}(F^m x)}{dm_{F^{-m}W}(x)}\, dm_{F^{-m}W}(x)\\
&=\int_{F^{-m}W}\rho_{F^{-m}W}(x)\, dm_{F^{-m}W}(x)=1,\end{align*}
which implies that $c(m,W,F)=1$.

Similarly, we have that for any $x\in W$, \beq\label{cmWF}\rho_{\cF^{-m}W}(\cF^{-m}x)\cJ_W(\cF^{-m} x)=\rho_W(x),\,\,\,\,\,\,\,\rho_{F^{-m}W}(F^{-m}x)\cJ_W(F^{-m} x)=\rho_W(x). \eeq
This implies that $\mu_W$ is invariant under both $\cF$ and $F$, i.e. $\cF_*\mu_W=\mu_{_{\cF W}}$ and $F_*\mu_W=\mu_{_{F W}}$.
\end{proof}

Furthermore,  it follows from the distortion bound (\ref{distor10}) that $\rho_W\sim |W|^{-1}$\footnote{We say $A_n\sim B_n$ if there exist $0<c_1<c_2$ such that $c_1B_n\leq A_n\leq c_2 B_n$ for all $n\geq 1$.} on $W$ for short $W$. More precisely, we have
\beq\label{rhoWbd}
\frac{1}{|W|} e^{-C_{\br}|W|^{\gamma_0}}\leq \rho_W(x)\leq \frac{1}{|W|} e^{C_{\br}|W|^{\gamma_0}}.
\eeq
Note that the $u$-SRB density function $\rho_W$ is not a bounded function on the phase space $\cM$, as unstable manifolds $W\in \cW^u$ can be arbitrarily short.

Besides $(W,\mu_W)$, we would like to introduce other probability measure on unstable manifold $W$. We call this standard pair, if the density is regular enough:
\begin{defn}[Standard pair]
 A probability
measure $\nu$  supported on an unstable curve $W$ is
called {\it{regular}}, if   $\nu$ is absolutely continuous with respect to
the $u$-SRB measure $\mu_W$ on $W$, and the  probability density function $g=d\nu/d \mu_W$ satisfies \beq\label{standardpair}
 \frac{|g(x) - g(y)|}{g(x)}
\leq C_{F} d(x,y)^{\gamma_0}, \eeq where  $C_F\geq C_{\br}$ is a fixed constant. In this case    $(W, \nu)$ is
called a \emph{standard pair}. Moreover, if the probability density $g=d\nu/d \mu_W$ satisfies
\beq\label{sstandardpair}
 \frac{|g(x) - g(y)|}{g(x)}
\leq \Lambda^{\gamma_0} C_{F} d(x,y)^{\gamma_0}, \eeq
we call $(W,\nu)$ a pseudo-standard pair.
\end{defn}
Here we also need  the concept of pseudo-standard pair, as in the proof of the Coupling Lemma, one needs to subtract a smooth function from the density of a standard pair. But the resulting conditional measure will only induce a pseudo-standard pair, as we will show in lemma \ref{defnN}.

We denote $\cW^u=\{\alpha\in \cA^u\,:\, W_{\alpha}^u\}$, where $\cA^u$ is the index set of the unstable manifolds $W^u_{\alpha}\in \cW^u$. On the index set $\cA^u$, one could define the  $\sigma$ algebra $\cB_{\cA}$ induced from the Borel $\sigma$-algebra on $\cM$. More precisely, we say $\cA\subset \cA^u$ is $\cB_{\cA}$- measurable if and only if the union $\cup_{\alpha\in \cA} W_{\alpha}^u$ is a Borel measurable set in $\cM$.    Moreover,  the SRB measure $\mu$ also induces  a factor measure $\lambda^u$ on the index set $(\cA^u, \cB_{\cA})$.  More precisely, for any Lebesgue measurable set $A=\cup_{\alpha\in \cA} W^u_{\alpha}$, then we define
\beq\label{defnfactor}\lambda^u(\cA):=\mu(A).\eeq
Next we extend this idea to the concept of standard families by introducing various factor measures on the index set.

\begin{defn}\label{defnSF}
Let $\cG=\{(W_{\alpha}, \nu_\alpha), \alpha\in \cA\}$ be a family of   (pseudo-) standard pairs  equipped with a factor measure $\lambda$ on the index set $(\cA^u_M,\cB_{\cA})$.  We call $\cG$ a
{(pseudo-) standard family} on $\cM$, if it satisfies the following conditions:
\begin{enumerate}
\item[(i)] $\cA\in \cB_{\cA}$, and $\lambda(\cA^u\setminus \cA)=0$;
     \item[(ii)] There is a Borel  measure $\nu$ satisfying:
\beq\label{muMdisintegrate}
   \nu(B)=\int_{\alpha\in\cA} \nu_{\alpha}(B\cap
   W_{\alpha})\,
   \lambda(d\alpha),
\eeq
for any measurable set $B\subset \cM$.

 \end{enumerate}
For simplicity, we denote such a  family by
$$\cG = (\cW, \nu)=\{(W_{\alpha},\nu_{\alpha}), \alpha\in \cA, \lambda\}.$$

\end{defn}

In particular, we define $\cG^u:=(\cW^u,\mu)$ to be the standard family defined by the SRB measure $\mu$; and $\cG^u_M:=(\cW^u_M,\mu_M)$ to be the standard family defined by the SRB measure $\mu_M$.  It follows from \cite{CZ09} that under condition (\ref{epscs11}), one can show that there exists $q_0\in (0, s_1)$, such that  $$\mu_M\left(\tfrac{1}{|W^u_M(x)|^{q_0}}\right)<\infty,\,\,\,\,\,\,\,\,\mu\left(\tfrac{1}{|W^u(x)|^{q_0}}\right)<\infty$$
This is equivalent to
\beq\label{muMproper}\int_{\cA^u}|W_{\alpha}|^{-q_0}\,\lambda^u(d\alpha)<\infty,\,\,\,\,\,\,\,\,\,\,\,\,\,
\int_{\cA^u_M}|W_{\alpha}|^{-q_0}\,\lambda_M^u(d\alpha)<\infty.\eeq
To understand the distribution of   short unstable manifolds in any standard family, the above facts provide us motivations for defining a characteristic function  on these standard families.

\begin{defn} Consider any  standard family $\cG=(\cW,\nu)$, with $\cW\subset M$  (or $\cW\subset \cM$, respectively). We define
\beq\label{cZ}
\cZ(\cG)=\frac{1}{\nu(\cM)}\,\int_{\cA}|W_{\alpha}|^{-q_0}\,\lambda(d\alpha), \eeq
where $q_0\in(0,s_1)$, with $s_1$ was defined in (\ref{epscs11}).
IWe denote $\mathfrak{F}(M)$ ( or $\mathfrak{F}(\cM)$) as the collection of all standard families $\cG$ on  $M$ (or $\cM$), such that $\cZ(\cG)<\infty$.  One can check that the collection of all standard family is closed under positive scalar multiplications.
\end{defn}

The bound (\ref{muMproper}) implies that both $\cZ(\cG^u)$ and $\cZ(\cG^u_M)$ are finite.
We fix a large number $$C_{\rm q}>\max\{100 C_F,\cZ(\cG^u_M)\},$$ whose value will be chosen in (\ref{delta0}).  Given a  standard family
$\cG=(\cW,\nu)$, if $\cW\subset M$ and  $\cZ(\cG)<C_{\rm q}$ we say $\cG$ is a \emph{proper} (standard)
family on $M$. Note that by our definition,  $\cG^u_M$ is a  proper family.

\begin{lemma}\label{densitybd} There exists $\delta_0>0$, such that for any standard pair $(W,\nu_W)$, with  $|W|\leq 20\delta_0$, and $g_W:=d\nu_W/d\mu_W$, then the density $g_W$ and the $u$-SRB density $\rho_W$ satisfy:
$$|g_W(x)-1|\leq \eps_{\bd}/10,\,\,\,\,\,\,\text{ and } \frac{1-\eps_{\bd}/10}{|W|}\leq \rho_W(x)\leq \frac{1+\eps_{\bd}/10}{|W|}$$
where $\eps_{\bd}<\min\{10^{-5}, \frac{1}{2}(1-\Lambda^{\gamma_0}(1+C_F^{-1}))$.
\end{lemma}
\begin{proof}

We choose $C_{\rm q}$ large enough,  $\eps_{\bd}<\min\{10^{-5}, \frac{1}{2}(1-\Lambda^{\gamma_0}(1+C_F^{-1}))$and $\delta_0$ small enough to satisfy
\beq\label{delta0}
\frac{1}{C_{\rm q}}<\delta_0^{q_1},\,\,\,\,\,\,\text{ and }\,\,\,\,\, \frac{1}{C_{\rm q}}<10 (20\delta_0)^{\gamma_0}<\frac{\eps_{\bd}}{C_F}.
\eeq
 Combining with (\ref{standardpair}),  if one makes a standard pair
$(W,\nu_W)$ with density function $g_W\in\cH^+(\gamma_0)$ with respect to the u-SRB measure $\mu_W$ on $W$, then (\ref{delta0}) implies that for any $x,y\in W$,
\beq\label{densitypair}
 \frac{|g_W(x) - g_W(y)|}{g_W(x)}\leq C_F(20\delta_0)^{\gamma_0}<\eps_{\bd}/10.
\eeq
 Note that $\mu_W$ is a probability measure, with H\"{o}lder continuous  probability density $\rho_W=d\mu_W/d m_W$. Thus there exists $x_0\in W$, such that $\rho_W(x_0)=|W|^{-1}$, which is the average value of $\rho_W$. By the distortion bound (\ref{rhoWbd}),
$$
\frac{1}{|W|} (1-\eps_{\bd}/10)\leq \frac{1}{|W|} (1-C_{\br}|W|^{\gamma_0})\leq \rho_W(x)\leq \frac{1}{|W|} (1+C_{\br}|W|^{\gamma_0})\leq \frac{1}{|W|} (1+\eps_{\bd}/10)).
$$

Since $g_W(x)=d\nu_W/d \mu_W$ is also a probability density, there exists $x_0\in W$, such that $g_W(x_0)=1$. This implies that for any $x\in W$,
\beq\label{densitypair1}|  g_W(x)- 1|\leq \eps_{\bd}/10.\eeq
 Thus  for any standard pair with length
 $|W|<20\delta_0$,  the density function $g_W$ is bounded by $$1-\eps_{\bd}/10\leq g_W(x)\leq 1+\eps_{\bd}/10.$$   
   \end{proof}

Given a standard family $\cG=(\cW, \nu)=\{(W_{\alpha}, \nu_\alpha),\ \alpha\in \cA, \ \lambda\}$ and $n\ge 1$,
the forward image family $\cF^n\cG=(\cW_n, \nu_n)=\{(V_\beta, \nu^n_\beta), \ \beta\in \cB, \ \lambda^n\}$
is such that each $V_\beta$ is a connected component of $\cF^nW_\alpha\backslash \cS_{1}$
for some $\alpha\in \cA$,
associated with
$$\nu^n_\beta(\cdot)=\cF^n_*\nu_\alpha(\cdot \ |\ V_\beta)\,\,\,\,\,\text{ and }\,\,\,\,\,\,
d\lambda^n(\beta)=\nu_\alpha(\cF^{-n}V_\beta) d\lambda(\alpha).$$
Similarly, one can define $F^n \cG$.

\begin{proposition}\label{SFinvariant}
For any standard family $\cG$, then  $\cF^n\cG$  and $F^n\cG$ are both standard families, for any $n\geq 1$. Or equivalently, $\mathfrak{F}(M)$ ( resp. $\mathfrak{F}(\cM)$) is invariant under $F$ ( resp. $\cF$).  \end{proposition}
\begin{proof} It was proved in   \cite{CZ09} that $\mathfrak{F}(M)$ is invariant under $F$.  Here we only prove that $\mathfrak{F}(\cM)$ is invariant under $\cF$.
 According to  \cite{CZ09}, it is enough to prove that the image of a standard pair is a standard family.

Let $(W, \nu)$ be any standard pair, any $k\geq 1$. Assume $\cF^k W=\{(V_{\alpha},\nu_{\alpha}),\alpha\in \cA, \lambda\}$.
 We denote $W_{\alpha}=\cF^{-k}V_{\alpha}$, for any index $\alpha\in \cA$.
Using (\ref{cmWF}) we know that for any $\alpha\in \cA$, any $x\in V_{\alpha}$,
\beq\label{ccFkW}\frac {d\cF_*^k \mu_{W_{\alpha}}(x)}{d \mu_{V_{\alpha}}(x)}=1.\eeq


 We denote $g=d\nu/d\mu_W$, then for any $k\geq 1$, the density function  $g_k$ of $\cF^k_*\nu$ can be written as
\begin{align}\label{dcFknu}
   g_k(x)&=\frac{d\cF_*^k\nu_{W_{\alpha}}(x)}{d\mu_{V_{\alpha}}(x)}=
   \frac{d\nu_{W_{\alpha}}(\cF^{-k}x)}{d \mu_{W_{\alpha}}(\cF^{-k} x)}\cdot \frac{d \mu_{W_{\alpha}}( \cF^{-k}x)}{d \mu_{V_{\alpha}}( x)}\nonumber\\
   & =\frac{d\nu_{W_{\alpha}}(\cF^{-k}x)}{d \mu_{W_{\alpha}}(\cF^{-k} x)}\cdot \frac{d \cF^k_*\mu_{W_{\alpha}}(x)}{d \mu_{V_{\alpha}}( x)}=\frac{d\nu_{W_{\alpha}}(\cF^{-k}x)}{d \mu_{W_{\alpha}}(\cF^{-k} x)}\cdot \frac{d \mu_{V_{\alpha}}(x)}{d \mu_{V_{\alpha}}( x)}\nonumber\\
   &= \frac{d\nu_{W}(\cF^{-k}x)}{d \mu_{W}(\cF^{-k} x)}\cdot \frac{\nu(W_{\alpha})}{\mu_W(W_{\alpha})}= g(\cF^{-k}(x))\cdot \frac{\nu(W_{\alpha})}{\mu_W(W_{\alpha})},\end{align}
for all $ x\in V_{\alpha}$, here $\nu_{W_{\alpha}}$ is the conditional measure of $\nu$ restricted on $W_{\alpha}$.

 Thus for any $x, y$ belong to one smooth component of
$\cF^kW$, $k\geq 1$,
\begin{align*}|
  \ln g_{k}(x)&-  \ln g_{k}(y)|= |\ln g( \cF^{-k}x)- \ln  g(
\cF^{-k}y)|
\\
&\leq
C_F d( \cF^{-k}x, \cF^{-k}y)^{\gamma_0}\leq C_F d( x, y)^{\gamma_0}.
\end{align*}
 Thus we can see that
 $\cF^k \cG=(\cF^k W,\cF^k_* \nu)$ is a  standard family.
\end{proof}

\subsection{Dynamically H\"{o}lder continous unctions.}
Indeed for  any nonnegative bounded function $g: \cM\to \mathbb{R}$, let  $\cW_g =\{W_{\alpha}\,|\, \alpha\in \cA_g\} $ be the measurable partition of the support of $g$ into unstable curves.  We call $g$ to be a {\it{dynamically H\"{o}lder function}} if (\ref{standardpair}) holds, on any unstable manifold $W\in\cW_g$. Let $DH(\gamma_0)$ be the class of all dynamically H\"{o}lder functions $g$, equipped with  leaf support $\cW_g$.  The lemma below  explores the fact that any dynamically H\"{o}lder function  uniquely generates a standard family.

\begin{lemma}\label{barggWg} Let $g\in DH(\gamma_0)$ be any dynamically H\"{o}lder function, with leaf support $\cW_g=\{W_{\alpha}, \alpha\in \cA_g\}$.\\
 (1) If $\mu(\cW_g)>0$,  then $g$ uniquely determines a standard family $\cG_g=(\cW_g,\nu_g)$, with $g=d\nu_g/d\mu$.\\
 (2) If $\cA_g$ is a finite or countable set, then  $g$ generates a unique standard family $\cG_g=(\cW_g, \nu_g)=((W_{\alpha},\nu_{\alpha}), \lambda_g, \alpha\in \cA_g)$, such that  $g/\mu_{W_{\alpha}}(g)=d\nu_{\alpha}/d\mu_{W_{\alpha}}$ and $\lambda_g(\alpha)=\mu_{W_{\alpha}}(g)$.
\end{lemma}
\begin{proof}
 For any dynamically H\"{o}lder function $g$, the fact that $\mu(\cW_g)>0$ implies that $g$ defines a measure $\nu_g$ on the Borel $\sigma$-algebra of $M$, with  $d\nu_g:=g\,d\mu$. Clearly, $g$ is a Radon-Nikodym  derivative. Let $\cA_g\subset \cA^u$ be the index set for $\cW_g$. Note that for any Borel set $B\subset \cM$,
\begin{align}\label{cGphi}
\nu_g(B)&=\int_{\alpha\in \cA_g} \int_{x\in W_{\alpha}} \bI_B(x)\, g(x) d\mu_{W_{\alpha}}(x) \, \lambda^u(d\alpha)\nonumber\\
&=\int_{\alpha\in \cA_g} \left(\int_{x\in W_{\alpha}} \bI_B(x)\, \frac{g(x)}{\mu_{W_{\alpha}}(g) } d\mu_{W_{\alpha}}(x) \right)\mu_{W_{\alpha}}(g)\,  \lambda^u(d\alpha)\nonumber\\
&=\int_{\alpha\in \cA_g} \int_{W_{\alpha}} \bI_B(x)\, d\nu_{_{W_{\alpha}}}(x) \, \lambda_g(d\alpha),\end{align}
where $$d\nu_{_{W_{\alpha}}}(x)=  \frac{g(x)}{ \mu_{W_{\alpha}}(g) } d\mu_{W_{\alpha}}(x),  \,\,\,\,\,\,\,\,\,\lambda_g(d\alpha)=\mu_{W_{\alpha}}(g)  \cdot \lambda^u(d\alpha).$$

Now we write $\cW_g:=\{W_{\alpha}\,:\, \alpha\in \cA_g\}$, then $\cG_g:=(\cW_g, \nu_g)$ is a standard family according to the definition \ref{defnSF}.

On the other hand, given a standard family $\cG=(\cW,\nu)$, such that  $\mu(\cW)>0$, then there exists a unique Radon-Nikodym $g=d\nu/d\mu$, which satisfies (\ref{cGphi}).  Since $g_{\alpha}=\frac{g(x)}{ \mu_{W_{\alpha}}(g) }=d\nu_{\alpha}/d\mu_{W_{\alpha}}$ defines a standard pair $(W_{\alpha}, \nu_{\alpha})$. Equation (\ref{standardpair}) implies that $g$ is dynamically H\"{o}lder continuous. By the uniqueness of the Radon-Nikodym $g=d\nu/d\mu$, we see that $g$ uniquely determines the standard family $\cG_g=(\cW_g, \nu_g)$.

Next we consider the case when $\cW_g$ consists of finite or countably many unstable curves. Then the leaf support of $g$ is a null set of $\mu$. We first define a factor measure on $\cA_g$, such that for any $\alpha\in \cA_g$,
$$\lambda_g(\alpha):=\mu_{W_{\alpha}}(g)=\int_{x\in W_{\alpha}} g(x)\, d\mu_{W_{\alpha}}(x).
$$ This defines a measure $\nu_g$ on the Borel sigma algebra of $\cM$ $(\cM, \cB)$,
such that for any Borel set $B\subset \cM$,
\begin{align}\label{cGphi}
\nu_g(B)&=\sum_{\alpha\in \cA_g} \left(\int_{x\in W_{\alpha}} \bI_B(x)\, \frac{g(x)}{\mu_{W_{\alpha}}(g) } d\mu_{W_{\alpha}}(x) \right)\,  \mu_{W_{\alpha}}(g)\nonumber\\
&=\sum_{\alpha\in \cA_g} \left(\int_{W_{\alpha}} \bI_B(x)\, d\nu_{_{W_{\alpha}}}(x)\right) \, \lambda_g(\alpha),\end{align}
where $g/\mu_{W_{\alpha}}(g)=d\nu_{\alpha}/d\mu_{\alpha}$.

\end{proof}

Note that although both standard pairs and their forward iterations generate measures that are singular with respect to $\mu$,  by the mixing property of $(\cF, \cM,\mu)$, one still have that
$\cF^n_*\nu \to \mu$ weakly, as $n\to\infty$. Indeed this is one of the main  advantages for using standard pairs/families to study dynamical systems with singularities.

 From now on, we call $\cG_g$ the {\it{standard family associated}} with the dynamically H\"{o}lder function $g\in DH(\gamma_0)$.
Proposition \ref{SFinvariant} implies that  $DH(\gamma_0)$ is invariant under $\cF$. Similarly, we have that $DH(\gamma_0)|_M$ is invariant under $F$.

For any $\alpha\in\cW_g$, let $\bar g(\alpha)=\mathbb{E}(g|\cW_g)(\alpha)$ be the conditional expectation of $g$ on $W_{\alpha}\in \cW_g$ with respect to $\mu$, i.e.
\begin{align*}
\bar g(\alpha)&=\int_{W_{\alpha}} g\,d\mu_{\alpha},\,\,\,\,\forall\ x\in W_{\alpha},\ \forall\ \alpha\in \cA_g.
\end{align*}
By the  dynamically H\"older continuity of $g$, we have
\beq\label{bargg}\sup_{x\in W_{\alpha}} |g(x) -\bar g(\alpha)|\leq \bar g(\alpha) C_{F}|W_{\alpha}|^{\gamma_0}.\eeq
Note that $\mu(\bar g)=\mu(g)$. So we can approximate $g$ by $\bar g$, which is constant on each curve $W_{\alpha}$.  Thus the standard family associated with $\bar g$ is $\cG_{\bar g}:=(\cW_g, \nu_{\bar g})$.
 For any small $\eps>0$, we define $$\cA_{g,\eps}=\{\alpha\in \cA_g\,:\, |W_{\alpha}|<\eps\},\,\,\,\,\,\,\,\text{ and }\,\,\,\,\,\,\,\,\, A_{g,\eps}=\cup_{\alpha\in \cA_{g,\eps}} W_{\alpha}.$$

\begin{lemma}\label{bargWg} Let $g\in DH(\gamma_0)$ be any dynamically H\"{o}lder function with leaf support $\cW_g$, and $\bar g=\mathbb{E}(g|\cW_g)$.\\
 (1) If $\mu(\cW_g)>0$,  then $\bar g$ uniquely determines a standard family $\cG_{\bar g}=((\cW_{\alpha}, \mu_{W_{\alpha}}), \lambda)$, where $\lambda_{\bar g}$ is absolutely continuous with respect to $\lambda^u$, with $\bar g=d\lambda_{\bar g}/d\lambda^u$. Moreover,
  $$\cZ(\cG_{\bar g})=\cZ(\cG_g)\leq \|g\|_{\infty}\frac{(1+C_F) }{\nu_g(\cM)}\cZ(\cG^u).$$  If $\|g\|_{\infty}\leq \frac{\nu_g(\cM)}{1+C_F}$, then  $\cG_g$ is a proper family.\\
 (2) If  $\cA_g$ is a finite or countable set, then  $\bar g$ generates a unique standard family, which is made of standard pairs  $\{(W_{\alpha}, \mu_{\alpha}), \alpha\in \cA_g\}$ with factor measure $\lambda(\alpha)=\bar g(x)$, for $x\in W_{\alpha}$.\\
 (3) Moreover, $\nu_g(A_{g,\eps})\leq \nu_g(\cM) \cZ(\cG_g) \eps^{q_0}$.
\end{lemma}
\begin{proof}
 For any dynamically H\"{o}lder function $g$, the fact that $\mu(g)>0$ implies   for any Borel set $B\subset \cM$,
\begin{align}\label{cGphi}
\nu_{\bar g}(B)&=\int_{\alpha\in \cA_g} \int_{x\in W_{\alpha}} \bI_B(x)\, \bar g(x) d\mu_{W_{\alpha}}(x) \, \lambda^u(d\alpha)\nonumber\\
&=\int_{\alpha\in \cA_g} \left(\int_{x\in W_{\alpha}} \bI_B(x)\, d\mu_{W_{\alpha}}(x) \right)\mu_{W_{\alpha}}(g)\,  \lambda^u(d\alpha)\nonumber\\
&=\int_{\alpha\in \cA_g} \int_{W_{\alpha}} \bI_B(x)\, d\mu_{_{W_{\alpha}}}(x) \, \lambda_{\bar g}(d\alpha),\end{align}
where $$\lambda_{\bar g}(d\alpha)=\bar g  \cdot \lambda^u(d\alpha).$$

Comparing with the  above  Lemma, we know that  any $g\in DH(\gamma_0)$, with $\mu(g)>0$,  generates a standard family $\cG_{g}$, which is made of standard pairs $\{(W_{\alpha}, \nu_{\alpha}), \alpha\in \cA_g\}$, with factor measure that is uniquely determined by $\bar g=\mathbb{E}(g|\cW^u)$. More precisely, $\lambda(d\alpha)=\bar g \, d\lambda^u(\alpha)$ or $\lambda(\alpha)=\bar g$. Since $\cZ$ function only depends on the factor measure, thus $\cZ(\cG_g)=\cZ(\cG_{\bar g})$.

Thus by (\ref{bargg}),
\beq\label{cZbarg}
\cZ(\cG_g)=\cZ(\cG_{\bar g})=\frac{1}{\nu_g(\cM)}\,\int_{\cA}\frac{\bar g}{|W_{\alpha}|^{q_0}}\,\lambda^u(d\alpha)\leq \frac{(1+C_{F})\|g\|_{\infty}}{\nu_g(\cM)} \cZ(\cW^u,\mu)\leq \frac{(1+C_{F})C_q\|g\|_{\infty}}{\nu_g(\cM)} \eeq
Thus if $\|g\|_{\infty}\leq \frac{\nu_g(\cM)}{1+C_F}$, then $g$ generates a proper family $\cG_g$.

It follows from the above definition that
\begin{align*}\frac{\nu_g(A_{g,\eps})}{\nu_g(\cM)}&=\frac{\lambda_g(\cA_{g,\eps})}{\nu_g(\cM)}=\frac{\eps^{q_0}}{\nu_g(\cM)}\, \int_{\alpha\in \cA_{g,\eps}}\eps^{-q_0}\,\lambda_g(d\alpha)\\
&\leq\frac{\eps^{q_0}}{\nu_g(\cM)}\,\int_{\alpha\in\cA_{g,\eps}}|W_{\alpha}|^{-q_0}\,\lambda_g(d\alpha)\leq \eps^{q_0}\cZ(\cG_g).\end{align*}

Next we consider the case when $\cA_g$ is at most countable.  We  define  for any $\alpha\in \cA_g$,  $$\lambda_{\bar g}(\alpha):=\bar g(x), \,\,\,\,\,\,\,\forall x\in W_{\alpha}.$$ Thus for any Borel set $B\subset \cM$,
\begin{align}\label{cGphi}
\nu_{\bar g}(B)&=\sum_{\alpha\in \cA_g} \left(\int_{W_{\alpha}} \bI_B(x)\, d\mu_{_{W_{\alpha}}}(x)\right) \, \lambda_{\bar g}(\alpha).\end{align}

If $\cA_g$ is a countable set, with $\sum_{\alpha\in \cA_g} \lambda_g(\alpha)=1$, then $\nu_{g}(\cM)=1$ and $\sum_{\alpha\in \cA_g} \bar g=1$. This also implies that $\mu_{W_{\alpha}}(g)=\bar g\leq 1$, i.e. $$\|g\|_{\infty}\leq 1+ C_F|W_{\alpha}|^{\gamma_0},$$ and $$\cZ(\cG_g)=\sum_{\alpha\in\cA}\frac{\bar g}{|W_{\alpha}|^{q_0}}.$$

Next we verify item (3).
It follows from the above definition that
\begin{align*}\nu_g(A_{g,\eps})=\lambda_g(\cA_{g,\eps})
=\eps^{q_0}\, \sum_{\alpha\in \cA_{g,\eps}}\eps^{-q_0}\,\lambda(\alpha)\leq\eps^{q_0}\,\sum_{\alpha\in\cA_{g,\eps}}|W_{\alpha}|^{-q_0}\,\lambda_g(\alpha)\leq \eps^{q_0}\cZ(\cG_g).\end{align*}

Thus for any proper family $\cG=(\cW,\nu)$, the distribution of short unstable manifolds in $\cW$ satisfies:
$$\nu(A_{\eps})\leq C_{\rm q}\nu(\cM) \eps^{q_0}.$$

\end{proof}

In this paper, we always take the u-SRB measure $\mu_{\alpha}:=\mu_{W_{\alpha}}$ to be the reference measure on indexed unstable curve $W_{\alpha}$, sometimes we also denote it as $\mu_W$ for general unstable curve $W$. It follows from the fact that $DH(\gamma_0)$ is invariant under $F$ and $\cF$, that  the set  $\mathfrak{F}(M)$ is invariant under $F$, and $\mathfrak{F}(\cM)$ is invariant under $\cF$.

\subsection{Growth Lemma .}

The Growth Lemma was originally designed for the induced map, see \cite{CM, CZ09}.
   For any $n\geq 1$, any $x\in M$, such that $W^u(x)$ exists, we define $r_n(x)$ as the minimal distance between $F^n x$ and  the two end points of $W^u(F^n x)$. In particular, the following facts were proved in \cite{CZ09}, which was also called the Growth Lemma for the induced system $(F,M,\mu_M)$, see \cite{C99}.
   \begin{lemma}\label{properagain}
For the induced system $(F,M,\mu_M)$, the following statements hold:\\
(1) There exists a uniform constant $\chi>0$, such that for any standard pair $(W,\nu)$, $F^n(W,\nu)$ is a proper  family for any $n>\chi|\ln |W||$;\\
 (2) For any $x\in M$, let $\br^{u}(x)$ be the minimal distance of $x$ to the boundary points of $W^{u}(x)$ measured along $W^{u}(x)$. Then  any standard pair $(W,\nu)$ with length $|W|>\delta_0$ is proper; and there exists $C>0$ such that for any stable/unstable manifold $W^{s/u}$ with length $|W^{u}|>\delta_0$, we have
 \beq\label{delta00}m_{W^{s}}(\br^u(x)<\eps)<C \eps^{q_0};
 \eeq
 \noindent(3) For any standard family $\cG=((W_{\alpha},\alpha\in \cA),\nu)$, any $\eps\in (0,1)$,
  \beq\label{epsq}\nu(x\in W_{\alpha}\,:\, \br^u(x)<\eps, \alpha\in \cA)<\cZ(\cG) \eps^{q_0},\,\,\,\,\,\,\,\,\nu(x\in W_{\alpha}\,:\, \br_n(x)<\eps, \alpha\in \cA)<\cZ(F^n\cG) \eps^{q_0}. \eeq
  \noindent(4) There exist constants $c>0$, $C_z>0$, and $\vartheta_1\in (0,1)$, such
that for any  standard family $\cG=(\cW,\nu)$ supported in $M$, if $\cZ(\cG)<\infty$, then for any $n\geq 1$,
\beq\label{firstgrowth}  \cZ
(F^n \cG)\leq c\vartheta_1^{n}\cZ (\cG)+C_z;
\eeq
\beq\label{firstgrowth1}
F^n_*\nu(r^{u}<\eps)\leq c\,\vartheta_1^{n}
\nu(r^{u}<\eps)+C_z\eps^{q_0};
\eeq
\noindent(5) Let $\cG=(\cW,\nu)$ be a standard family, with $\cZ(\cG)<\infty$, then there exists $N_{\cG}>1$, such that $F^{N_{\cG}}\cG$ is a proper family and $\vartheta_1^{-N_{\cG}-1}\leq C_1\cZ(\cG)$, for some uniform constant $C_1=\frac{c C_z}{C_q}$.
 \end{lemma}

In  the present literatures,  the concept of proper standard families are only considered for the induced system $(F, M,\mu_M)$, as the system enjoys exponential decay rates of correlations. In particular the Growth Lemma \ref{properagain} was designed only for the induced systems (or systems with fast mixing rates). However, for systems with slower decay rates of correlations, it is crucial to analyze the evolution of proper standard families. Here we first state the Growth Lemma and  then extend the concept of proper family for the original system.

\begin{lemma}\label{proeprcF} \textit{[\it{Growth Lemma for the system $(\cF,\cM,\mu).$}]} Let $\cG=(\cW,\nu)$ be a standard family for $(\cF,\cM,\mu)$, and $g=d\nu/d\mu$ is uniformly bounded. For any $n\geq 1$,
$$\cZ(\cF^n \cG)\leq C_z+cn^{-1-\alpha_0}\cZ (\cG|_{C_{n,b}^c})\nu(C_{n,b}^c) +\nu(C_{n,b})\cZ (\cG|_{C_{n,b}});$$
$$ \cZ (\cF^n(\cG|_{C_{n,b}^c}))\leq cn^{-1-\alpha_0}\cZ (\cG|_{C_{n,b}^c})+C_z.$$
In particular, let $\cG=(\cW,\nu)$ be a proper family on $M$, then  $\cZ(\cF^n \cG)$ is uniformly bounded. \end{lemma}
\begin{proof}

For any $k\geq 0$, we define $A_k^n$ as points in $\cM$ that returns to $M$ exactly $k$ times under $\cF^n$; and $A_0^n$ as the points that do not return to $M$ under $\cF^n$.  Clearly, $C_{n,b}=\cup_{k=0}^{(b\ln n)^2} A^n_k$. And let $\cG^u|_{A_k^n}=(A_k^n, \mu|_{A_k^n})$ be the standard family obtained from $\cG^u$ by restricting on $A_k^n$.
Let $\cG^u=(\cW^u,\mu)$ be the standard  family on $\cM$, generated by the SRB measure $\mu$.  We first decompose $\cG^u=\cG^u|_{C_{n,b}}\cup \cG^u_{C_{n,b}^c}$ into two standard families. Note that
$$\cZ(\cG^u)=\cZ(\cG^u|_{C_{n,b}})\mu(C_{n,b})+\cZ(\cG^u|_{C_{n,b}^c})\mu(C_{n,b}^c). $$
This implies that
$$\cZ(\cG^u|_{C_{n,b}})\leq \cZ(\cG^u)/\mu(C_{n,b}).$$
Similar estimations as in (\ref{cZbarg}) implies that
$$
\cZ(\cG|_{C_{n,b}})\nu(C_{n,b})\leq \cZ(\cG)\leq  (1+C_F)\|g\|_{\infty}\cZ(\cG^u).$$


For any standard pair $\cG=(W,\nu)$, any $n\geq 1$, using the remark before assumption (\ref{length WinMn}), there exists $N_0\geq 1$, such that $W\cap C_{n,b}$ consists of at  most $N_n\leq N_0n$ connected components, denoted as $V_1,\cdots, V_{N_n}$.  We let $\nu_k=\frac{\nu|_{V_k}}{\nu(V_k)}$ obtained by conditioning $\nu$ on $V_k$. Then $(V_k, \nu_k)$ is a standard pair, for $k=1,\cdots, N_n$. Since the forward images of $V_k$ returns to $M$ at least $(b\ln n)^2$ times under iterations of $\cF$, thus we use the Growth Lemma \ref{properagain} for the induced map:
\begin{align*}
\cZ(\cF^n(V_k,\nu_k) \leq c\vartheta_1^{(b\ln n)^2} \cZ(V_k,\nu_k)+C_z\leq \frac{c}{n^{1+\alpha_0}} \frac{1}{|V_k|^{q_0}} +C_z\end{align*}
here we require $b$ to be chosen large enough, such that \beq\label{condonb1}\vartheta_1^{(b\ln n)^2} \leq n^{-1-\alpha_0}.\eeq

Thus
\begin{align}\label{WnuCZ}
\cZ(\cF^n ((W,\nu)|_{C_{n,b}^c}))&=\sum_{k=1}^{N_n}\cZ(\cF^n(V_k,\nu_k) \nu(V_k)\nonumber\\
&\leq \frac{c}{n^{1+\alpha_0}} \sum_{k=1}^{N_n}\frac{1}{|V_k|^{q_0}} \nu(V_k)+C_z=\frac{c}{n^{1+\alpha_0}} \cZ((W,\nu)|_{C_{n,b}^c})+C_z.\end{align}

Now we consider a standard family $\cG=((W_{\alpha}, \nu_{\alpha}),\cA,\lambda)=(\cW,\nu)$, and let $\cG|_{C_{n,b}^c}=(\cW,\nu)|_{C_{n,b}^c}$.
Clearly,
$$\cF^n (\cG|_{C_{n,b}^c})=\int_{\alpha\in \cA}\sum_{k=1}^{N_{\alpha,n}} \cF^n (V_{\alpha,k}, \nu_{\alpha, k})\, \nu_{\alpha}(V_{\alpha,k} )\lambda(d\alpha),$$
where for any standard pair $(W_{\alpha},\nu_{\alpha})$, any $n\geq 1$,  there exists $N_{\alpha, n}\geq 1$, such that $W_{\alpha}\cap C_{n,b}$ consists of at  most $N_{\alpha,n}$ connected components, denoted as $V_{\alpha,1},\cdots, V_{N_{\alpha,n}}$.  We let $\nu_{\alpha,k}=\frac{\nu|_{V_{\alpha,k}}}{\nu(V_{\alpha,k})}$ obtained by conditioning $\nu_{\alpha}$ on $V_{\alpha,k}$. Then $(V_k, \nu_k)$ is a standard pair, for $k=1,\cdots, N_n$.
We apply (\ref{WnuCZ}) to get
\begin{align*}\cZ(\cF^n (\cG|_{C_{n,b}^c}))&=\int_{\alpha\in \cA}\sum_{k=1}^{N_{\alpha,n}} \cZ(\cF^n (V_{\alpha,k}, \nu_{\alpha, k}))\, \nu_{\alpha}(V_{\alpha,k} )\lambda(d\alpha)\\
&\leq \frac{c}{n^{1+\alpha_0}} \cZ((\cG|_{C_{n,b}^c}))+C_z.\end{align*}

If $\cG$ is proper, then we have
\begin{align*}
\cZ(\cF^n \cG)
&\leq cn^{-1-\alpha_0}\cZ (\cG|_{C_{n,b}^c})\nu(C_{n,b}^c) +C_z+\nu(C_{n,b})\cZ (\cG|_{C_{n,b}})\\
&\leq C_z+(cn^{-1-\alpha_0}+1)\cZ(\cG)\\&\leq C_z+(c+1) C_p,
 \end{align*}
 which is uniformly bounded.

 \end{proof}

\begin{defn}
We define the number  $C_{\rm q,\cF}=C_z+(c+1)C_p$.   Given a  standard family
$\cG=(\cW,\nu)$ on $\cM$.  We say $\cG$ is a \emph{proper} (standard)
family for the original system $\cF$, if  \beq\label{2nddefnproepr}
\cZ(\cG)<C_{\rm q,\cF}.\eeq
\end{defn}
It follows from Lemma \ref{bargWg}(iii) that, (\ref{defnpropercM}) and Lemma \ref{proeprcF} imply that  if $g=d\nu/d\mu$, and $\mu(\cW_g)>0$,  if  \beq\label{defnpropercM}
\|g\|_{\infty}\leq \nu(\cM)/(1+C_F),\eeq then (\ref{2nddefnproepr}) implies that:
\beq\label{cGucGZ}
\cZ(\cG)\leq \cZ(\cG^u)\leq C_{\rm q,\cF}.\eeq Note that by our definition and lemma \ref{proeprcF},  for any proper family $\cG$ of the induced map $F$,  $\cF^n \cG$ is a proper family for the original system.
We next consider properties about a standard pair $(W,\nu)$.
\begin{lemma}\label{chi1W} There exists $\delta=C_{q,\cF}^{-\frac{1}{q_0}}$, such that for any $|W|>\delta$, the standard pair $(W,\nu)$ is proper. Let $N_W:=\chi_1 |W|^{-\frac{1}{a}}$, where $\chi_1>1$ is a uniform constant. One denotes  $\cG=(W,\nu)|_{C_{N_W,b}^c}$. Then $\cF^n \cG$ is proper, for any $n\geq N_W$.
\end{lemma}
\begin{proof}
The first statement  follows from Definition 5. It is enough to prove the second one, using Lemma \ref{proeprcF} .
Assumption (\textbf{H2}) implies that there exist $a\leq \alpha_0$, and $n_1\leq C |W|^{-1/a}$, such that $\cF^{n_1} W$ returns to $M$ for the first time.  Now we consider $\cF^n(W,\nu)$, for $n\geq n_1$.   Note that $\cF^{n+n_1}_*\nu(\cF^{-n_1}C_{n,b})\leq C n^{-\alpha_0}$. Moreover, for any $m\geq n_1$, we denote $\cG_m:=(W,\nu)|_{C_{m,b}^c}$.
 Note that
all standard pairs in $\cF^{m}\cG_m$ have returned to $M$ at least $(b\ln n_1)^2$  times within the $m$-th iteration under $\cF$. Thus $$\cZ(\cF^m(\cG))\leq \cZ(F^{(b\ln m-n_1)^2}(\cF^{n_1}\cG)).$$ Lemma \ref{properagain} implies that  $F^{\chi|\ln |W||}(W,\nu)$ is already a proper family, so $\cF^{m} \cG$ is also a proper family, for any $m$ satisfies $(b\ln (m-n_1))^2\geq \chi|\ln |W||+1$.
Now combining with the definition of $n_1$, we know that there exists $\chi_1>1$, such that for any $n\geq N_W:=\chi_1 |W|^{-\frac{1}{a}}$, $\cF^n \cG_{N_W}$ is also a proper family.
\end{proof}

\section{Coupling Lemma for the induced map}

It was proved in \cite{C99, CZ09} that Assumptions \textbf{(h1)-(h4)} imply exponential decay of correlation for the induced system $(M, F,\mu_M)$ and  any  observables $f\in \cH^-(\gamma_1)$ and $g\in \cH^+(\gamma_2)$, where $\gamma_1,\gamma_2>0$, with $\supp\, f \subset M$ and $\supp\, g\subset M$.  More precisely, we have
\beq\label{expcz09}
    \left|\int_{M} (f\circ F^n)\, g\, d\mu_M -
    \int_{M} f\, d\mu_M    \int_{M} g\, d\mu_M\right|\leq C \|f\|^-_{_{C^{\gamma_1}}}\|g\|^+_{_{C^{\gamma_2}}} \vartheta^n,
\eeq
for some uniform constants $\vartheta=\vartheta(\gamma_1,\gamma_2)\in (0,1)$ and $C>0$.

 We  will  review in this section the
coupling method developed in \cite{D01,CD,CM,Y99} for the induced system, but we have to construct a special hyperbolic set.

\subsection{Construction of a hyperbolic set $\cR^*$}
We first construct  a  hyperbolic
set $\cR^*\subset M$ with  positive  measure,  which  will be used as the reference set for the coupling procedure.

\begin{defn}\label{defRsu} Let $\Gamma^{s}$ be a family of stable manifolds, and $\Gamma^u$ a family of unstable manifolds with length $\in( 10\delta_0, 20\delta_0)$. We say that $\cR^*=\Gamma^u\cap\Gamma^s$ is a {\it{hyperbolic set with product structure}}, if it satisfies the following four conditions:\\
(i) There exist  a family of stable manifolds $\hat\Gamma^{s}$,  a family of unstable manifolds $\hat\Gamma^u$, and a  region $\cU^*$  bounded by two stable manifolds $W^s_i\in \hat\Gamma^s$ and two unstable manifolds  $W^u_i\in \hat\Gamma^{u}$, for $i=1,2$;\\
(ii) Any stable manifold $W^s\in \hat\Gamma^s$ and any unstable manifold $W^u\in\hat\Gamma^u$ only intersect at exactly one point;\\
(iii) The two defining families $\Gamma^{s/u}$ are obtained by intersecting $\hat\Gamma^{s/u}$ with $\cU^*$, such that
$$\Gamma^{s/u}:=\hat\Gamma^{u/s}\cap \cU^*;$$
\noindent(iv) Let $\nu^u=\mu|_{\Gamma^u}$ be obtained by restricting the SRB measure on $\Gamma^u$, then $(\Gamma^u, \nu^u)$ defines a standard family, and $\nu^u(\Gamma^s)>0$.

We say a stable or unstable curve $W$ {\it{properly across}} $\cU^*$, if the two end points of the closure of $W\cap\cU^*$ are contained in the boundary $\partial \cU^*$.   We say a set $U^u\subset \cU$ is a
$u$-subrectangle, if there exist two unstable manifolds $W_1, W_2\in \Gamma^u$, such that $U^u$ is bounded by $W_1$ and $W_2$, as well as sharing the two stable boundary of $\cU$. Similar, we say a set $U^s\subset \cU$ is a
$s$-subrectangle, if there exist two stable manifolds $W_1, W_2\in \Gamma^s$, such that $U^s$ is bounded by $W_1$ and $W_2$, as well as sharing the two unstable boundary of $\cU$. Furthermore, we say  a set $A\subset \cR^*$ is a
$u$-subset, if there exists a solid $u$-rectangle $U^u\subset \cU$, such that $A=U^u\cap \Gamma^s$. Similarly a set $A\subset \cR^*$ is called a
$s$-subset, if there exists a s-subrectangle $U^s$, such that  $A=U^s\cap \Gamma^u$.
 \end{defn}

  It follows from  condition (iv) that  we can define a factor measure $\lambda$ on the sigma algebra (induced by the Borel $\sigma$-algebra of $\cM$) of the index set of $\Gamma^u=\{W_{\alpha},\alpha\in \cA\}$,  such that for any Borel set $A\subset \cU^*$,
 $$\nu^u(A)=\int_{\alpha\in\cA} \mu_{\alpha}(W_{\alpha}\cap A)\,\lambda(d\alpha).$$
Hyperbolic product sets were constructed in several references, see for example \cite{D01,CD,CM,Y98, CZ09, CWZ}.

\begin{proposition}\label{firstproper}
There exist $\hat\delta_1$,  a hyperbolic set with product structure $\cR^*=\Gamma^s\cap\Gamma^u$ and the   rectangle $\cU^*$ containing $\cR^*$ bounded by two stable manifolds and two unstable manifolds with length approximately $10\delta_0$, such that  the following properties hold:\\
 (i) $\mu(\cR^*)>\hat\delta_1$ and for any  unstable $W$ that fully crosses $\cU^*$, $\mu_{W}(\cR^*\cap W)>\hat\delta_1$;\\
(ii) There exists $n_0\geq 1$, such that  there exists a $s$-subset $ R_{n_0}\subset \cR^*$, with the property that $F^{n_0} R_{n_0}$ properly return to $\cR^*$ as a $u$-set; \\
(iii) Moreover, if a point $x\in \cR^*$ return to $\cR^*$ under $\cF^n$, for some $n\geq 1$, then $W^u(\cF^n x)\subset \Gamma^u$;\\
(iv) Let $W^{s/u}_i$, for $i=1,2$ be the stable/unstable boundary of $\cR^*$. We assume all forward images of $W^s_i$ will never enter the interior of the solid rectangle $\cU^*$, and all the backward images of $W^u_i$ will never enter the interior of the solid rectangle $\cU^*$ .
\end{proposition}
\noindent \textbf{Remark}: The construction of such hyperbolic set with property (i)- (iii) was done in details in \cite{CM, CZ09}, so we will not repeat it here.  In this paper,  to avoid notation complications, we always assume $n_0=1$.  Next we only prove property (iv), according to the arguments similar in \cite{CWZ}.

\begin{proof}

 Next we will be devoted to prove property (iv). However, we need to reconstruct a hyperbolic set $\tilde\cR^*$ satisfying properties (i) - (iii), our goal is by adding some more constructions, to get a new hyperbolic set $\cR^*$ which inherits the property (i) - (iii), yet also enjoys (iv).

More precisely, using techniques described  in \cite{CM, CZ09}, we first construct a family of stable/unstable manifolds $\tilde\Gamma^{s/u}$, a hyperbolic set $\tilde \cR^*=\tilde\Gamma^s\cap\tilde\Gamma^u$, together with a rectangle $\tilde\cU^*$ being the smallest  region containing $\cR^*$. Moreover they have the properties listed as in item (i) -  (iii).

Note that any stable manifold can only contains at most one periodic point. Thus using the dense property of periodic points in the support of the SRB measure $\mu$,  see \cite{H02} page 206-209, and \cite{Katok}, one can in addition show that most of these periodic points have long stable and unstable manifolds. Mainly because a periodic orbit only consists of a finite number of points,  they can easily get away from the singular set $S_{\pm 1}$ which only consists of countably many smooth curves.  Moreover, for any two pairs of periodic orbits $\gamma_i$ with period $p_i\geq 1$, $i=1,2$, the stable and unstable manifolds $W^{s/u}(\gamma_i)$ consists of al most $2(p_1+p_2)$ number of smooth  curves.
Thus for any small enough $\delta>0$,  there exist   two pairs of periodic orbits $\gamma_i$ with period $p_i\geq 1$, $i=1,2$,   and one (indeed many) hyperbolic rectangle $U^*$ with $\mu(U^*)<\delta$, such that its $s/u-$ boundary is contained in  $W^{s/u}(\gamma_i)$, and also has the property that the interior of $U^*$ does not contain  $W^{s/u}(\gamma_i)$.

We fix $\tilde\delta_1>0$ small enough, and fix two periodic trajectories $\gamma_p$ and $\gamma_q$ with period $p, q\geq 1$, such that a solid hyperbolic rectangle $U^*$ has the above specified properties, with defining stable/unstable families $\Gamma^{s/u}$. We denote $\cR^*=\Gamma^s\cap \Gamma^u$  as the hyperbolic product set.
 In addition, one can check that this new hyperbolic set has the properties given by item (i)-(iv).

\end{proof}

From now on, according to the construction in Proposition \ref{firstproper}, we will fix the hyperbolic set $\cR^*$, as well as its defining families $\Gamma^s$ and $\Gamma^u$, with \beq\label{defnR}
\cR^*=\Gamma^u\cap\Gamma^s.\eeq

In the coupling scheme that will be described below, we will  consider a standard pair $(W,\nu)$ by subtracting from its density  a smooth function. Next lemma explains that after one more iteration under $F$, the resulting measure also induces a standard pair.
\begin{lemma}\label{defnN}  Let $(W,\nu)$ be a standard pair properly crossing $\cR^*$, with $h=d\nu/d\mu_W$. Assume $ g\in DH(\gamma_0)$ is a dynamically Holder function, such that $(1-\ba)h/2<g<(1-\ba)h$, with $\ba=2\Lambda^{-\gamma_0}$. We denote $\eta$ as the measure with density $h_0=h-g$. Then $F(W,\eta/\eta(\cM))$ is a standard pair, and $F (W, \eta)$ is a standard family.
\end{lemma}
\begin{proof} By the definition of standard pair, we know that
  the positive density function $h=d\nu/d\mu_W$ satisfies (\ref{standardpair}):
 \beq
 | h(x) -  h(y)|
\leq C_{F} h(x) d(x,y)^{\gamma_0}, \eeq where $\gamma_0\in (0,1)$ was given in (\ref{distor10}), and $C_F>C_{\br}$ is a fixed large constant.
Now for $ g\in \cH(\gamma_0)$, with $\|g\|_{\gamma_0}\leq 1$ and $(1-\ba)h/2<g<h(1-\ba)$,  we denote $$h_0=h-g$$ as the  density of $\eta$. Note that
$$|g(x)-g(y)|\leq C_F g(x) d(x,y)^{\gamma_0}.$$
Then one can check that the new probability density $h':=h_0/\eta(M)$ satisfies
\begin{align*}
|\ln h'(x)-\ln h'(y)|&\leq \frac{h'(x)-h'(y)}{h'(x)}\leq \frac{|h(x)-h(y)|}{h(x)-g(x)}+\frac{|g(x)-g(y)|}{h(x)-g(x)}\\
&\leq \frac{|h(x)-h(y)|}{\ba h(x)}+\frac{(1-\ba)|g(x)-g(y)|}{\ba g(x)}\\
&\leq \frac{2 C_F}{\ba}d(x,y)^{\gamma_0}\leq \Lambda^{\gamma_0}  C_F d(x,y)^{\gamma_0},
\end{align*}
which implies that $(W,\eta/\eta(M))$ is a psudo-standard pair.

We define $F^n W=\{V_{\alpha}, \alpha\in \cA_n\}$, where $\cA_n$ is a countable index set.

Note that by (\ref{dens}) and (\ref{cmW}), for any $x\in V_{\alpha}$, the $u$-SRB density $\rho_{V_{\alpha}}=d\mu_{V_{\alpha}}/dm_{V_{\alpha}}$ satisfies
$$\rho_{F^{-n}V_{\alpha}}(F^{-n}x)\cJ_{V_{\alpha}}(F^{-n} x)=\rho_{V_{\alpha}}(x) .$$

We denote  for any $n\geq 1$, the density function  $h_n$ of $F^n_*\eta$ as
\begin{align*}
   h_n(x)&=\frac{d F_*^n\eta(x)}{d\mu_{V_{\alpha}}(x)}=
   \frac{d\eta(F^{-n}x)}{d\mu_{F^{-n}V_{\alpha}}(F^{-n} x)}\cdot \frac {d\mu_{F^{-n}V_{\alpha}}(F^{-n} x)}{d m_{F^{-n}V_{\alpha}}( F^{-n}x)}\cdot \frac{d m_{F^{-n}V_{\alpha}}( F^{-n}x)}{d m_{V_{\alpha}( x)}}\cdot \frac{dm_{V_{\alpha}}( x)}{d\mu_{V_{\alpha}( x)}}\\
   & = h_0(F^{-n}(x))\cJ_{V_{\alpha} (F^{-n}(x))}\cdot \frac{\rho_{F^{-n}V_{\alpha}}(F^{-n}x)}{\rho_{V_{\alpha}}(x)}\\
&=h_0(F^{-n}(x)),\end{align*}
for all $ x\in V_{\alpha}\subset F^nW$.

According to above analysis, and use the notation $\mu_{\alpha}=\mu_{V_{\alpha}}$, one can check that for any measurable set $A$,
\begin{align*}
F^n_*\eta(A)&=\int_W \chi_A(F^n x) h_0(x)\,d\mu_W(x)\\
&=\sum_{\alpha\in \cA_n} \int_{V_{\alpha}} \chi_A(y) h_0(F^{-n}y) d \mu_{\alpha}(y)\\
&=\sum_{\alpha\in \cA_n} \int_{V_{\alpha}} \frac{\chi_A(y) h_0(F^{-n}y)}{\mu_{\alpha}(h_0\circ F^{-n})}\, d\mu_{\alpha}(y) \cdot \mu_{\alpha}(h_0\circ F^{-n})\\
&=\sum_{\alpha\in \cA_n}\int_{V_{\alpha}}\chi_A(x) d\nu_{\alpha}\,\cdot \lambda_n(\alpha),\end{align*}
 where $$d\nu_{\alpha}= h_{\alpha}\,d\mu_{\alpha},$$ and \beq\label{lambdan}\lambda_n(\alpha)=\int_{V_{\alpha}} h_0(F^{-n} y)\, d\mu_{\alpha}(y)\eeq is the factor measure on index set $\cA_n$. Here  $h_{\alpha}$ is the probability density function defined only on $V_{\alpha}$ such that for any $x\in V_{\alpha}$, $$ h_{\alpha}(x)=\frac{h_0(F^{-n} x)}{\int_{V_{\alpha}}h_0(F^{-n} y)d\mu_{\alpha}(y)}.$$
Note that using the fact that $h/3\leq g\leq h/2$, we have for any $x, y\in V_{\alpha}\in F^n W$,
\begin{align*}
| \ln h_{\alpha}(x)- \ln h_{\alpha}(y)|&=|\ln  h_0(F^{-n}x)-\ln  h_0(F^{-n}y)|\leq  \Lambda^{\gamma_0}C_F d(F^{-n}x, F^{-n}y)^{\gamma_0} \leq  \Lambda^{\gamma_0} C_F \Lambda^{-n \gamma_0} d(x,y)^{\gamma_0}.\end{align*}
Thus we know that  $F(W,\eta)$ is already a standard family.
If we define $\eta'=\eta/\eta(\cM)$, then one can check that $F(W,\eta')$ is indeed a standard pair.
\end{proof}


\begin{lemma}\label{t01} There exist  $\hat\delta_0\in (0,\mu(\Gamma^s))$, $N_0>1$, such that for any  proper family $\cG=(\cW,\nu)$ with $\nu(M)=1$, then  $F^{n}_*\nu$ has at least $\hat\delta_0$ portion of measure properly returned to $\cR^*$, for any $n\geq N_0$.
\end{lemma}
\begin{proof}
Let $\hat\delta_1>0$ be defined as in Proposition \ref{firstproper},   such  that  $\mu(\cR^*)>\hat\delta_1$.
By the uniform mixing property of the induced map $(F,M,\mu_M)$, and the fact that $\cG$ is a proper family,  (\ref{expcz09}) implies that for $n>1$,
 $$|F^n_*\nu(\cR^*)- \mu_M(\cR^*)|\leq C \vartheta^n.$$  Moreover for any standard pair $(W_{\alpha},\nu_{\alpha})$, since $W_{\alpha}$ only has two end points, say $x_1, x_2$, so if $F^n W_{\alpha}$ intersects $\cR^*$ at some $x\in \cR^*$, then it must consist of  the entire unstable manifold $W^u(x)$, unless $W^u(x)$ consists of  one of points in  $\{F^n x_1, F^n x_2\}$. Thus a majority of curves in  $F^n \cW$ must properly cross  $\Gamma^s$.

 Thus by taking  a large $N_0$ and a small number $\hat\delta_0\in (0, \hat\delta_1)$, we have that for any $n\geq N_0$, $F^n_*\nu$ has at least $ \hat\delta_0$ portion of measure properly returned to $\cR^*$.
Moreover,  our choice of $N_0$ and $\hat\delta_0$ are uniform for all proper families.

\end{proof}

 We define  \beq\label{defn1}n_1=\max \{n_0,N_0\}\eeq Note that by Lemma \ref{defnN},  if we subtract a ``nice" function from the density of a proper standard pair $(W,\nu)$, then after $n_1$ iteration of $F$, the image $F^{n_1}(W,\eta)$ becomes a new proper family, where $\eta$ is the new conditional measure and has at least $\hat\delta_0$  portion of measure properly returned to $\cR^*$.

\subsection{The Coupling Lemma for the induced map}

 We first introduce a new concept called the {\it{generalized standard family}}.
 \begin{defn}\label{pseodugs}
 Let $(\cW,\nu)$ be a standard family, such that $\cW\subset \Gamma^u$ is a collection of $u$ subsets of $\cU^*$. Then we define $(\cW,\nu)|_{R^*}:=(\cW\cap \cR^*, \nu|_{\cR^*})$, which is call a generalized standard family with index 0.  For any $n\geq 1$, we call $(\cW_n,\nu_n):=\cF^{-n} ((\cW,\nu)|_{R^*})$ as an $\cF-$ generalized standard family with index $n$.
Similarly, we define $(\hat\cW_n,\hat\nu_n):=F^{-n} ((\cW,\nu)|_{R^*})$ as an $F-$ generalized standard family with index $n$.  \end{defn}

Next we restate  the Coupling Lemma \cite{CD,CM} for the induced
system $(F, \mu_M)$ using the concept of generalized standard families.

\begin{lemma}\label{coupling}
Under assumptions \textbf{(h1)-(h4)}. Let  $\cG^i=(\cW^i, \nu^i), i=1,2$,  be two  proper standard families on $M$.  There exist two sequences of $F-$ generalized standard families $\{(\cW^i_n,\nu^i_n), n\geq 0\}$,  such that \beq\label{decomposeGi}\cG^i=\sum_{n=0}^{\infty}(\cW^i_n,\nu^i_n):=\left(\bigcup_{n=0}^{\infty}\cW^i_n,\sum_{n=0}^{\infty}\nu^i_n\right).
\eeq And they also have the following properties, for each $n\geq 0$:
\begin{itemize}
\item[(i)] \textbf{Proper returned to $\cR^*$ at $n$.}\\
 \,\,\,\,\, Both $(\cW^1_n,\nu^1_n)$ and $(\cW^2_n,\nu^2_n)$ are $F-$ generalized standard families of index $n$;
  \item[(ii)] \textbf{Coupling $F^n_*\nu^1_n$ and $F^n_*\nu^2_n$ along stable manifolds in $\Gamma^s$.}\\
   \,\,\,\,\,\,\,For any measurable collection of stable manifolds $A\subset \Gamma^s$, we have $$F^n_*\nu_n^1(A)=F^n_*\nu^2_n(A).$$
\item[(iii)] \textbf{Exponential tail bound for uncoupled measure at $n$.}\\
     \,\,\,\,\, For any $n\geq 1$, we denote $\bar\nu_n^i:=\sum_{k\geq n}\nu^i_k$ as the uncoupled measure at $n$-th step, supported on $\bar\cW^i_n$, then
   \beq\label{ctail}\bar\nu_n^i(M)<C\vartheta^n,\eeq
   where $C>0$ and $\vartheta\in (0,1)$ are uniform constants.
      \item[(iv)] There exists $N_1\geq 1$, such that after each step of coupling, the remaining conditioned family $(\bar\cW^i_n, \bar\nu_n^i/\bar\nu_n^i(\cM))$ becomes a proper family again after another $F^{N_1}$ iteration.
   \end{itemize}

\end{lemma}

Note that the original Coupling Lemma was stated only for  proper families, so to deal with  a standard family which  is not proper, we need to iterate $N$ times to make it proper, according to Lemma \ref{properagain}. Moreover, $\cW^i_m$ and $\cW^i_n$ may not be disjoint, for $m\neq n$, unless during the coupling process, one can  couple the entire measure that properly returned to $\cR^*$.  Here $F^n_*\nu^i_n, i=1,2$, are the coupled components of $F^n_*\nu^i$.  In practice, item (iii) implies that a
coupling procedure occurs at a sequence of times $0<N_1<2N_1<\cdots<nN_1<\infty$. In
particular, $\nu^i_j=0$, when $j\neq n N_1$ for all $1\leq n$,
which means that $F^j_*\nu^i$ remains unchanged between successive
coupling times.

According to item (ii) of the above Lemma, for any bounded function $f$ that is constant on each stable manifold,  we have $F^n_*\nu^1_k(f)-F^n_*\nu^2_k(f)=0$. This implies that:
\begin{align}\label{fconstant}|F^n_*\nu^1(f)-F^n_*\nu^2(f)|&\leq \sum_{k=1}^n |F^{n-k}_*(F^k_*\nu^1_k(f)-F^k_*\nu^2_k(f))|+|\bar\nu^1_n(f)-\bar\nu^2_n(f)|\nonumber\\
&=|\bar\nu^1_n(f)-\bar\nu^2_n(f)|\leq 2C \|f\|_{\infty}\vartheta^n.
\end{align}

Similarly,  the above coupling lemma implies the exponential rates
for any bounded H\"{o}lder function $f\in \cH^-(\gamma_f)$, any proper families $(\cW^i,\nu^i)$, $i=1,2$, for the induced system $(F,M)$:
\begin{align}\label{fdecay1}
&\left|\int f \circ F^n \,d\nu^1 - \int f \circ F^n \,d\nu^2 \right| \leq  2 \|f\|_\infty  \bar\nu^1_n(M)  +
\sum_{j \leq n} \|f\|^-_{\gamma_f} \Lambda^{-(n-j)\gamma_f} \nu^1_j(M)
\nonumber\\
& \leq 2C \|f\|_{\infty}\vartheta^n+ C_d\|f\|^-_{\gamma_f}\Lambda^{-n\gamma_f}\sum_{j\leq n} (\Lambda^{\gamma_f}\vartheta)^j\leq C_1 \|f\|^-_{_{C^{\gamma_f}}} \vartheta_2^{ n},
\end{align}
where $C_1=2C+\frac{C_d}{|\Lambda^{\gamma_f}\vartheta-1|}$, and  $\vartheta_2=\max\{\vartheta, \Lambda^{-\gamma_f}\}$.

Next we will show that there is a generalized Markov partition on $M$, see (\ref{Yytower}), following the above Coupling Lemma. The proof of the existence of a generalized Markov partition as a consequence of the Coupling Lemma was first derived in \cite{CZ09} implicitly, and also proved in \cite{WZZ}.
\begin{proposition}\label{Ytower}
The induced map $F$ defines a   countable  partition of $\cR^*$ into  $s$-subsets $$\cR^*=\cup_{n\geq 1} \hat\cR_n$$ with the following properties:
\begin{itemize}
\item[(a)]  There exists $\hat\delta_2>0$ such that $\mu(\hat\cR_1)>\hat\delta_2$.
\item[(b)] For any $n\geq 1$,  and $F^n\hat\cR_n$ properly return to $\cR^*$ for the first time. and   \beq\label{ctail8} \mu_M(\hat\cR_n) < C\vartheta^n.\eeq
 \item[(c)] For each $n\geq 1$, there exists at most countably many $s$-subset $\hat\cR_{n,i}$, $i\geq 1$, such that $$\hat\cR_n=\cup_{i\geq 1}\hat\cR_{n,i},$$ where $F^n\hat\cR_{n,i}$ properly return to $\cR^*$ for the first time, for each $i\geq 1$.
 \item[(d)]
$gcd(n\geq 1\,:\, \mu_M(\hat\cR_n)>0)=1. $
   \end{itemize}
\end{proposition}
\begin{proof}
We first construct a partition of $\cR^*$ into $s$-subsets.  We apply the Coupling Lemma \ref{coupling} to the proper standard family $\cE:=(\cW^u_M,\mu_M)$, induced by   the family of all unstable manifolds $\cW^u_M$ together with the SRB measure $\mu_M$.  For $i=1,2$, let $\cE^i=\cE$ be two copies of the same family.  Then  Lemma \ref{coupling} implies that there exists a sequence of $F-$ generalized standard families $\{ (\cW_n,\nu_n), n\geq 1\}$, such that $\cE^i=\sum_{n\geq 0}(\cW_n,\nu_n)$, with properties given by (i)-(iii); and $\cW_n\cap \cR^*, n\geq 1$ are disjoint sets in $\cR^*$, with the property that $F^n \cW_n$ properly return to $\cR^*$.  Thus we define $\bar\cR^u_n:=F^n \cW_n$, which forms a decomposition of $\cR^*$ into $u$-subsets.

Next we show that $\bar\cR_n:=\cW_n\cap\cR^*$ are $s$-subsets.  Suppose $\cW_n\cap \cR^*$ consists of only one subset, which lies in the interior of a $s$-solid rectangle $U_n$, but sharing the stable boundary with $U_n$, denoted as $W^s_i$, for $i=1,2$.
Since the forward mages of stable manifolds  $W^s_i$  will never be broken, thus $F^n W^s_i$ forms the boundary of $F^n \cW_n$,  as well as the boundary of $F^n U_n$.

Case (i). Assume  $F^n U_n$ is also a solid $u$-rectangle of $\cU^*$. Since $\cW_n$ is the set in $\cR^*$ that properly returns to $\cR^*$ for the first time. The fact that $U_n\cap \cR^*$ also properly returns to $\cR^*$ for the first time, implies that $\cW_n=U_n\cap \cR^*$, which is a contradiction.

Case (ii). Assume $F^n U_n$ intersects one of the unstable boundaries  of $\cU^*$, which implies that the $F^{-n}$ image of the unstable boundary  of $\cU^*$ enters the interior of $\cU^*$. See Figure 1 (b). This contradicts Proposition \ref{firstproper} item (iv).

  Case (iii). Assume $U_n$ does not entirely contained in $\cU^*$, i.e. it has a nontrivial subset contained in the complement of $\cU^*$. Clearly the forward image of stable boundary of $\cU^*$ enters the interior of $\cU^*$, which is again a contradiction.\\

  Thus we have verified that $F^{-n}\cW_n$ is a collection of $s$-subsets.
 Let    $\hat\delta_2\in (0,\hat\delta_0)$ (where $\hat\delta_0$ was chosen in Lemma \ref{t01}), such that $\mu(\hat\cR^*_{1})>\hat\delta_2.$
This also verifies  item  (a).\\

Because our singular set $S_{\pm 1}$  contains at most countably many smooth curves, it is possible that for an unstable manifold $W\in \Gamma^u$, its  image $F^n W$ contains countably many smooth components that  properly returned to $\cU^*$, for $n\geq 1$. This implies that for each $n\geq 1$, there exists  at most countably many solid minimal rectangles $\{U_{n,k}, k\geq 1\}$, such that for each $k\geq 1$, $U_{n,k}$ is the smallest rectangle, such that $F^n(U_{n,k}\cap \hat\cR_n)\cap \cR^*$  is a $u$-subset of $\cR^*$.
 We define $\hat\cR_{n,k}= U_{n,k}\cap \hat\cR_n$. Then $\{\hat\cR_{n,k}, k\geq 1\}$ are disjoint sets, and $\hat\cR_n=\cup_{k\geq 1} \hat\cR_{n,k}$.

	 We define $$\hat\Gamma^s_{n,k}=\{W^s(x)\in \Gamma^s\,:\, x\in \hat\cR_{n,k}\}.$$ i.e. $\hat\Gamma^s_{n,k}$ is the collection of stable manifolds in $U_{n,k}\cap\Gamma^s$. We denote $\hat\Gamma^s_n=\cup_{k\geq 1} \hat\Gamma^s_{n,k}$. Then we can check that  $$\Gamma^s=\cup_{n=1}^{\infty}\hat\Gamma^s_n\,\,\,\,\,(\text{mod} \,0),$$ with the following properties:
\begin{itemize}
\item[(1)]  $F^n\hat{ \Gamma}^s_n$ properly return to $\cR^*$ for the first time under $F$, and $\{\hat\Gamma^s_n, n\geq 1\}$ are almost surely disjoint $s$-subsets of $\cU^*$ in the following sense: $$\mu(\hat\Gamma^s_m\cap\hat\Gamma^s_n)=0,$$ for any $m\neq n$;
\item[(2)] Furthermore
   \beq\label{ctail4}\sum_{k=n}^{\infty} \mu(\hat\Gamma^s_k)<C\vartheta^n,\eeq
   where $C>0$ and $\vartheta$ is the constant in (\ref{ctail}).
\end{itemize}

This verifies property (b)-(c). Item (d) follows from (a) and the fact that the system is mixing.

\end{proof}

Note that  the partition $\cR^*=\cup_{n\geq 1} \hat\cR_n$ induces a first proper return time $\tau_0:\cR^*\to \mathbb{N}$, such that  each level set $(\tau_0=n)=\hat\cR_n^s$.
We also define a first proper return  map $T=F^{\tau_0}: \cR^*\to \cR^*$, such that for any $n\geq 0$, for $\mu$-almost every $x\in \hat\cR_n$, we put \beq\label{defntau}T x:=F^{n}x.\eeq This stopping time $\tau_0$   is crucial in our coupling scheme.  In addition according to our definition of  $T$, we know that $T$ only sees proper returns to $\cR^*$. This is important in proving the  Coupling Lemma  \ref{coupling1}.

Indeed one can easily build up the generalized Markov partition of $M$ based on the  partition $\cR^*=\cup_{n\geq 1}\hat\cR_n$:
\beq\label{Yytower}
M=\bigcup_{n\geq 1}\bigcup_{k=0}^{n-1} F^k \hat\cR_n=\bigcup_{k\geq 0}\bigcup_{n=k}^{\infty} F^k \hat\cR_n \,\,\,( \text{ mod }\, 0).
\eeq
Moreover,  the first proper return time $\tau_0: \cR^*\to\mathbb{N}$ can be extended to $M$, as the first proper hitting function: $\tau_0: M\to\mathbb{N}$, such that
 $$(x\in M \, :\, \tau_0(x)=n)=\bigcup_{k=0}^{\infty} F^k \hat R_{n+k}.$$ Thus $M$ can be decomposed as the union of these level sets of $\tau$:
 $$M=\bigcup_{n=1}^{\infty} (\tau_0=n)=\bigcup_{n=1}^{\infty} \bigcup_{k=0}^{\infty} F^k \hat R_{n+k}.$$
 Using (\ref{ctail8}), we know that
 $$\mu(x\in M,\,:\,\tau_0(x)\geq n)\leq \sum_{m=n}^{\infty}\sum_{k=1}^{\infty} \mu(\hat R_{m+k})\leq C_1 \vartheta^n. $$

Next we prove a lemma which directly follows from the above Coupling Lemma \ref{coupling}.
\begin{lemma}\label{extralemma1}
For any proper standard family $\cG=(\cW,\nu)$ in $M$, there exists a sequence of $F-$ generalized standard families $\{(\cW_{n_k},\nu_{n_k}), k\geq 0\}$  with the following properties:
\begin{itemize}
\item[(i)]  $(\cW_{n_k},\nu_{n_k})$  is a $F-$ generalized standard families of index $n_k$;
\item[(ii)] $F^{n_k}\cW_{n_k}$ return to $\cR^*$ properly for the first time almost surely;
\item[(iii)] $\cG=\sum_{k\geq 0} (\cW_{n_k},\nu_{n_k})$, with $\{\cW_{n_k}, n_k\geq 0\}$ are disjoint sets, a.s. and the index set satisfies $gcd\{n_k, k\geq 1\}=1$;
\item[(iv)] Furthermore  for any $n\geq 1$,
   \beq\label{ctail1}\nu(\tau_0>n)=\sum_{k=n}^{\infty} \nu_k(M)<C\vartheta^n,\eeq
   where $C>0$ and $\vartheta$ are the constants in (\ref{ctail}).\end{itemize}
\end{lemma}
\begin{proof} 
 For $i=1,2$, we take
	 $$\cG^i:=(\cW, \nu)$$ as two copies of the standard family. Then according to  the Coupling Lemma \ref{coupling}, we can couple the entire measure that  properly return to $\cR^*$ at each step $n\geq 0$. This guarantees item (ii). Moreover, this coupling process enables us to construct  a sequence of $F-$ generalized standard families  $\{(\cW_n,\nu_n), n\geq 1\}$, such that (i)-(iv)  hold (as stated in Lemma \ref{coupling}). In addition, $\{\cW_n, n\geq 0\}$ are almost surely disjoint.
	 Moreover, since one could couple everything that properly returned to $\cR^*$ at every step, so the remaining measure at $n$-th step is
	 $$\bar\nu_n(\cM)=\nu(\tau_0>n).$$
	 Thus $\nu(\tau_0>n)=\sum_{k=n}^{\infty} \nu_k(M)<C\vartheta^n,$
   where $C>0$ and $\vartheta$ are the constants in (\ref{ctail}).
\end{proof}

Note that it follows from the above Lemma that the tail bound for any proper standard family $(\cW,\nu)$ is uniformly bounded:
\beq\label{uniformbarnun}
\nu(\tau_0>n)\leq C\theta^n,
\eeq
with a constant $C=C(C_p)$, which does not depend on $\nu$.

Next we investigate the relation between the set $C_{n,b}$ defined in (\ref{Cnb}) and the reference set $\cR^*$. According to the definition of $C_{n,b}$, we know that for any $x\in C_{n,b}$, its stable manifold $W^s(x)\in C_{n,b}$. Indeed by Assumption (\textbf{H3}), we know that $$\mu_M(M\cap C_{n,b})= \cO(\mu_M(R>n)).$$ We would like to see similar property for standard pairs that properly cross $\cR^*$.
\begin{lemma}\label{nuwcb2}
There exists $ c_0>0$, such that for any standard pair $(W,\nu)$ properly crossing $\cR^*$, with $d\nu/d\mu_W=g\in \cH^+(\gamma_0)$, we have for any $n\geq 1$,
\beq\label{MUWCNB1}\nu(C_{n,b}\cap W\cap\Gamma^s)\leq c_0\mu_M(C_{n,b}\cap \Gamma^s).\eeq
\end{lemma}
\begin{proof}
Note that Proposition \ref{firstproper} implies that $\mu_M(\Gamma^s)>\hat\delta_1$.
Now we disintegrate the measure $\mu_M$ restricted on $\cR^*$ along unstable leaves in  $\Gamma^u=\{W_{\alpha},\alpha\in \cA\}$, and let $\lambda$ be the factor measure on the index set $\cA$, such that $\lambda(\cA)=\mu_M(\Gamma^s)$, and for any measurable set $A$,
$$\mu_M(\Gamma^s\cap A)=\int_{\alpha}\mu_{\alpha}(W_{\alpha}\cap A\cap \Gamma^s)\,\lambda(d\alpha).$$
By the absolute continuity of the stable holonomy map \textbf{(h3)}, there exist $0<c_1<c_2$, such that for any $\alpha, \alpha_1\in \cA$, any measurable set $A\subset \cM$ satisfies
\begin{align*}c_1\mu_{\alpha_1}(W_{\alpha_1}\cap A\cap \Gamma^s)&\leq \mu_{\alpha}(W_{\alpha}\cap A\cap \Gamma^s)\leq c_2\mu_{\alpha_1}(W_{\alpha_1}\cap A\cap\Gamma^s).\end{align*}
This implies that
$$\mu_{\alpha_1}(W_{\alpha_1}\cap A\cap \Gamma^s)\lambda(\cA)\leq c_1^{-1}\mu_M((\Gamma^s\cap A).$$
Now we take any unstable manifold $W\in \Gamma^u$, and $A=C_{n,b}$,  then we have proved that
$$\mu_{W}(W\cap C_{n,b}\cap \Gamma^s)\leq  c_1^{-1}\mu(\Gamma^s\cap C_{n,b})/\mu(\Gamma_s)\leq c_1^{-1} \hat\delta_1^{-1} \mu(\Gamma^s\cap C_{n,b}).$$
Since unstable manifolds in $\Gamma^s$ have length $>10\delta_0$, Lemma \ref{properagain} implies that  a standard pair $(W,\nu)$ is  proper whenever  $W$ cross $\Gamma^s$. So $(W,\nu)$ and $(W,\mu_W)$ are equivalent:
$$\nu(W\cap C_{n,b}\cap \Gamma^s)\leq C_1c_1^{-1} \hat\delta_1^{-1}\|g\|_{\infty} \mu(\Gamma^s\cap C_{n,b})\leq C_1c_1^{-1} \hat\delta_1^{-1} e^{\eps_{\bd}}\mu(\Gamma^s\cap C_{n,b}),$$
for some constant $C_1$ depending on $C_F$ in (\ref{standardpair}), where we used Lemma \ref{densitybd} in the last step estimation. Now we take $$c_0=C_1 c_1^{-1}\delta_1^{-1} e^{\eps_{\bd}},$$ then (\ref{MUWCNB1}) has been proved.
\end{proof}

\subsection{Improved proof of the coupling lemma for the induced map}

Given a proper standard family $\cG=(\cW,\nu)$, we  decompose $(\cW,\nu)$ into generalized standard families according to its first proper return to the hyperbolic set: $(\cW,\nu)=\sum_{n=0}^{\infty}(\cW_n,\nu_n)$, with $\nu_n=\nu|_{(\tau=n)}$.  Let $n_0\geq 1$ be the smallest integer $n$, such that $\nu_n(\cM)>0$. For any $n\geq n_0$, let $A_n(\cW)\subset \cW\cap F^{-n}\cR^*$ be the  subset in $\cW$ that properly returned to $\cR^*$ under $F^n$. Let $X_0=\tau_0$, we define $$T_n=X_0+X_1+\cdots+X_{n-1}$$ as the $n$-th arrival time, for $n\geq 0$, and $N(n)=\min\{k\geq 1\,:\, T_k>n\}$ as the delayed renewal process. Let $Z_n$  be the indicator variable of the event
$\{\text{a renewal occurs at instant n}\}$, i.e. $Z_n(x)=1$  if there exists $m$, such that $T_m(x)=n$ and $Z_n(x)=0$ otherwise. Clearly, $N(n)=Z_1+\cdots Z_n$, and $A_n=(Z_n=1)=\cup_{m=0}^{n-1} (T_{m}=n)$.

Then
 $$A_n(\cW)=\{x\in \cW\,:\, T_{k-1}(x)=n, \text{ for some } k=1,\cdots,n\},$$ as all points in $\cR^*$ that will return to $\cR^*$ properly after $n$-iterations.
 On the other hand, for any $n>n_0$, $A_n(\cW)$ can be written as the union of $n$ disjoint $s$-subsets \beq\label{AntaunW}A_n(\cW)=\cR_{n}\cup \left(F^{-1}\cR_{n-1}\cap A_1(\cW)\right)\cup\cdots \cup \left(F^{-(n-1)}\cR_{1}\cap A_{n-1}(\cW)\right),\eeq
where we define $A_{n_0}(\cW)=\cW_{n_0}$.
Moreover, if we denote $\alpha_n=\nu_n(\cR_n)$, then
\begin{align*}
\nu(A_n(\cW))&=\sum_{k=0}^{n-1}\nu(F^{-k}\cR_{n-k}\cap A_k(\cW))=\sum_{k=0}^{n-1}F^k_*\nu(\cR_{n-k}\cap F^kA_k(\cW))\\
&=\sum_{k=0}^{n-1} F^k_*\nu\left(\cR_{n-k} |F^kA_k(\cW)\right)\nu(A_k(\cW)).\end{align*}
This defines  a  renewal process.

In particular, for $\cG^u=(\cW^u_M, \mu_M)$, {one can check that $ F^k A_k(\cW^u_M)$ properly returns to $\cR^*$, and $ F^k A_k(\cW^u_M)=\cR^*$, i.e. $\mu_M(A_k(\cW^u_M))=\mu_M(\cR^*)$.}
Moreover, if we define $\mu^*=\mu_M|_{\cR^*}/\mu_M(\cR^*)$, then
$$\mu^*(A_n(\cW^u_M))=\sum_{k=0}^{n-1} F^k_*\mu^*\left(\cR_{n-k} |F^kA_k(\cW^u_M)\right)\mu^*(A_k(\cW^u_M)).$$
Moreover, $$\mu^*=\sum_{m=1}^{\infty} \mu_m^*,$$  where $\mu^*_m= \mu^*|_{\cR^m}$.

 \begin{proposition}\label{Fproperc} Let $(\cW,\nu)$ be a proper standard family with probability measure $\nu$. Then there exists a constant $C>0$, such that:
$$ |\mu_M(\cR^*)-\nu(A_m(\cW))|< C \theta^n,\,\,\,\,\,\,\,\text{ and }\,\,\,\,\,\,\,\,\sum_{m=n+1}^{\infty}
\nu((\tau=m)\cap F^{-n}\cR^*)\leq C \theta^n,$$ i.e. the tail distribution of points that return, but not properly return under the induced map $F$,  to $\cR^*$ is exponentially small.
  \end{proposition}
  \begin{proof}
  We  decompose $(\cW,\nu)$ into generalized standard families according to its first proper return to the hyperbolic set: $(\cW,\nu)=\sum_{n=0}^{\infty}(\cW_n,\nu_n)$.   Thus by the mixing property of the SRB measure, we know that
\begin{align*}\mu_M(\cR^*)&=\lim_{m\to\infty} F^m _*\nu(\cR^*)=\lim_{m\to\infty}\sum_{n=0}^{\infty} F^m_*
\nu_n( \cR^*)\\
&=\lim_{m\to\infty}\sum_{n=0}^{m} F^m_*
\nu_n(\cR^*)+\lim_{m\to\infty}\sum_{n=m+1}^{\infty} F^m_*
\nu_n((\tau=n)\cap \cR^*),
\end{align*}
where we used the fact that $\nu_n=\nu|_{(\tau=n)}$.

  We claim that for any $m\geq 1$, one has
$$\nu(A_m(\cW))=\sum_{n=1}^{m} F^m_*
\nu_n(\cR^*)= F^m_*\nu(\cR^*)-\sum_{n=m+1}^{\infty} F^m_*
\nu_n(\cR^*).
$$
Indeed note  that $\sum_{n=1}^{m} F^m_*
\nu_n(\cR^*)=\sum_{n=1}^{m}
\nu((\tau=n)\cap F^{-m}\cR^*)$ is the measure of points in $\cW$ that properly return to $\cR^*$ at time $m$, which is exactly  $\cF^m _*\nu( F^mA_m)=\nu(A_m(\cW))$. 

This implies that
  \begin{align*} \mu_M(\cR^*)=\lim_{m\to\infty}\left( \nu(A_m(\cW))+\sum_{n=m+1}^{\infty} F^m_*
\nu_n(\cR^*)\right).
\end{align*}

We first estimate
$|F^m_*\nu(\cR^*)-\mu_M(R^*)|$. Using the exponential decay for the induced system, see Proposition \ref{exp decay}, we know that
$$|F^m_*\nu(\cR^*)-\mu_M(R^*)|\leq C \theta^m.$$
Moreover, using (\ref{ctail1}), we get
$$\sum_{n=m+1}^{\infty} F^m_*
\nu_n(\cR^*))= \sum_{n=m+1}^{\infty}
\nu_n(F^{-m}\cR^*))\leq \sum_{n=m+1}^{\infty}
\nu_n(M)\leq C\theta^m.$$

Combining the above facts, we get
$$|\nu(A_m(\cW))- \mu_M(\cR^*)|\leq 2 C \theta^n.$$


\end{proof}

Note that
$$\mu_M(\cR^*\cap F^{-n}\cR^*)=\mu(\Gamma^u\cap F^{-n}\Gamma^s)=\mu_M(\bI_{\Gamma_*}\circ F^n\cdot \bI_{\Gamma_u})$$
 Using the exponential decay for the induced system, see Proposition \ref{exp decay}, and the fact that $(\Gamma^u, \mu|_{\Gamma^u})$ is a proper family, we get
 $$|\mu_M(\cR^*\cap F^{-n}\cR^*)-\mu_M(\cR^*)^2|\leq C\theta^n.$$

In particular, this implies that for $\mu^*=\sum_{n=1}^{\infty}\mu^*_n$, we have
\begin{align*}
\mu^*(A_m(\Gamma^u))&=F^m_*\mu^*(\cR^*)-\sum_{n=m+1}^{\infty} F^m_*
\mu^*_n(\cR^*)\\
&=\mu^*(\cR^*\cap F^{-m}\cR^*)-\sum_{n=m+1}^{\infty} \mu^*(\cR_n\cap F^{-m}\cR^*)=\mu_M(\cR^*)+\cO(\theta^n).
\end{align*}

\subsection{$\alpha$-mixing property for the induced map.}\label{alphamixforF}
We first recall the exponential decay of correlations for the system $ (M, F, \mu_M)$
for bounded dynamically H\"older observables, which was proven in \cite{CZ09}
by using the coupling lemma (see e.g. \cite{CM}).

\begin{proposition}[\cite{CZ09}]\label{exp decay}
There exist $C_0>0$ and $\vartheta_0\in (0, 1)$ such that
for any pair of functions $f\in \cH^+(\gamma_f)\cap L^\infty(\mu_M)$ and $g^-\in \cH(\gamma_g)\cap L^\infty(\mu_M)$
and $n\ge 1$,
\beq
\left| \IE(f\cdot g\circ F^n)  -\IE(f)\IE(g) \right| \le C_0\|f\|_{C^{\gamma_f}} \|g\|_{C^{\gamma_g}}\vartheta_0^n,
\eeq
where $\vartheta_0=\max\{\theta, \gamma_0^{1/4}\}<1$.\end{proposition}

We then introduce the following  natural family of $\sigma$-algebras for the system $F: M\to M$.
Recall that $S_{\pm n}$ is the singularity set of $F^{\pm n}$ for $n\ge 1$.
Let $\xi_0:=\{M\}$ be the trivial partition of $M$,
and denote by $\xi_{\pm n}$ the partition of $M$ into connected components of
$M\backslash F^{\mp(n-1)} S_{\pm 1}$ for $n\ge 1$.
Further, let
$$
\xi_m^n:=\xi_m\vee \dots \vee \xi_n
$$
for all $-\infty\le m\le n\le \infty$.
By Assumption (\textbf{H2}), $\xi_0^\infty$ is the partition of $M$ into maximal unstable manifolds,
and $\xi_{-\infty}^0$ is that into maximal stable manifolds.
Also, $\mu(\partial \xi_m^n)=0$ by Assumption (\textbf{H4}), where
$\partial \xi_m^n$ is the set of boundary curves for components in $\xi^n_m$.

Let $\fF_m^n$ be the Borel $\sigma$-algebra generated by the partition $\xi_m^n$.
Notice that
$\fF_{-\infty}^\infty$ coincides with the Boreal $\sigma$-algebra of  $M$.
We denote by $\fF:=\{\fF_m^n\}_{-\infty\le m\le n\le \infty}$ the family of those $\sigma$-algebras.

\begin{proposition}\label{prop: alpha mixing}
 The family $\fF$ is $\alpha$-mixing with an exponential rate, i.e.,
there exist $C_0>0$ and $\vartheta_0\in (0, 1)$ (which are the same as in Proposition \ref{exp decay}) such that
\beq
\sup_{k\in \IZ}  \alpha(\fF_{-\infty}^k, \fF_{k+n}^\infty)\le C_0\vartheta_0^n,
\eeq
where the definition of $\alpha(\cdot, \cdot)$ is given by \eqref{alpha}.
\end{proposition}

\begin{proof} By the fact that $F^{-k}\xi_m^n=\xi_{m+k}^{n+k}$ and the invariance of $\mu$, it suffices to show that
\beq\label{alpha}
\alpha(\fF_{-\infty}^0, \fF_{n}^\infty)=\sup_{A\in \fF_{-\infty}^0} \sup_{B\in \fF_{n}^\infty}
\left|\mu(A\cap B)-\mu(A)\mu(B) \right| \le C_0\vartheta_0^n.
\eeq
Since $A\in \fF_{-\infty}^0$ is a union of some maximal stable manifolds, we have that $\b1_A\in \cH^-(\gamma_0)$.
Similarly, $B\in \fF_{n}^\infty$ implies that $F^{-n}(B)\in \fF^\infty_0$
is a union of some maximal unstable manifolds,
and thus $\b1_{F^{-n}B}\in \cH^+(\gamma_0)$.
Therefore, by Proposition~\ref{exp decay}, for any $A\in \fF_{-\infty}^0$ and $B\in \fF_{n}^\infty$,
\beq
\left|\mu(A\cap B)-\mu(A)\mu(B) \right|=
\left| \IE(\b1_{F^{-n}B}\cdot \b1_A\circ F^n) -\IE(\b1_{F^{-n}B}) \IE(\b1_A ) \right|
\le C_0\vartheta_0^n.
\eeq
This completes the proof of Proposition~\ref{prop: alpha mixing}.
\end{proof}

\section{Markov tower for the original map}
\subsection{Construction of the generalized  Markov partition for the original map}
In this subsection, we will  construct  a countable Markov partition of $\Gamma^s$  for the nonuniformly hyperbolic map and then use the return time to $\cU^*$ to define a stopping time for our coupling scheme.  To investigate the  map $(\cF, \cM,\mu)$ based on the induced system $(F,M,\mu_M)$, we know that $F$ and $\cF$  share the same stable/unstable manifolds on $M$ almost surely. This allows us to use the same reference set $\cR^*$ and $\cU^*$, as well as the stable/unstable manifolds $\Gamma^{s/u}$ that defines $\cR^*$.
First we extend the partition according to the original map by the following construction.
  \begin{proposition}\label{YtowercM}
\begin{itemize}
\item[(i)] For any $n\geq 1$, the set $$\hat\cR_n=\bigcup_{m\geq n} \cR_{n,m}$$ has a decomposition into s-subsets $\cR_{n,m}$, such that  for any nontrivial  $\cR_{n,m}$, the set  $\cF^m\cR_{n,m}$ properly returns to $\cR^*$ for the first time under $\cF$;
\item[(ii)] $\cR^*$ has a partition into $s$-subsets $\cR^*=\cup_{n\geq 1} \cR_n$, such that for any nontrivial set $\cR_n$, the set $\cF^n \cR_n$ properly returns to $\cR^*$ for the first time under iterations of $\cF$;
 \item[(iii)] There exist $\hat\delta_3\in (0,\hat\delta_2)$,  such that  $\mu(\cR_{1})>\hat\delta_3$.
 \item[(iv)] There exist $C, C_1>0$, such that any $n\geq 1$,
 $\mu(\cup_{m\geq n}\cR_m)\leq C \mu(\Gamma^s\cap  C_{n,b})\leq C_1 n^{-\alpha_0}.$
 \item[(v)] Let $\cN=\{n\geq 1\,:\, \mu(\cR_{n})>0\}$ be the index set for the nontrivial sets $\cR_{n}$.  Then $gcd(\cN)=1$.
     \end{itemize}
 \end{proposition}

 \begin{proof}

 By Proposition \ref{Ytower}, the set $\Gamma^s$ has a countable partition into $s$-subsets   $$\Gamma^s=\cup_{n\geq 1} \hat\Gamma^s_n.$$


Inductively for any $n\geq 1$, and $m\geq n$, we can define $\Gamma_{n,m}$,  which is the maximal $s$-subset of $\Gamma^s$, such that $\cF^{m} \Gamma_{n,m}$ is a $u$-subset of $\cU^*$. Moreover $\Gamma_{n,m}$ will also properly return to $\cU^*$ under the induced map $F^n$.  Then it follows that
$$\hat\Gamma^s_n=\bigcup_{m\geq n} \Gamma_{n,m}$$ is a disjoint decomposition of $\hat\Gamma^s_n$, for any $n\geq 1$.
 Next we rearrange $\{\Gamma_{n,m}\}$ according to the index $m$. Note that  \begin{align*}\Gamma^s&=\bigcup_{n\geq 1}\hat\Gamma^s_n=\bigcup_{n\geq 1}\bigcup_{m\geq n}\Gamma_{n,m}=\bigcup_{m=1}^{\infty}\left(\bigcup_{n=1}^{m}\Gamma_{n,m}\right)=\bigcup_{m=1}^{\infty}\Gamma^s_m,\end{align*} where $$\Gamma_m^s=\cup_{n=1}^{m} \Gamma_{n,m}.$$ Then by the definition of $\Gamma_{n,m}$, we know that  $\cF^m \Gamma^s_m$ returns properly to $\cR^*$, and is nonempty.
 Moreover, $$\mu(\hat\Gamma_{1})>\hat\delta_2$$ implies that  $$\mu(\Gamma^s_{1})>\hat\delta_3,$$ for some $\hat\delta_3\in (0,\hat\delta_2)$.

Now we define $\cR_m=\Gamma^s_m\cap\cR^*$, and $\cR_{m,n}=\Gamma^s_{m,n}\cap\cR^*$.   Item (v) follows from Item (c) of Proposition \ref{Ytower}. Therefore, the above analysis verifies items (i)-(iii).

Note that \begin{align*}
\mu_M(\cup_{m=n}^{\infty}\Gamma^s_m)&\leq \mu_M( \cup_{m=n}^{\infty}\Gamma^s_m\cap C_{n,b})+\mu_M( \cup_{m=n}^{\infty}\Gamma^s_m\cap C_{n,b}^c) \\
&\leq C \mu_M(\Gamma^s\cap  C_{n,b})+\mu_M(\cup_{m=n}^{\infty}\Gamma^s_m\cap C_{n,b}^c).
\end{align*}
 We denote $$E_n:= \cup_{m=n}^{\infty}\Gamma^s_m \cap C_{n,b}^c,$$ which contains points that have returned to $M$ at least $(b\ln n)^2$ times within $n$ iterations under $\cF$. We claim that \beq\label{muMEn}\mu_M(E_{n})=\cO(\mu_M(C_{n,b}\cap M)/n).\eeq
 This implies that $$\mu(\cup_{m\geq n}\cR_m)\leq \mu(\cup_{m=n}^{\infty}\Gamma^s_m)\leq 2C \mu(\Gamma^s\cap  C_{n,b})$$ as desired.

We prove  our claim  (\ref{muMEn}) by considering $E_{n}$, and we divide the proof into two cases.

\quad(\textbf{a}). Let $I_n$ be all points $x\in E_{n}$, such that  the  forward orbit of $x$ hits $\cR^*$ at most $b\ln n$ times within $n$ iterations.
Then there exists $k\in [1,n-b\ln n]$, such that $\cF^kx\in M$ and the forward orbit of $\cF^k x$ return to $M\setminus \cR^*$ at least $b\ln n $ consecutive times under $\cF$, but not hits $\cR^*$.

According to Lemma \ref{extralemma1},  we know that the standard family $(\cW^u_M,\mu_M)$ has a decomposition into $\{(\cW_n,\mu_n), n\geq 1\}$,  and by (\ref{ctail1}),
$\sum_{m=n}^{\infty}\mu_m(M)\leq C \vartheta^n$.
 Thus $\cF^k x$ belongs to the support of $\sum_{m\geq b\ln n} \mu_m$, which satisfies:
 \begin{align*}
 \mu_M(E_{n}\cap I_n)&\leq \sum_{m=b\ln n}^{\infty}\mu_m(M)\leq C\vartheta^{b\ln n}\leq C\mu_M(C_{n,b}\cap M)/n,\end{align*}
where we have used (\ref{dpbchi}) in the last step, by choosing $b$ large.

\quad(\textbf{b}). Now we consider points in $E_{n}\setminus I_n$. Then we claim that  for any $x\in (E_{n}\setminus I_n)$, iterations  of $x$ hit $\cR^*$ at least $b\ln n$ times within the $(b\ln n)^2$ returns to $M$. This is true  because otherwise there must be an interval of length $b\ln n$  in these at least $(b\ln n)^2$ returns to $M$ such that iterates of $x$ never hit $\cR^*$, and this contradicts the assumption that  (\textbf{a}) does not hold.  Thus there exists  $k\in [1,n-b\ln n]$, such that $\cF^kx\in \cR^*$ and the forward trajectory of $\cF^k x$ return to $\cR^*$ at least $b\ln n $  times. Note that $E_{n}\subset \cup_{m=n}^{\infty}\Gamma^s$, which implies that  the forward trajectory of $\cF^k x$ never return to $\cR^*$ properly within the next $n-k$ iterations. Thus again $\cF^k x$ belongs to the support of $ \bigcup_{m\geq b\ln n}\mu_m.$ Similar to case (\textbf{a}), we know that
 $$\mu_M(E_{n}\setminus I_n)\leq  C \mu_M(C_{n,b}\cap M)/n.$$

 This finished the proof of our  claim  (\ref{muMEn}).
 \end{proof}

\begin{remark}\label{remark2}
 By Proposition \ref{YtowercM},  we have $\mu(\Gamma^s_{1})>\hat\delta_3$; and by Proposition \ref{Ytower},  we have $\mu_M(\hat\Gamma^s_{1})>\hat\delta_2.$
\end{remark}

Similar to (\ref{Yytower}), we can decompose $\cM$ as:
\beq\label{YytowercM}
\cM=\bigcup_{n\geq 1}\bigcup_{k=0}^{n-1} \cF^k \cR_n \,\,\,( \text{ mod }\, 0).
\eeq
Now we define for any $n\geq 1$, the set \beq\label{levelset}
\cW^u_n:=\bigcup_{m=n}^{\infty} \cF^{m-n} \cR_m.
\eeq
Then one can check that each set $\cW^u_n$ has the property that $\cF^n \cW^u_n$ return to $\cR^*$ properly for the first time. This also enable us to define and extend the first proper hitting time function $\tau$ from $\cR^*$ to the full phase space $\cM$,  such that $\cW^u_n$ is the $n$-th level set of $\tau$, i.e. \beq\label{defntau0}(\tau=n)=\cW^u_n, \,\,\,\,a.s.\eeq
Next we explore a relation between $\mu(R>n)$ and $\mu(\tau>n)$.
\begin{lemma} For any standard family $(\cW,\nu)$,  with $g=d\nu/d\mu$, we have,
$$|\nu(\tau> n)-\nu(R> n)|\leq \|g\|_{\infty}\mu((R\leq n)\cap C_{n,b}\cap (\tau>n))
+\|g\|_{\infty}\cO(\mu_M(C_{n,b}\cap M))\leq C\|g\|_{\infty}n^{1-\alpha_0}.$$
Moreover, if $\supp(\nu)\subset M$, then
$$|\nu(\tau> n)-\nu(R> n)|\leq \|g\|_{\infty} n^{-\alpha_0}.$$\end{lemma}
\begin{proof}
For any $n\geq 0$, let
$\cW^u_n=(\tau=n)$ be its $n$-th level set, as in (\ref{levelset}). We define $\cE=(\cW^u,\mu)$ as the proper family generated by the invariant SRB measure $\mu$ on unstable manifolds $\cW^u$. For any $n\geq 0$, let
$$\cW^u_n=(\tau=n),\,\,\,\,\mu_n=\mu|_{(\tau=n)}.$$
Then one can check that $(\cW^u_n,\mu_n)$ is a $\cF-$ generalized standard family of index $n$, and
$$(\cW^u,\mu)=\sum_{n\geq 0}(\cW^u_n,\mu_n).$$
In addition, Item (iv) of Proposition \ref{YtowercM} implies that there exists a constant $C>0$ such that
$$\mu_n(\cM)=\mu(\tau=n)\leq C n^{-\alpha_0}.$$

Next   we claim that
\beq\label{cnbcmu}
\mu_M((\tau\geq n)\cap C_{n,b}^c\cap (R\leq n)) \leq C \vartheta^{b\ln n}\leq C\mu_M(C_{n,b}\cap M)/n.
\eeq

To see this, note that points in $(\tau\geq n)\cap C_{n,b}^c\cap (R\leq n)$  will
mostly visit cells with small indices  and return to $M$ at least $(b\ln n)^2$ times within $n$ iterations. But the  iterations of $x\in (\tau>n)$ hits $\cR^*$ only once within $n$ iterations.
Then there exists $k\in [1,n-b\ln n]$, such that $\cF^kx\in M$ and the forward trajectory of $\cF^k x$ return to $M\setminus \cR^*$ at least $b\ln n $ consecutive times under $F$.

According to Lemma \ref{extralemma1},  we know that the standard family $(\cW^u,\mu_M)$ has a decomposition into $\{(\cW_n,\nu_n), n\geq 1\}$,  and by (\ref{ctail1}),
$\sum_{m=n}^{\infty}\nu^i_m(M)\leq C \vartheta^n$.
 Thus $\cF^k x$ belongs to the support of $\sum_{m\geq b\ln n} \nu^i_m$, which satisfies:
 \begin{align*}
 \mu_M(C_{n,b}^c\cap (R\leq n)\cap (\tau\geq n))&\leq \sum_{m=b\ln n}^{\infty}\nu^i_m(M)\leq C\vartheta^{b\ln n}.\end{align*}
This finishes the claim, so we have
$$\mu((\tau\geq n)\cap C_{n,b}^c\cap (R\leq n)) \leq C\mu_M(C_{n,b}\cap M).$$

Combining the fact that
$\mu((\tau\geq n)\cap C_{n,b}) \leq \mu(C_{n,b}),$
which implies that
\begin{align*}\mu(\tau>n)&=\mu(R>n)+\mu((R\leq n)\cap C_{n,b}\cap (\tau>n))
+\mu((R\leq n)\cap C_{n,b}^c\cap (\tau>n))\\
&= \mu(R>n)+\mu((R\leq n)\cap C_{n,b}\cap (\tau>n))
+\cO(\mu_M(C_{n,b}\cap M))\\
&=\mu(C_{n,b}\cap (\tau>n))
+\cO(\mu_M(C_{n,b}\cap M)).\end{align*}

The above analysis can be applied to a general standard family $\cG=(\cW,\nu)$, which leads to
\begin{align*}          \nu((\tau>n)\setminus (R>n))&\leq \nu((R\leq n)\cap C_{n,b}\cap (\tau>n))
+\cO(\nu (C_{n,b}\cap M)). \end{align*}
In particular, if $\supp(\nu)\subset M$, then
\begin{align*} \nu((\tau>n)\setminus (R>n))&=\nu(C_{n,b}\cap (\tau>n))
+\cO(\nu(C_{n,b}\cap M))\leq\|g\|_{\infty} \mu((R\leq n)\cap C_{n,b}\cap (\tau>n))\\
&\leq\|g\|_{\infty}\mu(M\cap C_{n,b}\cap (R<n))\leq C\|g\|_{\infty} n^{-\alpha_0}.\end{align*}
Thus
\begin{align*}|\nu(\tau>n)-\nu(R>n)|&\leq C\|g\|_{\infty} n^{-\alpha_0}.\end{align*}
\end{proof}

 \begin{lemma}\label{extralemma2}
Let $\cG=(\cW,\nu)$ be a  standard family. Then there exists a sequence of $\cF-$ generalized standard families $\{(\cW_n,\nu_n), n\geq 0\}$, such that $\cG=\sum_{n=0}^{\infty}(\cW_n,\nu_n)$; moreover,  $$\nu_n(\cM)\leq  \|g\|_{\infty}\mu(\cM),$$ for some constant $C>0$.
In addition, let $\cN=\{n\geq 1\,:\, \nu_n(\cM)>0\}$, then $gcd \cN=1$.
\end{lemma}
   \begin{proof}
  We  extend the first return time function $\tau$ from  $\cR^*$ to $\cM$, as in (\ref{YytowercM}).
Next, we consider $\cG=(\cW,\nu)$, which  is a  proper   standard family, such that $\cW\subset \cW^u$ is a measurable partition of a Borel set $B\subset \cM$ into unstable manifolds, and  $\mu(B)>0$, $g=d\nu/d\mu\in\cH^+(\gamma_0)$. Then we define $\cW_n:=\cW^u_n\cap \cW$, and $\nu_n=\nu|_{\cW_n}$. Then $(\cW_n,\nu_n)$ is a $\cF-$ generalized standard family of index $n$.
Note that $\nu_n(\cM)\leq \|g\|_{\infty}\mu_n(\cM)$.
Using Proposition \ref{YtowercM}, we know that $$\nu_n(\cM)\leq C_{\mu} \|g\|_{\infty} n^{-\alpha_0}\leq \frac{C_{\mu} }{1+C_F} n^{-\alpha_0},$$
where we used the definition 5, especially (\ref{defnpropercM}).
This verifies the two statements as claimed.
 \end{proof}

\begin{lemma}\label{extralemma2}
Let $\cG=(\cW,\nu)$ be a   proper standard family. Then there exists a sequence of $\cF-$ generalized standard families $\{(\cW_n,\nu_n), n\geq 0\}$, such that $\cG=\sum_{n=0}^{\infty}(\cW_n,\nu_n)$; moreover,  $$\nu_n(\cM)\leq C n^{-\alpha_0},$$ for some uniform constant $C>0$.
In addition, let $\cN=\{n\geq 1\,:\, \nu_n(\cM)>0\}$, then $\gcd (\cN)=1$.
\end{lemma}
   \begin{proof}
  We  extend the first return time function $\tau_0$ from  $\cR^*$ to $\cM$, as in (\ref{YytowercM}).  For any $n\geq 0$, let
$\cW^u_n=(\tau_0=n)$ be its $n$-th level set, as in (\ref{levelset}). We define $\cE=(\cW^u,\mu)$ as the proper family generated by the invariant SRB measure $\mu$ on unstable manifolds $\cW^u$. For any $n\geq 0$, let
$$\cW^u_n=(\tau=n),\,\,\,\,\mu_n=\mu|_{(\tau=n)}.$$
Then one can check that $(\cW^u_n,\mu_n)$ is a $\cF-$ generalized standard family of index $n$, and
$$(\cW^u,\mu)=\sum_{n\geq 0}(\cW^u_n,\mu_n).$$
In addition, Item (iv) of Proposition \ref{YtowercM} implies that there exists a constant $C>0$ such that
$$\mu_n(\cM)\leq C n^{-\alpha_0}.$$

Next, we consider $\cG=(\cW,\nu)$, which  is a  proper standard family, such that $\cW\subset \cW^u$ is a measurable partition of a Borel set $B\subset \cM$ into unstable manifolds, and  $\mu(B)>0$, $g=d\nu/d\mu\in\cH^+(\gamma_0)$. Then we define $\cW_n:=\cW^u_n\cap \cW$, and $\nu_n=\nu|_{\cW_n}$. Then $(\cW_n,\nu_n)$ is a $\cF-$ generalized standard family of index $n$.
Note that $\nu_n(\cM)\leq \|g\|_{\infty}\mu_n(\cM)$.
Using Proposition \ref{YtowercM}, we know that $$\nu_n(\cM)\leq C \|g\|_{\infty} n^{-\alpha_0}.$$

This verifies the two statements as claimed.
 \end{proof}

  To make comparison, we  need to construct a generalized standard family of index zero on $\cR^*$. Let $\mu^*=\mu|_{\cR^*}/\mu(\cR^*)$ be the conditional measure of $\mu$ on $\cR^*$. Then $(\cR^*,\mu^*)$ can be viewed as a generalized standard family of index zero.  Moreover, $(\cR^*,\mu^*)$ can be decomposed into a sequence of $\cF-$ generalized standard families, i.e.
  $(\cR^*,\mu^*)=\sum_{n=1}^{\infty} (\cR_n, \mu^*_n)$
  with $\mu^*_n=\mu^*|_{\cR_n}$, and $(\cR_n,\mu^*_n)$ is a generalized standard family of index $n$.\\

\begin{lemma}\label{Tmerateo}  Let $\cG^i=(\cW^i,\nu^i)=((W_{\alpha},\nu_{\alpha}),\alpha\in \cA^i, \lambda^i)$ , $i=1,2$, be two generalized families of index $0$, such that  $\cW^i=\cW^i\cap \cR^*$ properly cross $\cR^*$.   Assume that  $\nu^1(\cR^*)=\nu^2(\cR^*)$.\\
(i) For any $m\geq n\geq 1$  we have
\beq\label{nu1gammanmm1}|\nu^1(\Gamma_{n,m})-\nu^2(\Gamma_{n,m})|\leq \eps_{\bd} \nu^2(\Gamma_{n,m}).\eeq
 Here  $\Gamma_{n,m}\subset \Gamma^s$, such that $F^n \Gamma_{n,m}=\cF^m \Gamma_{n,m}$ is a $u$-subset of $\cR^*$.\\
(ii) For any $n\geq 1$, we have
$$|\nu^1(\cR_n)-\nu^2(\cR_n)|\leq \eps_{\bd}\nu^2(\cR_n).$$
\end{lemma}
\begin{proof}
Note that (i) implies (ii) according to Proposition \ref{YtowercM}, as $\nu^i(\cR_n)=\sum_{k=1}^{n}\nu^i(\Gamma_{k,n})$, for $i=1,2$. Thus $$|\nu^1(\cR_n)-\nu^2(\cR_n)|\leq \eps_{\bd} \sum_{k=1}^n \nu^2(\Gamma_{k,n})\leq  \eps_{\bd}\nu^2(\cR_n).$$

Now it suffices to prove  statement (i) .
Let $\cG^i=(\cW^i,\nu^i)=((W_{\alpha},\nu_{\alpha}),\alpha\in \cA^i, \lambda^i),$  $i=1,2$, be two pseudo generalized  families of index $0$.
For any $\alpha\in \cA^i$, let $W^u_{\alpha}$ be the (unique) unstable manifold in $\Gamma^u$ that contains $W_{\alpha}$; and $\nu^u_{\alpha}$ a regular measure on $W_{\alpha}^u$, with $\nu^u_{\alpha}|_{\Gamma^s}=\nu_{\alpha}$. Then  we can start from the standard family $$\hat{\cG}^i:=((W^u_{\alpha},\nu^u_{\alpha}),\alpha\in \cA^i, \lambda^i).$$ Now we can use   the assumption (\ref{Jh}) on distortion bounds for the Jacobian of the stable holonomy map defined by $\Gamma^s$. More precisely, for any $\alpha\in \cA^1$, $\beta\in \cA^2$, we define $$\bh_{\alpha,\beta}: W_{\alpha}\to W_{\beta}$$ as the stable holonomy map, with $$\bh_{\alpha,\beta}(x)=W^s(x)\cap W_{\beta},$$ for any $x\in W_{\alpha}\cap \Gamma^s$.  Then by the absolute continuity property of the holonomy map, especially   (\ref{Jh}), as well as the fact that $\dist(W_{\alpha}, W_{\beta})\leq 20\delta_0$, for any $W_{\alpha},W_{\beta}\in \Gamma^u$.  Thus we have  $$|\ln \bh_{\alpha,\beta}|\leq C_{\br} \dist(W_{\alpha}, W_{\beta})^{\gamma_0}\leq C_{F} (20\delta_0)^{\gamma_0}<\eps_{\bd}/10,$$  by the choice of $\delta_0$  in Lemma \ref{densitybd}. Thus combining with Lemma \ref{densitybd}, for any measurable collection of stable manifolds $A\subset \Gamma_{n,m}$, any standard pair $( W^u_{\alpha}, \nu^u_{\alpha})$ and $( ( W^u_{\beta}, \nu^u_{\beta}))$, we have
\beq\label{alphbeta1}|\nu^u_{\alpha}(A\cap W^u_{\alpha})-\nu^u_{\beta}(A\cap  W^u_{\beta})|\leq \eps_{\bd}\,\nu^u_{\beta}(A\cap W^u_{\beta}).\eeq
Since $\nu^u_{\alpha}|_{\Gamma^s}=\nu_{\alpha}$ and $\nu^u_{\beta}|_{\Gamma^s}=\nu_{\beta}$, we indeed have

\beq\label{alphbeta}|\nu_{\alpha}(A\cap W^u_{\alpha})-\nu_{\beta}(A\cap  W^u_{\beta})|\leq \eps_{\bd}\,\nu_{\beta}(A\cap W^u_{\beta}).\eeq
Since (\ref{alphbeta}) is true for all $\alpha\in \cA^1$ and all $\beta\in \cA^2$, using the fact that $\lambda^i(\cA^i)=1$, we have,
$$|\nu^1(\Gamma_{n,m})- \nu^2(\Gamma_{n,m})|\leq  \eps_{\bd}\nu^1(\Gamma_{n,m}).$$

\end{proof}



\subsection{Proper returns to the hyperbolic set }

We consider another first proper return map $L :\cR^*\to \cR^*$, and a first proper return time $\tau:\cR^*\to\mathbb{N}$, such that each level set $(\tau=n)=\cR_n$ is a $s$-subset of $\cR^*$, and $\cF^{n } (\tau=n)$ properly return to $\cR^*$ for the first time. Moreover, for $\mu$-almost every $x\in (\tau=n)$,  we define $L x=\cF^{n }x$, for any $n\geq 1$. We denote $\mu^*=\mu|_{\cR^*}/\mu(\cR^*)$ as the conditional measure obtained from $\mu$ restricting on $\cR^*$. Then the new system $(L, \cR^*, \mu^*)$ is also a mixing system, with exponential decay rates:
$$|\hat\mu(L^n \cR_{m_1}\cap\cR_{m_2})-\hmu(\cR_{m_1})\hmu(\cR_{m_2})|\leq C\hmu(\cR_{m_1}) \vartheta^n,\,\,\,\,\,\,\,\forall m_1, m_2\geq 1.$$

For any standard family $\cG=(\cW,\nu)$, by applying Lemma \ref{extralemma2}, $\cG$ has a decomposition into generalized standard families $\cG=(\cW_n,\nu_n)$.
Let $A_n(\cW)\subset \cW\cap \cF^{-n}\cR^*$ be the  subset in $\cW$ that properly returned to $\cR^*$ under $\cF^n$. Let $X_0=\tau$, we define $$T_n=X_0+X_1+\cdots+X_{n-1}$$ as the $n$-th arrival time, for $n\geq 0$, and $$N(n)=\max\{k\geq 1\,:\, T_k\leq n\}$$ as the delayed renewal process. Let $Z_n$ be the  indicator variable of the event
$\{\text{a renewal occurs at instant n}\}$, i.e. $Z_n(x)=1$  if there exists $m$, such that $T_m(x)=n$ and $Z_n(x)=0$ otherwise. Clearly, $$N(n)=Z_1+\cdots Z_n,\,\,\,\,\,\,\,\,\text{ and }\,\,\,\,\,\,\,\,A_n(\cW)=(Z_n=1)=\cup_{m=0}^{n-1} (T_{m}=n).$$

Then
 $$A_n(\cW)=\{x\in \cW\,:\, T_{k-1}(x)=n, \text{ for some } k=1,\cdots,n\},$$ as all points in $\cW*$ that will return to $\cR^*$ properly after $n$-iterations. In particular, if we denote $A_n^*=A_n(\cR^*)$ and $A_0^*=\cR^*$, then
  \beq\label{AntaunWbb}A_n(\cW)=\bigcup_{k=1}^n \left(\cF^{-k}A^*_{n-k}\cap \cW_k\right). \eeq
 This implies that
 \begin{align}\label{nuAnpnR}
\nu(A_n(\cW))&=\sum_{k=1}^{n}\nu(\cF^{-k} A^*_{n-k}\cap \cW_k)\nonumber\\
&=\sum_{k=1}^{n}\cF^k_*\nu(A^*_{n-k}\cap \cF^k\cW_k)\nonumber\\
&=\sum_{k=1}^{n}\cF^k_*\nu(A^*_{n-k} |\cF^k\cW_k)\nu(\cW_k).\end{align}

 On the other hand, it can be written as the union of $n$ disjoint $s$-subsets \beq\label{AntaunW}A_n(\cW)=\cW_{n}\cup \left(\cF^{-1}\cR_{n-1}\cap A_1\right)\cup\cdots \cup \left(\cF^{-(n-1)}\cR_{1}\cap A_{n-1}\right),\eeq
where we define $A_1=\cW_1$.
Moreover, if we denote $\alpha_n=\nu_n(\cW_n)$, then
\begin{align}\label{nuAnpn}
\nu(A_n)&=\alpha_n+\sum_{k=0}^{n-1}\nu(\cF^{-k}\cR_{n-k}\cap A_k)\nonumber\\
&=\alpha_n+\sum_{k=0}^{n-1}\cF^k_*\nu(\cR_{n-k}\cap \cF^kA_k)\nonumber\\
&=\alpha_n+\sum_{k=0}^{n-1}\cF^k_*\nu(\cR_{n-k} |\cF^kA_k)\nu(A_k).\end{align}
This defines  a delayed renewal process.

We denote \beq\label{defnpk} p_{k}=\mu(\cR_{k}|\cR^*),\,\,\,\,\,\, p_k(n-k)= \nu(\cR_{n-k}| \cF^kA_k).\eeq Then  by Lemma \ref{Tmerateo}, we have
\beq\label{edtau0}
p_{n-k}(1-\eps_{\bd})\leq \nu(\cR_{n-k}|\cF^k A_k) \leq p_{n-k}(1+\eps_{\bd}).
\eeq

 \begin{proposition}\label{properc} Let $(\cW,\nu)$ be a  standard family with probability measure $\nu$ and $g=d\nu/d\mu$.\begin{itemize}
 \item[(1)]  $ |\mu(\cR^*)-\nu(A_n)|< C\|g\|_{\infty} n^{1-\alpha_0}$, for constant $C>0$, $n\geq 1$.
 \item[(2)]  Given two proper standard families, $(\cW^i,\nu^i)$, if $\cW^i\subset M$, for $i=1,2$, then  $ |\nu^1(A_n(\cW^1))-\nu^2(A_n(\cW^2))|< C_1 n^{-\alpha_0}$, for some constant $C_1>0$, any $n\geq 1$.
 \item[(3)]  In particular, there exists $c>0$, such that if  $(\cW,\nu)$  is a generalized standard family of index zero, then $\nu(A_n)\geq c\nu(\cR^*)$ for any $n\geq 1$.
\end{itemize}  \end{proposition}
  \begin{proof}   We  decompose $(\cW,\nu)$ into generalized standard families according to its first proper return to the hyperbolic set: $(\cW,\nu)=\sum_{n=0}^{\infty}(\cW_n,\nu_n)$.   Thus by the mixing property of the SRB measure, we know that
\begin{align*}\mu(\cR^*)=\lim_{m\to\infty} \cF^m _*\nu(\cR^*)=\lim_{m\to\infty}\sum_{n=0}^{\infty}\cF^m_*
\nu_n(\cR^*)=\lim_{m\to\infty}\sum_{n=0}^{m}\cF^m_*
\nu_n(\cR^*)+\lim_{m\to\infty}\sum_{n=m+1}^{\infty}\cF^m_*
\nu_n(\cR^*).
\end{align*}

  We claim that for any $m\geq 1$, one has
$$\nu(A_m(\cW))=\sum_{n=1}^{m}\cF^m_*
\nu_n(\cR^*).
$$
Indeed note  that $$\sum_{n=1}^{m}\cF^m_*
\nu_n(\cR^*)=\sum_{n=1}^{m}\cF^{m-n}_*\left(\cF^n_*\nu(\cF^n(\tau=n)\cap A^*_{m-n})\right)=\sum_{n=1}^{m}\nu(\cW_n \cap \cF^{-n}A^*_{m-n})$$ is the measure of points in $\cW\cap\Gamma^u$ that properly return to $\cR^*$ at time $m$, which is exactly  $ \nu(A_m(\cW))$ by (\ref{nuAnpnR}). 

This implies that
  \begin{align*} \lim_{m\to\infty}\left(\mu(\cR^*)-\nu(A_m(\cW))\right)=\lim_{m\to\infty}\sum_{n=m+1}^{\infty}\cF^m_*
\nu_n(\cR^*)=0.
\end{align*}

We first estimate
$|\cF^m_*\nu(\cR^*)-\mu(R^*)|$. Using assumption (\textbf{H3}), we know that
$$|\cF^m_*\nu(\cR^*\cap\cF^{m} C_{m,b})-\mu(R^*\cap \cF^{m}C_{m,b})|\leq C(1+\|g\|_{\infty}) n^{1-\alpha_0}.$$
On the other hand,
\begin{align*}
\sum_{n=1}^{m}&\cF^m_*
\nu_n(\cR^*\cap\cF^{m} C^c_{m,b})=\sum_{n=1}^{m}\cF^{m-n}_*
(\cF^n_*\nu_n)(\cR^*\cap\cF^{m} C^c_{m,b})\\
&=\sum_{n=1}^{m}F^{n_1}
(\cF^n_*\nu_n)(\cR^*\cap\cF^{m} C^c_{m,b})\\
&\leq \sum_{n=1}^{m}F^{n_1}
(\cF^n_*\nu_n)(\cR^*\cap\cF^{m} C^c_{m-n,b}),
 \end{align*}
 where $n_1$ is a stopping time on $R^*$, such that $F^{n_1}(x)=\cF^{n-m}(x)$, for any $x\in R^*$. We denote the probability measure $$\eta= \frac{\cF^{\tau}_*\nu|_{\tau\leq m}}{\nu(\tau<m)}=\frac{\sum_{n=1}^{m}\cF^n_*\nu_n}{\nu(\tau<m)}. $$ Using the exponential decay for the induced system, see Proposition \ref{exp decay}, as well as the fact that $n_1\geq (b\ln m)^2$, we get
 \begin{align}\label{Cnbc}
 |\frac{1}{\nu(\tau< m)}\sum_{n=1}^{m}F^{n_1}_*(\cF^n_*\nu_n)(\cF^{m} C^c_{m,b}\cap R^*)&-\mu(\cF^{m} C^c_{m,b}\cap R^*)|= |F^{n_1}_*\eta(\cF^{m} C^c_{m,b}\cap R^*)-\mu(\cF^{m} C^c_{m,b}\cap R^*)|\nonumber\\
 &\leq|F^{(b\ln m)^2}_*\eta(R^*\cap \cF^{m} C^c_{m,b})-\mu( R^*)|\nonumber\\
 &\leq C \vartheta^{(b\ln m)^2}\leq
 C_2 m^{-\alpha_0},
 \end{align}
 where we used  (\ref{dpbchi}).
Combining the above facts, we get
\begin{align*}|\nu(A_m)- \mu(\cR^*)|&\leq C_2  m^{-\alpha_0}+|\cF^m_*\nu(\cR^*\cap\cF^{m} C_{m,b})-\mu(R^*\cap \cF^{m}C_{m,b})|+ \sum_{n=m+1}^{\infty}\cF^m_*
\nu_n(\cR^*).
\end{align*}
Note that $\sum_{n=m+1}^{\infty}\cF^m_*
\nu_n(\cR^*)$ is the measure of points in $\cW$ that return to $\cR^*$ at the $\cF^m$-th iteration, but they do not return properly at or before $\cF^m$-th iteration.

Using Proposition \ref{YtowercM}, we get
$$\sum_{n=m+1}^{\infty}\cF^m_*
\nu_n(\cR^*)= \sum_{n=m+1}^{\infty}
\nu_n(\cF^{-m}\cR^*)\leq \sum_{n=m+1}^{\infty}
\nu_n(\cM).$$
Thus we have shown that $\sum_{n=m+1}^{\infty}\cF^m_*
\nu_n(\cR^*)$ is dominated by $\sum_{n=m+1}^{\infty}\cF^m_*
\nu_n(\cR^*\cap \cF^{m}C_{m,b}))$.

Combining the above facts, we get item (1) as claimed:
\begin{align*}
|\nu(A_m)- \mu(\cR^*)|
&\leq \sum_{n=m+1}^{\infty}\cF^m_*
\nu_n(\cR^*\cap \cF^{m}C_{m,b}))\\
&+|\cF^m_*\nu(\cR^*\cap\cF^{m} C_{m,b})-\mu(R^*\cap \cF^{m}C_{m,b})|+C_3  m^{-\alpha_0}\leq C(1+\|g\|_{\infty}) m^{1-\alpha_0}. \end{align*}

In particular, given two proper generalized standard families $(\cW^i,\nu^i)$,    if $\cW^i \subset M$, for $i=1,2$, then using (\textbf{H3}), we get
\begin{align*}|\nu^i(A_m(\cW^i))- \mu(\cR^*\cap\cF^{-m} C_{m,b}^c)|&\leq \sum_{n=m+1}^{\infty}\cF^m_*
\nu^i_n(\cR^*\cap \cF^{m}(C_{m,b}\cap M))\\
&+\cF^m_*\nu^i(\cR^*\cap\cF^{m} (C_{m,b}\cap M))+C_3  m^{-\alpha_0}\leq C m^{-\alpha_0}.
\end{align*}
This implies item (2) as claimed:
\begin{align*}
|\nu^1(A_m(\cW^1))- \nu^2(A_m(\cW^2))|&\leq |\nu^1(A_m)- \mu(\cR^*\cap\cF^{-m} C_{m,b}^c)|+|\nu^2(A_m)\\
&- \mu(\cR^*\cap\cF^{-m} C_{m,b}^c)|\leq 2C m^{-\alpha_0}.\end{align*}


Next we consider a  generalized standard family $(\cW,\nu)$ on $M$ with density $g$. And we now consider the generalized standard family $(\cR^*, \mu^*)$, similar arguments as in the proof of item (1) will  show that
$$|\nu(A_m(\cW))-\mu(\cR^*)|\leq C (1+\|g\|_{\infty}) m^{1-\alpha_0}.$$
 If  $(\cW,\nu)$  is a generalized standard family of index zero, then
 for any $n\geq 1$,
 $$ \frac{\nu(A_n)}{\nu(\cR^*)}\geq \mu(\cR^*)-\cO(n^{1-\alpha_0})\geq c,$$ for some uniform constant $c>0$.

\end{proof}

\section{Coupling Lemma for the original system}

In this section, we will prove the Coupling Lemma for the original nonuniformly hyperbolic map, which is  new to our knowledge, since the construction is significantly  different from that for systems with uniformly hyperbolicity. This will enable us to  define the coupling decompositions of probability measures on $\cM$, which will be used to investigate the rate of decay of correlations for the iterations of those measures.\\

\noindent {\textbf{Notation:}}  For simplicity, given a standard family $\cG=(\cW,\nu)$, {\textbf{we define $\cG(A):=\nu(A)$ }}to avoid using too many notations.
\subsection{Renewal properties}
Similar to \cite{Rog}, we denote
$\mathbf{R}_+(\cO(n^{-\alpha_0}))$ (resp. $\mathbf{R}_+(o(n^{-\alpha_0}))$) as the set of all absolutely convergent series $X(z):=\sum_{n=0}^{\infty} x_n z^n$ for $|z|\leq 1$, such that $x_n=\cO(n^{-\alpha_0})$ (resp. $x_n=o(n^{-\alpha_0})$); or in other words, $$\overline{\lim}_{n\to\infty} |x_n| /(n^{-\alpha_0})<\infty,\, \text{(resp.} \lim_{n\to\infty} |x_n| /(n^{-\alpha_0})=0).$$  It was proved in \cite{Rog}~Lemma 2 and  Lemma 3 that both sets are closed under addition, multiplication, and multiplication by a constant of
their generating functions.
\begin{lemma}\label{Roglemma}[\cite{Rog}~Lemma 2 and Lemma 3] If $X(z)\neq 0$ for $|z|\leq 1$, then we define $X(z)^{-1}=\lambda(z):=\sum_{n=0}^{\infty}\lambda_n z^n$. Then
$$X(z), Y(z)\in \mathbf{R}_+(\cO(L(n)n^{-\alpha_0}))\Rightarrow \lambda(z)\in \mathbf{R}_+(\cO(L(n)n^{-\alpha_0})),\,\,\,X(z)Y(z)\in\mathbf{R}_+(\cO(L(n)n^{-\alpha_0})) ;$$$$ X(z), Y(z)\in \mathbf{R}_+(o(L(n)n^{-\alpha_0})) \Rightarrow \lambda(z)\in \mathbf{R}_+(o(L(n)n^{-\alpha_0})),\,\,\,\,\,\,X(z)Y(z)\in \mathbf{R}_+(o(L(n)n^{-\alpha_0})).$$
\end{lemma}

\begin{lemma}\label{P(z)} We assume $\{p_k, k\geq 1\}$ is a discrete probability distribution, with $gcd\{k\,:\, p_k>0\}=1$. Let $P(z)=\sum_{n=1}^{\infty} p_n z^n$, for $|z|\leq 1$.   Then $P(z)=1$ if and only if $z=1$
\end{lemma}
\begin{proof}
Indeed if $|z|<1$, then
$$|P(z)|=|\sum_{n=1}^{\infty} p_n z^n|< \sum_{n=1}^{\infty} p_n=1.$$
We assume $z=e^{i\theta}$, for $\theta\in [0,2\pi]$. Then
$$P(z)=\sum_{n=1}^{\infty} p_n \cos n\theta+i\sum_{n=1}^{\infty} p_n \sin n\theta. $$
Note that $P(z)=1$ if and only if $\sum_{n=1}^{\infty} p_n \sin n\theta=0$ and $\sum_{n=1}^{\infty} p_n \cos n\theta=1$.
However, the second condition can be written as
\beq\label{secondII}\sum_{n=1}^{\infty} p_n \cos n\theta=1=\sum_{n=1}^{\infty} p_n, \eeq
which implies that
$$\sum_{n=1}^{\infty} p_n (1-\cos n\theta)=0.$$
Note that if $z=e^{i\theta} \neq 1$,  then using the fact that $gcd\{k\,:\, p_k>0\}=1$,  there exists $k\geq 1$, such that $|\cos k\theta|<1$ and $p_k>0$, this implies that
$$\sum_{n=1}^{\infty} p_n (1-\cos n\theta)>p_k (1-\cos k\theta)>0.$$
This contradicts the second condition (\ref{secondII}). Thus $P(z)=1$ for $|z|\leq 1$ if and only if $z=1$, as we have claimed.
\end{proof}

Next we will use the above Lemma to prove a result that is rather crucial for our coupling lemma for the original map.
\begin{lemma}\label{lemmarenewal}
Assume there exist $\{\delta_n\}$, $\{\tilde\delta_n\}$,$\{a_n\}$, $\{p_{n}\}$, such that for $n\geq 1$, \beq\label{hatsik1}
\delta_n=a_n+\alpha\sum_{l=1}^{n}p_{l} \delta_{n-l},\,\,\,\,\,\,\tilde\delta_n=\alpha\sum_{l=1}^{n}p_{l} \tilde\delta_{n-l},\eeq
where $\alpha\in (0,1)$; and $gcd\{k\,:\, p_k>0\}=1$.  For any complex number $z\in \mathbb{C}$, with $|z|\leq 1$, we define
$$\delta (z):=\delta _0+\sum_{n=1}^{\infty} \delta _n  z^n,\,\,\,\,\,\,P(z)=\sum_{n=1}^{\infty}p_n z^n,\,\,\,\,A(z)=a _0+\sum_{n=1}^{\infty} a _n z^n,$$
where $ a_0=\delta_0$.
We assume that  $P(z)\in\mathbf{R}_+(\cO(L(n)n^{-1-\alpha_0}))$, $A (z)\in \mathbf{R}_+(\cO(L(n)n^{-\alpha_0}))$, with $P(1)=1$ and $A(1)=1$.
 Then there exists a constant $c_1>0$, such that for any $n\geq 1$,
$$ \delta_n\leq c_1 L(n)n^{-\alpha_0},\,\,\,\,\,\,\text{ and }\,\,\,\,\,\tilde\delta_n\leq c_1 L(n)n^{-1-\alpha_0}.$$
\end{lemma}
\begin{proof} By assumption, for any $n\geq 1$, we have
$\delta _n=a_n +\alpha\sum_{l=1}^{n}p_{l}  \delta _{n-l}$,
where $\alpha\in (0,1)$.

It follows from Lemma \ref{P(z)} that $P(z)=1$ if and only if $z=1$.  Using  the definition of $\delta(z)$, $P(z)$ and $A(z)$, we can check that (\ref{hatsik1}) implies that $\delta(z)=A(z)+(1+\alpha\delta(z)P(z))$, which implies that $$\delta (z)(1-\alpha P(z))=A (z)+1.$$
Thus $0<1-\alpha P(z)\leq 1-\alpha$ for $|z|\leq 1$, since $\alpha\in (0,1)$. Moreover, $P(z)\in \mathbf{R}_+(\cO(n^{-1-\alpha_0}))$ implies that $1-\alpha P(z)\in \mathbf{R}_+(\cO(n^{-1-\alpha_0}))$. It follows from \cite{Rog}~Lemma 3 that
$$(1-\alpha P(z))^{-1}\in \mathbf{R}_+(\cO(n^{-1-\alpha_0})).$$

Again it follows from Lemma \ref{Roglemma} that
\begin{align}\label{deltaz1}
\delta(z)&=(1-\alpha P(z))^{-1}(A(z)+1)\in \mathbf{R}_+(\cO(n^{-\alpha_0})).
\end{align}
This proved the upper bound $\delta_n\leq C n^{-\alpha_0}$, for some constant $C>0$.

Next we consider the sequence $\tilde\delta_n$, using $\tilde\delta_n=\alpha\sum_{l=1}^{n}p_{l} \tilde\delta_{n-l}$, we get $\tilde\delta(z)=1+\alpha\tilde\delta(z)P(z)$, which implies that $\tilde\delta (z)(1-\alpha P(z))=1$, or equivalently, $\tilde\delta(z)=(1-\alpha P(z))^{-1}$. It follows from Lemma \ref{Roglemma} that  \begin{align*}
\tilde\delta(z)&=(1-\alpha P(z))^{-1}\in \mathbf{R}_+(\cO(n^{-1-\alpha_0})).
\end{align*}

Note that $\delta_n\geq a_n$,   we obtain
\begin{align*}
\delta_n=a_n +\alpha\sum_{k=1}^{n-1} p_k \delta_{n-k}\geq a_n+\alpha\sum_{k=1}^{n-1} p_k a_{n-k}.
\end{align*}

\end{proof}

\subsection{Statement of the Coupling Lemma}
  We now state the coupling lemma for the original nonuniformly hyperbolic system $(\cF,\mu)$.
First, we remind that $\cG^u:=(\cW^u,\mu)$ is the standard family generated by the SRB measure $\mu$.

\begin{lemma}\label{coupling1} Let $\cG^i=(\cW^i,\nu^i)$, $i=1,2$,  be  two  proper standard families,  with  $d\nu^i=g^i d\mu$, where $g^i\in\cH^+(\gamma_0)$ is a probability density function.  \\
(\textbf{C1})  There exist   $C_1>0, C>0$, such that for any $n\geq 1$, there is a  decomposition $$\cG^i=\sum_{k=1}^{n}(\hat\cW^i_k,\nu^i_k)+(\bar\cW^i_n,\bar\nu^i_n),$$ for $i=1,2$, with the following properties for any $k=1,\cdots, n$:
 \begin{itemize}\item[(i)]   Both $(\hat\cW^i_k,\nu^i_k)$  are  generalized standard family with index $k$;
\item[(ii)] For any measurable function $f$ that is constant on each $W^s\in \Gamma^s$,  we have $\cF^{k}_*\nu^1_k(f)=\cF^{k}_*\nu^2_k(f)$;
\item[(iii)] For any bounded function $f\in L_{\infty}(\cM,\mu)$, the uncoupled measure satisfies:
$$|\bar\nu^i_n(f)|\leq C\|f\|_{\infty} (\nu^i(R>n)+\mu(\supp(g^i)\circ \cap(\cF^{-n}\supp(f))\cap(R\leq n)\cap C_{n,b}))),$$
\end{itemize}
\noindent(\textbf{C2}) Moreover, there exists $C_1=C_1(\gamma_0)>0$, such that  the portion of measure coupled at $n$-th step satisfies:
$$\cF^n_*\nu^i_n(\cR^*)\leq C_1L(n) n^{-{\alpha_0}}.$$
\noindent(\textbf{C3}) The uncoupled measure $\bar\nu^i_n(\cM)$ is dominated by $\bar\nu^i_n(C_{n,b})$ and $\nu^i(R>n)$:
\begin{align*}
\bar\nu^i_n(\cM)&=\nu^i(R>n)+\bar\nu^i_n((R\leq n)\cap C_{n,b})+\cO(L(n)n^{-\alpha_0})\\
&\leq \nu^i(R>n)+\mu(\cF^{-n}(\supp f) \cap(\supp g)\cap C_{n,b})+ \cO(L(n)n^{-\alpha_0}).
\end{align*}
 \end{lemma}
The proof of this Coupling Lemma can be found in Subsection \ref{proofcoupling}, after we describe in detail the coupling procedure for measures that  properly returned to $\cR^*$ in this subsection. \\

 We begin by considering a special situation, i.e. $\cG^i=(\cW^i,\nu^i)$, $i=1,2$, are two generalized standard families with index $0$.
We first prove a lemma describing the coupling process which will be used in our  proof of Lemma \ref{coupling1}.

 \begin{lemma}\label{couplescheme} Assume that for  $i=1,2$,  $\cG^i=(\cW^i,\nu^i)$ are generalized standard families with index $0$, and $$\min\{\nu^1(\Gamma^s), \nu^2(\Gamma^s)\}>0.$$  Then there exist a generalized standard family $\cE^i=(\cW^i,\eta^i)$ with index $0$,  and $$\cK^i:=\cG^i-\cE^i=(\cW^i,\xi^i),$$   with the following properties:\\
(\textbf{a})  $\cE^1$ and $\cE^2$ are coupled in the following sense:

(\textbf{a1}) For any bounded function $f$ that is constant on each $W^s\in\Gamma^s$, we have $\cE^1(f)=\cE^2(f)$;

(\textbf{a2}) The total coupled measure satisfies $$\cE^i(\Gamma^s):=c_0\min\{\cG^1(\Gamma^s), \cG^2(\Gamma^s)\},$$ where $c_{0}\in [(1-\ba)/2,1-\ba]$, $\ba$ was defined in Lemma \ref{defnN}, and we denote $d^i=\cE^i(\Gamma^s)/\cG^i(\Gamma^s)$.\\
(\textbf{b}) The remaining uncoupled family $\cK^i$ is a pseudo-generalized family with index $0$, it has the property that $\cF^{n}(\cK^i|_{\Gamma^s_n})$ becomes a generalized standard family of index $0$, for any $n\geq 1$.\\
(\textbf{c}) For any measurable collection of unstable manifolds $B\subset \Gamma^s$, the remaining uncoupled measure can be calculated as:
$$|\cG^1(B)-\cG^2(B)|=|\cK^1(B)-\cK^2(B)|\,\,\,\,\,\,\text{ \em{and} }\,\,\,\, \,\,\,
\cK^i(\Gamma^s)=(1-d^i)\cG^i(\Gamma^s).$$
\noindent(\textbf{d}) There exists $c_1>0$, such that for any $s$-subset  $U\subset\cR^*$,
$\cK^i(U)\geq c_1 \cG^i(U)$.
\end{lemma}
\begin{proof}
Since for $i=1,2$, $\cG^i$ is a generalized standard family of index $0$,  by definition, it has  dynamically H\"{o}lder density function $g^i$ supported on $\cW^i$ which is an u-subset of $\cR^*$. Then we denote $$\cG^i=\{(W_{\alpha},\nu_{\alpha})\,:\,\alpha\in \cA^i, \lambda^i\},$$ such that  for any measurable set $A\subset \cM$,
 $$\nu^i(A\cap\Gamma^s)=\int_{\alpha\in \cA^i}\int_{W_{\alpha}\cap A} g_{\alpha}\,d\mu_{\alpha}\,\lambda^i(d\alpha),$$
 where $$g^i_{\alpha}=g^i/\mu_{\alpha}(g^i)$$ with $$d\nu^i_{\alpha}=g^i_{\alpha}\,d\mu_{\alpha},$$ and $$\lambda^i(d\alpha)=\mu_{\alpha}(g^i)\lambda^u(d\alpha).$$ Clearly, $$\lambda^i(\cA^i)=\nu^i(\Gamma^s)=\mu(g^i).$$

  For any $\alpha\in \cA^i$, as the standard pair $(W_{\alpha},\nu^i_{\alpha})$   properly crosses  $\Gamma^s$, by Lemma \ref{densitybd}, we know that the density function satisfies $$g_{\alpha}\geq e^{-\eps_{\bd}}.$$ Thus one could match at least a positive portion of measures from both families along stable manifolds in $\Gamma^s$. Now we follow the coupling scheme as described by Chernov and Markarian in the book \cite{CM} -- page 200-202, by choosing a  function $\rho^i_{\alpha}\in \cH(\gamma_0)$ on $W_{\alpha}$, with the following properties.  We take  $c_0\in [(1-\ba)/2,(1-\ba)]$, as  described in \cite{CM} (which was chosen to be piecewise linear), thus we have the flexibility to choose  a  function  $\rho^i_{\alpha}\in \cH(\gamma_0)$, such that  $\|\rho^i_{\alpha}\|_{\gamma_0}\leq 1$, and $$\rho^i_{\alpha}\in ( g^i_{\alpha}(1-\ba)/2, (1-\ba)g^i_{\alpha}),$$ for any $\alpha\in \cA^i$, with $\ba$ defined as in Lemma \ref{defnN}, such that
  \beq\label{rhoialpha}\int_{\alpha\in \cA^i}\mu_{\alpha}(\rho^i_{\alpha})\,d\lambda^i(\alpha)=c_0\min\{\nu^1(\Gamma^s), \nu^2(\Gamma^s)\},\quad \forall i=1,2.\eeq

  Now we are ready to define the coupled families $\cE^i$ for $i=1,2$.
  Equation    (\ref{rhoialpha}) implies that for any $\alpha\in \cA^i$, one can define a standard pair corresponding to the measure defined by the probability density function $\rho^i_{\alpha}/\mu_{\alpha}(\rho^i_{\alpha})$ on $W_{\alpha}$,  denoted as $(W_{\alpha},\eta_{\alpha})$. Let  $$\cE^i:=(\cW^i,\eta^i)=((W_{\alpha}, \eta_{\alpha})|_{\Gamma^s}, \mu_{\alpha}(\rho^i_{\alpha})\lambda^i(d\alpha))$$ be the corresponding generalized standard family with index zero. One could  choose $\rho^i_{\alpha}$ and $c_0$ carefully to make sure that for  any measurable collection $A$ of stable manifolds  in $\Gamma^s$,  we have $$\eta^1(A)=\eta^2(A),$$
  and
  $$\eta^i( \Gamma^s\cap\cW^i)=c_0\min\{\nu^1(\Gamma^s), \nu^2(\Gamma^s)\}.$$  Let $d^i=\eta^i(\Gamma^s)/\nu^i(\Gamma^s)$. This verifies items \textbf{(a1)-(a2)}.

 Next, we define the remaining uncoupled family $\cK^i$ by subtracting the density of  $\eta^i$ from $\nu^i$. More precisely, for any $\alpha\in \cA^i$, we subtract  $\rho^i_{\alpha}$ from the density function  $g_{\alpha}$. The remaining family in $\cG^i$ is denoted as $\cK^i$, which may not  have the required regularity of being a generalized standard family. We apply  Lemma \ref{defnN}, which states that $F\cK^i$ is already a standard family. It follows that restricted on $\Gamma^s_n$, the family  $\cF^{n}(\cK^i|_{\Gamma^s_n})$ becomes a  generalized standard family with index $0$, for any $n\geq 1$.
 Note that (\textbf{a2}) implies that, for any measurable collection of stable manifolds $A\subset \Gamma^s$,
 \beq\label{coupleA}\nu^1(A)-\nu^2(A)=\eta^1(A)-\eta^2(A)+\xi^1(A)-\xi^2(A)=\xi^1(A)-\xi^2(A).\eeq
Combining above facts, we get $$\cG^i=\cE^i+\cK^i$$ satisfying \textbf{(a)-(c)}, as claimed.

Item (d) follows from the distortion bound for the density of $\xi^i$ and $\eta^i$, as well as the fact that $c_0\leq 1-\ba$.
\end{proof}

Next we prove will another lemma that will be used in the estimations of measures that proper return to $\cR^*$ at any step $n\geq 1$.

\subsection{Proof of the Coupling Lemma.}\label{proofcoupling}

 \noindent{\textbf{Step 0. Properly return to $\cR^*$ after $N$ iterations.}}\\

Since both families $\cG^i$ are proper standard families, by Proposition \ref{properc}, $$\mu(\cR^*)+Cn^{1-\alpha_0}\geq \nu^i(A_n)\geq \mu(\cR^*)-Cn^{1-\alpha_0},$$ for a uniform constant $C>0$.  Thus there exists $N\geq 1$,  such that $CN^{1-\alpha_0}\leq \mu(\cR^*)/2$; therefore they have  $$\hat\delta_0:=3\mu(\cR^*)/2\geq \nu^i(A_N)\geq \delta_0:=\mu(\cR^*)/2>0$$  portion of measure which properly return to $\cR^*$ after $N$ -iterations.
For $i=1,2$, let $\cG^i_0=(\cW^i_0,\nu^i_0):=\cF^N\cG^i$.   By Lemma \ref{extralemma2}, there exists a sequence of $\cF-$ generalized standard families $\{(\cW^i_{0,n},\nu^i_{0,n}), n\geq 0\}$, such that $$\cG^i_0=\sum_{m=0}^{\infty} \cG^i_{0,m}=\sum_{m=0}^{\infty}(\cW^i_{0,m},\nu^i_{0,m}),$$ where $\cG^i_{0,m}$ is a generalized standard family with index $m$, such that $\cF^{m}\cW^i_{0,m}$ properly returns $\cR^*$. In addition, $$\cG^i_{0,n}(\cM)\leq \cG^i_n(\tau=n)\leq C n^{-\alpha_0},$$ for some uniform constant $C>0$. Also for convenience, we denote $A_0=\cR^*$, and for $n\geq 0$, \beq\label{defna0i}
a_n^i=\cG^i_{0,n}(\cM).\eeq
Note that the only reason that we decide to start from two proper family is to guarantee the existence of a uniform $N\geq 1$. In the coupling process, we will never use the proper property anymore. Indeed we will couple certain amount of measures at each step,  not another $N$-th return. This is different from the coupling procedure used to prove exponential rates, see \cite{CM}.\\
 \\
 \noindent{\textbf{Step 1. Capture and then coupling along $\Gamma^s$ for $\cG^1_0$ and $\cG^2_0$.}}\\

Note that  both $\cG_0^1$ and $\cG_0^2$ have sets of positive measure  properly return to $\cR^*$: \beq\label{delta000}\hat\delta_0\geq \cG^i_0(A_0)>\delta_0.\eeq  We can apply Lemma \ref{couplescheme} to couple at least $c_0$ portion of the measure $\min\{\cG^1_0(A_0), \cG^2_0(A_0)\}$ from both families.
We define the index of $\cG^i_0$ that achieves the minimum of $\cG^1_0(A_0)$ and $ \cG^2_0(A_0)$: $$s_0:=\{i\in \{1,2\}\,|\, \cG^i_0(A_0)=\min\{\cG^1_0(A_0), \cG^2_0(A_0)\}\}$$
 More precisely, we put $\hat\cK^i_0=\cG^i_{0,0}.$
Since the family $\cG^i_{0,0}$ is a generalized standard family with index $0$,   we can apply Lemma \ref{couplescheme}, to get a decomposition $$\cG^i_{0,0}= \cE^i_{0}+ \cK^i_{0},$$  with $$\cE^i_{0}=(\cW^i_{0,0}, \eta^i_{0}),\,\,\,\,\,\,\text{ and }\,\,\,\,\,\,\cK^i_{0}=(\cW^i_{0, 0 }, \xi^i_{0 }),$$ where $\cE^i_0$ is a generalized standard family of index $0$.   Note that  $\cE^1_{0}$ and $\cE^2_{0}$  are coupled along stable manifolds in $\Gamma^s$; more precisely,
for any $f\in\cH^-(\gamma_0)$ that is constant on each $W^s\in\Gamma^s$, $\eta^1_{0}(f)=\eta^2_{0 }(f).$ Moreover, $$c_0\delta_0\leq \cE^i_{0}(A_0)= c_{0}\cG^{s_0}_0(A_0)\leq c_0\hat\delta_0,\,\,\,\,\,\,\cK^i_{0}(A_0)=(1-c_0) \cG^{s_0}_0(A_0)\leq (1-c_0)\hat\delta_0.$$ Here we choose  the total coupled amount of measure  as  $e_0\geq c_0\delta_0,$ where $c_{0}$ was chosen in $ [(1-\ba)/2,1-\ba]$, according to Lemma \ref{couplescheme}.

After the first step of coupling, we know that the remaining family has measure $1- \cE^i_{0}(A_0)$, thus they can be  denoted as:
$$\cG^i_0=\cE^i_0+\check\cG^i_1,\,\,\,\,\,\,\text{ with } \check\cG^i_1=(\check\cW^i_1,\check\nu^i_1):=\sum_{m\geq 1}\cG^i_{0,m}+\cK^i_0.$$
Thus the total uncoupled amount of measure is
$$\check\cG^{i}_1(\cM)=\sum_{m=1}^{\infty} \cG^i_{0,m}(\cM)+(1-c_0)\cG^{s_0}_0(A_0)\leq \sum_{m=1}^{\infty} a_m^i+(1-c_0)\hat\delta_0.$$

 \noindent{\textbf{Step 2. Release and recapture.}}\\

By Proposition \ref{properc}, Lemma \ref{Tmerateo} and (\ref{nuAnpn}), we know that
\begin{align}\label{xi1}
\cK^i_0(\cR_{1}\cap\cA_0)&=\cK^i_0(\cR_{1}|\cA_0)\cK^i_0(\cA_0)\nonumber\\
&\geq (1-\eps_p) p_1 (1-c_0) \nu^{s_0}_0(A_0)\geq  (1-\eps_p) p_1 (1-c_0)\delta_0.\end{align}
 Moreover,
 \begin{align}\label{xi1U}
\cK^i_0(\cR_{1}\cap\cA_0)=\cK^i_0(\cR_{1}|\cA_0)(1-c_0)\nu^{s_0}_0(A_0)\leq  (1+\eps_p) p_1 (1-c_0)\hat\delta_0.\end{align} This implies that
\begin{align}\label{delta1}a_1^i+\hat\lambda p_1\hat\delta_0\geq \check\cG^i_1(A_{1})&= \cG^i_0(A_1)-\cE^i_0(\cR_{1}\cap\cA_0)\geq a_1^i+\lambda p_1\delta_0,\end{align}
  where $a^i_1=\cG^i_{0,1}=\cG^i_0(\tau=1)$, and we define 	
  \beq\label{defnlamnbdahatl}\lambda=(1-c_0)(1-\eps_p)\,\,\,\,\,\,\text{ and }\,\,\,\,\,\,\,\hat\lambda=(1-c_0)(1+\eps_p).\eeq
  Combining with Lemma \ref{densitybd} and Lemma \ref{defnN}, we know that $0<\lambda< \hat\lambda<1$.

 Thus we can define $\cG^i_1:=\cF\check\cG^i_1:=(\cW^i_1, \nu^i_1)$, with the property that $$\hat\delta^i_1:=a^i_1+\hat\lambda p_1\hat\delta_0\geq \cG^i_1(A_0)= \check\cG^i_1(A_{1})>\delta_1^i:=a_1^i+\delta_0 p_1\lambda\geq  \delta_1:=\delta_0 p_1\lambda.$$ Thus we are in the same stage as of step 1. Then we repeat word by word of the coupling process to get the second decomposition.

We put $$\hat\cK^i_1=(\hat\cW^i_1,\hat\nu^i_1)=\cG^i_1|_{A_0},\,\,\,\,\,\,\,\bar\cK^i_1=(\bar\cW^i_1,\bar\nu^i_1)=\cG^i_1-\hat\cK^i_1.$$
Then  the total uncoupled portion after the first step can be represented as:
 \beq\label{cGi100}\cF(\cG^i_0- \cE^i_0)= \hat\cK^i_1+\bar\cK^i_1.\eeq

First,   $\cG^i_1(A_0)>\delta_1$ is   the total measure of points in $\cG^i_1$ that have returned to $\cR^*$ properly, for $i=1,2$. We define $s_1$, such that $$s_1:=\{i\in \{1,2\}\,|\, \cG^{i}_1(A_0)=\min\{\cG^1_1(A_0), \cG^2_1(A_0)\}\}.$$
We can apply Lemma \ref{couplescheme}, to couple
$$c_0 \cG^{s_1}_1(A_0)\geq c_0\delta_1 $$ amount of measure from both families. More precisely, we  get a decomposition
$$\hat\cK^i_1= \cE^i_{1}+\cK^i_{1},\,\,\,\,\,\,\text{ with }\,\,\,\,\,\cE^i_{1}=(\hcW^i_{1}, \eta^i_{1}),\,\,\,\,\,\cK^i_{1}=(\hcW^i_{1}, \xi^i_{1}).$$   Note that  $\cE^1_{1}$ and $\cE^2_{1}$  are coupled along stable manifolds in $\Gamma^s$, such that for any $f\in\cH^-(\gamma_0)$ that is constant on each $W^s\in\Gamma^s$, $\eta^1_{0}(f)=\eta^2_{0 }(f).$ Moreover, $$c_0\hat\delta^i_0\geq \cE^i_{1}(A_0)=e_1:= c_{0}\cG^{s_1}_1(A_0)\geq c_0\delta_1,\,\,\,\,\,\,\cK^i_{1}(A_0)=(1-c_0) \cG^{s_1}_1(A_0)\leq (1-c_0) \cG^{i}_1(A_0)\leq (1-c_0)\hat\delta^i_1.$$
After the second step of coupling, the remaining family is denoted as
 \beq\label{cGi1001}\check\cG^i_2:=\cF(\cG^i_0- \cE^i_0)-\cE^i_{1}= \cK^i_1+\bar\cK^i_1.\eeq

Note that after this step, the measure from both families that have properly return to $\cR^*$ satisfies:
$$\cG^i_0(A_{1})=\cE^i_0(A_{1})+\cG^i_0(\tau=1) +\check\cG^i_0(\cR_{1}\cap\cA_0).$$
According to our coupling algorithm described as in Lemma \ref{couplescheme}, we have coupled at least $$\cE^i_0(A_{1})+c_0\cG^{s_1}_0(\tau=1) +c_0\check\cG^{s_1}_0(\cR_{1}\cap\cA_0) \geq \cE^i_0(A_{1})+c_0\delta_1$$ amount of measure from both families, with remaining amount of measure
$$(1-c_0)\cG^{s_1}_1(\cA_0)= (1-c_0)\nu^{s_1}_0(\tau=1) +(1-c_0)^2\nu^{s_1}_0(\cR_{1}\cap\cA_0)$$ left uncoupled in $\cG^{s_1}_1(A_{0})$.

Thus we can denote the remaining families  as:
$\check\cG^i_2=(\check\cW^i_2,\check\nu^i_2)$. Similar to (\ref{xi1}), we get
\begin{align}\label{xi2}
\cK^i_0(\cR_{2}\cap\cA_0)&=\cK^i_0(\cR_{2}|\cA_0) \cK^i_0(A_0)\nonumber\\&\geq  (1-\eps_p)p_2 (1-c_0) \nu^{s_0}_0(A_0) \geq \lambda p_2\delta.\end{align}

Moreover, similar to (\ref{xi1}), we get
\begin{align}\label{xi21}
\cK^i_1(\cR_{1}\cap\cA_0)&
\geq (1-\eps_p)p_1(1-c_0)\nu^{s_0}_1(A_0)\geq \lambda p_1\delta_1.\end{align}

Combining with (\ref{nuAnpn}) we know that
\begin{align*}
\cG^i_0(A_{2})&=\cG^i_0(\tau=2)+\cG^i_0(\cR_{2}\cap\cA_0)+\cG^i_0(\cF^{-1}\cR_{1}\cap\cA_1) \\
&=\nu^i_0(\tau=2)+\cE^i_0(\cR_{2}\cap\cA_0) +\cK^i_0(\cR_{2}\cap\cA_0)+\cE^i_1(\cR_{1}\cap \cA_0) +\cK^i_1(\cR_{1}\cap \cA_0)\\
&=\cE^i_0(\cR_{2}\cap\cA_0) +\cE^i_1(\cR_{1}\cap \cA_0) +\nu^i_0(\tau=2)+\cK^i_0(\cR_{2}\cap\cA_0)+\cK^i_1(\cR_{1}\cap \cA_0)\\
&=\cE^i_0(\cR_{2}\cap\cA_0) +\cE^i_1(\cR_{1}\cap \cA_0) +\check\cG^i_2(A_{1}).
\end{align*}
Let $\cG^i_2:=\cF\check\cG^i_2:=(\cW^i_2, \nu^i_2)$.
Using similar estimations as in (\ref{delta1}), we get
\begin{align*} \hat\delta^i_2:&=\cG^i_0(\tau=2) +\hat\lambda(p_2\hat\delta_0+p_1\hat\delta_1)
\geq	  \cG^i_2(A_{0})\geq  \cG^i_0(\tau=2) +\lambda(p_2\delta_0^y+p_1\delta_1^i)\\
&\geq \delta_2^i:=a_2^i+\lambda(p_2\delta_0+p_1\delta^i_1)\geq \delta_2:=\lambda(p_2\delta_0+p_1\delta_1).\end{align*}

 Thus  we can  define $\cG^i_2:=\cF\check\cG^i_2:=(\cW^i_2, \nu^i_2)$, as well as
 $$\hat\cK^i_2=(\hat\cW^i_2,\hat\nu^i_2)=\cG^i_2|_{A_0}\,\,\,\,\,\,\,\bar\cK^i_2=(\bar\cW^i_2,\bar\nu^i_2)=\cG^i_2|_{A_0^c}.$$

We define $s_2$, such that $$s_2:=\{i\in \{1,2\}\,|\, \cG^{i}_2(A_0)=\min\{\cG^1_2(A_0), \cG^2_2(A_0)\}\}.$$
We can apply Lemma \ref{couplescheme}, to couple
$$c_0\hat\delta^i_2\geq e_2=c_0 \cG^{s_2}_2(A_0)\geq c_0\delta_2^i $$ amount of measure from both families.

Combining with (\ref{cGi100}), we  know that the total uncoupled portion at the second step can be represented as:
\beq\label{firststepFtau2''}\cF^{2}\cG^i_0-\cF^{2}\cE^i_0-\cF\cE_1^i=\cG^i_2=\hat\cK^i_{2}+ \bar\cK^i_2.\eeq

Note that after this step, the measure from both families that have properly return to $\cR^*$ satisfies:
$$\cG^i_0(A_{2})=\cE^i_0(A_{2})+\cE^i_1(A_1)+\check G^i_2(A_{1})\geq \cE^i_0(A_{2})+\cE^i_1(A_{1})+\delta_2^i.$$
According to our coupling algorithm described as in Lemma \ref{couplescheme}, we have coupled at least $$\cE^i_0(A_{2})+\cE^i_1(A_{1})+c_0\delta_2$$ amount of measure from both families, with remaining
$(1-c_0)\cG^{i}_2(A_0)\leq (1-c_0) \hat\delta^i_2 $ amount of measure left in $\cG^{i}_2(A_{0})$.\\
\\

    \noindent{\textbf{Step 3. Coupling at the repeated  returns.}}\\

Next we consider higher iterations by induction, with $n\geq 2$.
We  assume  for $n-1$,
 \beq\label{k-1inductG}
\cF^{{n-1}}\cG^i_0=\cF^{{n-1}}\cE^i_0+\cF^{n-2}\cE_1^i+\cdots+\cE^i_{n-1}+\check\cG^i_n. \eeq
Here for any $k\leq n-1$, $\cE^i_{k}=(\hcW^i_k,\eta^i_k)$ is a generalized standard family with index $k$, such that $\cE^1_k$ and $\cE^2_k$ are coupled along stable manifolds $\Gamma^s$, and satisfying
$$\eta^1_k(A)=\eta^2_k(A),$$ for any measurable collection $A$ of stable manifolds in $\Gamma^s$. Let $\cG^i_k=\cF\check\cG^i_k$.  Moreover, $\check\cG^i_k=(\check\cW^i_k,\check\nu^i_k)$ is the uncoupled remaining family that have been conditioned to carry a probability measure.
We define $s_k$, such that $$s_k:=\{i\in \{1,2\}\,|\, \cG^{i}_k(A_0)=\min\{\cG^1_k(A_0), \cG^2_k(A_0)\}\}.$$
 Moreover, $$\cE^i_{k}(A_0)=e_k:= c_{0}\cG^{s_k}_k(A_0),\,\,\,\,\,\,\cK^i_{k}(A_0)=(1-c_0) \cG^{s_k}_k(A_0).$$

 The coupled measure in $\cF^{n-1}\cG^i_0(A_0)$ at the $n-1$th step satisfies
\begin{align*}
c_0\hat\delta_k^i\geq \cE^i_{k}(A_0)\geq c_0\delta^i_{k}\geq c_0\delta_{k},\end{align*}
where $$\delta^i_{k}:=a^i_k+\lambda\sum_{j=0}^{k-1}\delta^i_j p_{k-j}, \,\,\,\,\,\,\,\delta_{k}:=\lambda\sum_{j=0}^{k-1}\delta_j p_{k-j},\,\,\,\,\,\,\,\hat\delta_k^i=a^i_k+\hat\lambda\sum_{j=0}^{k-1}\hat\delta^i_j p_{k-j}.$$
The remaining measure in $\cF^{k}\cG^i_0(A_0)$ at the $k$th step is at most
\begin{align*}
\cK^i_{k}(A_0)\leq (1-c_0)\hat\delta^i_{k}.\end{align*}

Using (\ref{properagain}) and similar to (\ref{xi1}), we know that
\begin{align}\label{cGn} \cG^i_{0}(A_n)&=a^i_n+\nu^i_0(\cR_{n}\cap\cA_0)+\cdots+ \nu^i_{0}(\cF^{-(n-1)}\cR_{1}\cap\cA_{n-1})\nonumber\\
&\geq a^i_n+\cE^i_0(\cR_{n}\cap\cA_0)+\cdots+ \cE^i_{0}(\cF^{-(n-1)}\cR_{1}\cap\cA_{n-1})+\lambda (p_n\delta_0+\cdots+p_1\delta_{n-1}).\end{align}
Moreover,
\begin{align}\label{cGnU} \cG^i_{0}(A_n)&=a^i_n+\cG^i_0(\cR_{n}\cap\cA_0)+\cdots+ \cG^i_{0}(\cF^{-(n-1)}\cR_{1}\cap\cA_{n-1})\nonumber\\
&\leq a^i_n+\cE^i_0(\cR_{n}\cap\cA_0)+\cdots+ \cE^i_{0}(\cF^{-(n-1)}\cR_{1}\cap\cA_{n-1})+\hat\lambda (p_n\hat\delta^i_0+\cdots+p_1\hat\delta^i_{n-1}).\end{align}
 Let  $\cG^i_n:=\cF\check\cG^i_n:=(\cW^i_n, \nu^i_n)$ and
  $$\delta_n=\lambda (p_n\delta+\cdots+p_1\delta_{n-1}),\,\,\,\,\, \delta^i_n= a^i_n+\lambda (p_n\hat\delta^i_0+\cdots+p_1\delta^i_{n-1}),\,\,\,\,\, \hat\delta^i_n= a^i_n+\hat\lambda (p_n\hat\delta^i_0+\cdots+p_1\hat\delta^i_{n-1}).$$
  Then we get
 $$\hat\delta^i_n\geq \cG^i_n(A_0)\geq \delta^i_n\geq \delta_n.$$
 We define
 $$\hat\cK^i_n=(\hat\cW^i_n,\hat\nu^i_n)=\cG^i_n|_{A_0},\,\,\,\,\,\,\,\bar\cK^i_n=(\bar\cW^i_n,\bar\nu^i_n)=\cG^i_n|_{A_0^c}.$$

We define $s_n$, such that $$s_n:=\{i\in \{1,2\}\,|\, \cG^{i}_n(A_0)=\min\{\cG^1_n(A_0), \cG^2_n(A_0)\}\}.$$
Now (\ref{cGn}) implies that we can couple $c_0\cG^{s_n}_n(A_0) $ amount of the measure from $\cG^1_n$ and $\cG^2_n$.
using Lemma \ref{couplescheme}. More precisely, we  get a decomposition
$$\hat\cK^i_{n}= \cE^i_{{n}}+\cK^i_{n},$$  with $$\cE^i_{n}=(\hcW^i_{n}, \eta^i_{n}),\,\,\,\,\,\cK^i_{n}=(\hcW^i_{n}, \xi^i_{n}).$$   Note that  $\cE^1_{n}$ and $\cE^2_{n}$  are coupled along stable manifolds in $\Gamma^s$ by $e_n$ amount of measure. Moreover,  for any measurable collection of stable manifolds we have:
$A\subset \Gamma^s$, $\eta^1_{n}(A)=\eta^2_{n}(A)$.
After the second step of coupling, we denote the remaining conditioned families  as:
$\check\cG^i_{n+1}=(\check\cW^i_{n+1},\check\nu^i_{n+1})$.
Then we get
\beq\label{ninductG}
\cF^{n}\cG^i_0=\cF^{n}\cE^i_0+\cF^{n-1}\cE_1^i+\cdots+\cE^i_{n}+\check\cG^i_{n+1}. \eeq

The coupled measure in $\cF^{n}\cG^i_0(A_0)$ at the $n$th step satisfies
\begin{align*}
c_0\hat\delta_n^i\geq \cE^i_{n}(A_0)\geq c_0\delta^i_{n}\geq c_0\delta_{n}.\end{align*}

The remaining uncoupled measure in $\cF^{n}\cG^i_0(A_0)$ is at most
\begin{align*}(1-c_0)\cF^{n}\cG^i_0(A_0) &\leq (1-c_0)\hat\delta^i_n.\end{align*}

Inductively, one can show that $\cG^i_0$ has a decomposition into $\{\cE^i_k, k\geq 0\}$, i.e we can also denote this decomposition as  $$\cF^n\cG^i_0=\sum_{k=0}^{n}\cF^{n-k}\cE^i_k+\check\cG^i_{n+1}.$$
 Moreover, the leftover at $n$-th step, $\check\cG^i_{n+1}$ can be represented as
$$\check\cG^i_{n+1}=\sum_{k=n+1}^{\infty}\cF^{n-k}\cE^i_k.$$
We apply Lemma \ref{lemmarenewal} to the sequences $\{\delta_n\}$,  $\{\delta^i_n\}$ and $\{\hat\delta^i_n\}$, using the fact that $a_n^i=\cO(L(n)n^{-\alpha_0})$ and $p_n=\cO(L(n)n^{-1-\alpha_0})$ to get
 $$\delta_n\leq C L(n)n^{-1-\alpha_0},\,\,\,\,\,\delta^i_n\leq C L(n)n^{-\alpha_0},\,\,\,\,\,\text{ and } \,\,\,\,\,\hat\delta_n\leq CL(n) n^{-\alpha_0}.$$
 This implies that the total measure coupled at the $n$-th step is
\beq\label{coupledatn}\cE^i_n(A_0)\leq  c_0\hat\delta^i_n\leq c_0 C L(n) n^{-\alpha_0}.\eeq
 This also implies that the remaining total measure after $n$-th step is:
 \beq\label{remain}
 \check\cG^i_{n+1}(\cM)\leq \sum_{k=n+1}^{\infty}\cF^{n-k}\cE^i_k(\cM)=\sum_{k=n+1}^{\infty}\cF^{n-k}\cE^i_k(A_0)\leq c_0 \sum_{k=n+1}^{\infty}\hat\delta_k \leq c_0\alpha_0 Cn^{1-\alpha_0}.\eeq

 Note that for any $n\geq 0$, $\cE^i_{n}=(\hcW^i_n,\eta^i_n)$ is a generalized standard family with index $0$, such that $\cE^1_n$ and $\cE^2_n$ are coupled along stable manifolds $\Gamma^s$, and satisfying
$$\cE^1_n(A)=\cE^2_n(A),$$ for any measurable collection $A$ of stable manifolds in $\Gamma^s$.

Thus for any measurable function $f$ that is constant on each $W^s\in \Gamma^s$, we have $$\cE^1_n(f)=\cE_n^2(f).$$  This verifies  items (\textbf{C1})(i)-(ii) in the Coupling Lemma \ref{coupling1}.

\bigskip

\noindent\textbf{Step 4. Rearrange the coupled measures by the real iteration time under $\cF$.}\\

Now we rearrange the above coupled and uncoupled families according to the real iteration time under $\cF$. Then
we have shown that  for any $n\geq 1$, there is a  decomposition
$$\cG^i_0=\bar\cE^i_0+\bar\cE_1^i+\bar\cE^i_2+\cdots+\bar\cE^i_{n}+\bar\cG^i_n,$$ for $i=1,2$, where $\bar\cG^i_n=(\bar\cW^i_n,\bar\nu^i_n)=\cF^{-n}\check\cG^i_{n+1}$ is defined as the remaining uncoupled family after step $n$, and $\bar\cE_k^i=\cF^{-k}\cE^i_k$ is the coupled family at time $n$.

Inductively, one can show by the mixing property, that $\cG^i_0$ has a decomposition into $\{\cE^i_k, k\geq 0\}$, and we can also denote this decomposition as  $$\cG^i_0=\sum_{k=0}^{\infty}\cE^i_k.$$
 Moreover, the leftover at $n$-th step, $\bar\cG^i_n$ can be represented as
$$\bar\cG^i_n=\sum_{k=n+1}^{\infty}\bar\cE^i_k=\sum_{k=n+1}^{\infty}(\cW^i_k,\eta^i_n).$$

 Note that for any $n\geq 0$, $\bar\cE^i_{n}=(\cW^i_n,\eta^i_n)$ is a generalized standard family with index $n$, such that $\cF^n\bar\cE^1_n$ and $\cF^n\bar\cE^2_n$ are coupled along stable manifolds $\Gamma^s$, and satisfying
$$\cF^n_*\eta^1_n(A)=\cF^n_*\eta^2_n(A),$$ for any measurable collection $A$ of stable manifolds in $\Gamma^s$.
Thus for any measurable function $f$ that is constant on each $W^s\in \Gamma^s$, we have $$\cF^{n}\eta^1_n(f)=\cF^n\eta_n^2(f).$$
Moreover, (\ref{coupledatn}) implies that the coupled measure at the $n$-th step satisfies:
$$\eta_n^i(\cM)\leq c_0 C n^{-\alpha_0}.$$

This verifies  items (\textbf{C1})(i)-(ii) and (\textbf{C2}) in the Coupling Lemma \ref{coupling1}.\\
\\

Next   we claim that
\beq\label{cnbc}
\bar\nu^i_n(C_{n,b}^c\cap (R\leq n)) \leq C \vartheta^{b\ln n}.
\eeq

To see this, note that points in $C_{n,b}^c\cap (R\leq n)$  will
mostly visit cells with small indices  and return to $M$ at least $\psi$ times within $n$ iterations. We prove this claim by considering two cases.\\

\quad(\textbf{a}). Let $I_n$ be all points $x\in C_{n,b}^c\cap (R\leq n)$, such that  the  iterations of $x$ hits $\cR^*$ at most $b\ln n$ times within $n$ iterations.
Then there exists $k\in [1,n-b\ln n]$, such that $\cF^kx\in M$ and the forward trajectory of $\cF^k x$ return to $M\setminus \cR^*$ at least $b\ln n $ consecutive times under $F$.

According to Lemma \ref{extralemma1},  we know that the standard family $(\cW^i,\nu^i_M)$ has a decomposition into $\{(\cW^i_n,\nu^i_n), n\geq 1\}$,  and by (\ref{ctail1}),
$\sum_{m=n}^{\infty}\nu^i_m(M)\leq C \vartheta^n$.
 Thus $\cF^k x$ belongs to the support of $\sum_{m\geq b\ln n} \nu^i_m$, which satisfies:
 \begin{align*}
 \mu_M(C_{n,b}^c\cap (R\leq n)\cap I_n)&\leq \sum_{m=b\ln n}^{\infty}\nu^i_m(M)\leq C\vartheta^{b\ln n}.\end{align*}

\quad(\textbf{b}). Now we consider points in $C_{n,b}^c\cap (R\leq n)\setminus I_n$. Then we claim that  for any $x\in (C_{n,b}^c\cap (R\leq n)\setminus I_n)$, iterations  of $x$ hit $\cR^*$ at least $b\ln n$ times within the $(b\ln n)^2$ returns to $M$, but have not been coupled. This follows  because otherwise there must be an interval of length $b\ln n$   at least $(b\ln n)^2$ returns to $M$ such that iterates of $x$ never hit $\cR^*$, and this contradicts the assumption that  (\textbf{a}) does not hold.  Thus there exists  $k\in [1,n-b\ln n]$, such that $\cF^kx\in \cR^*$ and the forward trajectory of $\cF^k x$ return to $\cR^*$ at least $b\ln n $  times.  Thus $\cF^k x$ belongs to the support of $\sum_{k\geq b\ln n}\nu^i_k$.  Again by by (\ref{ctail1})  $$\sum_{k\geq b\ln n}\nu^i_k(M)<C\vartheta^{b\ln n}.$$
This implies that  $$\bar\nu_n^i(C_{n,b}^c\cap (R\leq n)\setminus I_n)\leq C\vartheta^{b\ln n}.$$

 This claim implies the last statement (\textbf{C3}), namely,   the remaining uncoupled measure at the $n$-th step is:
\begin{align}\label{barnucM}
\bar\nu_n^i(\cM)&=\nu^i(R>n)+\bar\nu^i_n((R\leq n)\cap C_{n,b})+\bar\nu_n^i((R\leq n)\cap C_{n,b}^c)\nonumber\\
&= \nu^i(R>n)+\bar\nu_n^i((R\leq n)\cap C_{n,b})+ \cO(\vartheta^{b\ln n})\leq \|g^i\|_{\infty} L(n)n^{1-\alpha_0}.\end{align}

 From now on, we choose the large constant  $b=b(\gamma_0,\Lambda,\vartheta,{\alpha_0})>\chi+4$ (where $\chi$ was chosen in Lemma \ref{properagain})  such that
\beq\label{dpbchi}
\Lambda^{-\gamma_0 b\ln n}<\mu_M(C_{n,b}\cap M)/n ,\,\,\,\,\,\,\,\,\vartheta^{ b\ln n}<\mu_M(C_{n,b}\cap M)/n,
\eeq
where $\vartheta\in (0,1)$ is given by (\ref{ctail}).
\\

Note that (\ref{barnucM}) implies that for any $f\in L_{\infty}(\cM)$, we have
$$|\bar\nu^i_n(f)|\leq \|f\|_{\infty} \bar\nu^i_n(\cM)\leq C\|f\|_{\infty} L(n)n^{1-\alpha_0}.$$
This finishes the proof of the coupling lemma for  $\cF^N\cG^i$. \\

Now we have to consider the original family $\cG^i$, instead of $\cF^N\cG^i$. By rearranging the index (from $n$ to $n-N$),   we now denote
$$\cG^i=\sum_{k=1}^{n}(\hat\cW_k^i,\nu_k^i)+(\bar\cW^i_n,\bar\nu^i_n),$$ for $i=1,2$. Here $(\hat\cW_k^i,\nu_k^i)$ is the generalized standard family that coupled at $\cF^n$, and $(\bar\cW^i_n,\bar\nu^i_n)$ is defined as the remaining uncoupled family after $n$-th iteration under $\cF$. Note that for $k=1, \cdots, N-1$,  $\hat\cW_k^i$  is empty. Thus we obtain \textbf{(C1)}  and \textbf{(C3)}. As for \textbf{(C2)}, there exists a large constant $C=C(N)$, such that for  $n\geq 1$,
$$\eta^i_n(\cM)\leq C_1 L(n-N) (n-N)^{-\alpha_0}=C_1L(n)n^{-\alpha_0}\cdot \frac{L(n-N)}{L(n)}\cdot (1-\frac{N}{n})^{-\alpha_0}.$$
Since $\lim_{n\to\infty} \frac{L(n-N)}{L(n)}=1$, there exists $C=C(N)>0$, such that $\eta_n(\cM)\leq C L(n)n^{-\alpha_0}$.

Finally, for any $f\in L_{\infty}(\cM,\mu)$, we have
\begin{align*}|\bar\nu^i_n(f)|&\leq \|f\|_{\infty} \bar\nu^i_n(\cM)\\
&\leq C_2\|f\|_{\infty} (\nu^i(R>n)+\nu^i(\supp(g_i)\cap(\cF^{-n}\supp(f))\cap(R\leq n)\cap C_{n,b}))).
\end{align*}
This finishes the proof of the coupling lemma.\\
\\

 \section{Proof of Theorem \ref{main2}.}
 Now it is time to investigate   the rates of decay of correlations using the above Coupling lemma. We first prove a lemma that will be used later.

 We consider two proper standard families
 $\cG^i=(\cW^i,\nu^i)$, for $i=1,2$.
 Here $\nu^i$ has  probability density $g^i=d\nu^i/d\mu\in\cH^+(\gamma_0)$.
 According to the Coupling Lemma \ref{coupling1},  for any $n\geq 1$, there exists  a decomposition:
$$\nu^i=\sum_{m=1}^{n}\nu_m^i+\bar\nu^i_n,$$ for $n\geq 1$, where $\cF^m_*\nu^1_m$ is coupled with $\cF^m_*\nu^2_m$.
 \begin{lemma}\label{decayinflemma}
There exists  $C>0$,  such that for any $n\geq 1$, any $f\in \cH^-(\gamma_0)$,
\begin{align*}
&\left|\sum_{m=1}^{n}(\nu_m^1(f)-\nu^2_m(f))\right| \leq C\|f\|^-_{C^{\gamma_1}}L(n) n^{-{\alpha_0}}.\end{align*}
\end{lemma}
\begin{proof}

For any $n\geq 1$, by definition of $C_{n,b}^c$ and the choice of $b$, for  $f\in\cH^-(\gamma_f)$ there exists $C>0$ such that for $x\in C_{n,b}^c$ and
$y\in W^s(x)$,
\beq\label{holderbarx}|f(\cF^n(x))-f(\cF^n(y))|\leq C\|f\|^-_{\gamma_f}\Lambda^{-\gamma_f b\ln n}\leq C\|f\|^-_{\gamma_f}\Lambda^{-\gamma_0 b\ln n}\leq C\|f\|^-_{\gamma_f} n^{-{\alpha_0}}.\eeq

Now for any $x\in W^s\subset \Gamma^s$, we choose $\bar x\in W^s(x)$, such that $f(\bar x)=\max_{y \in W^s(x)} f(y)$ be the maximum value of $f$ along stable manifold $W^s(x)$. Then (\ref{holderbarx}) implies that for $x\in C_{n-m,b}^c\cap \Gamma^s$,
$$|f\circ \cF^{n-m}(x)-f\circ\cF^{n-m}(\bar x)|\leq C\|f\|_{\gamma_f}^-(n-m)^{-{\alpha_0}}.$$

Thus there exists $C_1>0$, such that for $i=1,2$,
\begin{align*}I_n^i:&=\sum_{m=1}^{n-1}\int_{\Gamma^s\cap C_{n-m,b}^c}\biggl| f \circ \cF^{n-m}(x)-f\circ\cF^{n-m}(\bar x)\biggr| d\cF^{m}_*\nu_m^i(x)\\
&\leq C\|f\|_{\gamma_f}^-\sum_{m=1}^{n-1}\cF^m_*\nu_m^i(\Gamma^s)(n-m)^{-{\alpha_0}}\\
&\leq C_1\|f\|_{\gamma_f}^- L(n)n^{-{\alpha_0}},\end{align*}
where we have used  the following estimate:
\beq\label{n-xtl}\int_{1}^{n-1} \frac{1}{L(x)x^l} \cdot \frac{1}{(n-x)^t} \, dx
\leq C_2 L(n)n^{-t}+ C_3 \frac{\ln n}{n^{l+t-1}}\leq C_1L(n)n^{-t},\eeq for any $l\geq t> 1$.
Moreover, we used  \textbf{(C2)} in the Coupling Lemma \ref{coupling1} in the last estimate, which states $$\cF^m_*\nu_m^1(\Gamma^s)=\cF^m_*\nu_m^2(\Gamma^s)\leq L(n) n^{-\alpha_0}.$$
Now we consider  for $i=1,2$,
\begin{align*} I\!I^i_n\colon&=\sum_{m=1}^n\biggl|\int_{\Gamma^s\cap C_{n-m,b}} f \circ \cF^{n-m}(\bar x) d\cF^{m}_*\nu_m^i(x)\biggr|\\
&\leq C\|f\|_{\infty}\sum_{m=1}^n \cF^m_*\nu^i_m ( C_{(n-m,b)})\\
&\leq C\|f\|_{\infty}\sum_{m=1}^n(n-m)^{-{\alpha_0}} \cF^m_*\nu^i_m(M)\\
&\leq C_2\|f\|_{\infty} L(n) n^{-{\alpha_0}},
\end{align*}
where  we have used Lemma \ref{coupling1} item \textbf{(C2)}, and $C_2>0$.

Combining the above estimates, we have \begin{align*}
I^1_n+I^2_n+I\!I^1_n+I\!I^2_n&\leq C\|f\|^-_{C^{\gamma_f}}n^{-{\alpha_0}},\end{align*}
for some constant $C>0$.

This implies that
\begin{align*}
&\sum_{m=1}^n\biggl|\int_{\cW_m^1} f\circ \cF^n d\nu_m^1-\int_{\cW_m^2} f\circ \cF^n d\nu_m^2\biggr|\\
&\leq \sum_{m=1}^n\biggl|\int_{\cF^m\cW_m^1} f \circ \cF^{n-m} -f\circ\cF^{n-m}(\bar x) d\cF^{m}_*\nu_m^1-\int_{\cF^m\cW_m^2} f\circ\cF^{n-m} -f\circ\cF^{n-m}(\bar x) d\cF^{m}_*\nu_m^2\biggl|\\
&\,\,\,+\sum_{m=1}^n\biggl|\int_{\cF^m\cW_m^1} f\circ\cF^{n-m}(\bar x) d\cF^{m}_*\nu_m^1-\int_{\cF^m\cW_m^2} f\circ\cF^{n-m}(\bar x) d\cF^{m}_*\nu_m^2\biggr|\\
&\leq (I^1_n+I^2_n+I\!I^1_n+I\!I^2_n)+\sum_{m=1}^n\biggl|\int_{\cF^m\cW_m^1} f\circ\cF^{n-m}(\bar x) d\cF^{m}_*\nu_m^1(x)-\int_{\cF^m\cW_m^2} f\circ\cF^{n-m}(\bar x) d\cF^{m}_*\nu_m^2(x)\biggr|\\
&=I^1_n+I^2_n+I\!I^1_n+I\!I^2_n
\leq C\|f\|^-_{C^{\gamma_f}} L(n) n^{-{\alpha_0}}.\end{align*}

\end{proof}

 Using this lemma we can estimate the following decay rates of correlations.
\begin{lemma}\label{decayinf}
For any two  proper standard  families  $\cG^i=(\cW^i, \nu^i)$, with $d\nu^i=g^i d\mu$, for any $f\in \cH^-(\gamma_f)$, with $\gamma_f\geq \gamma_0$, then:
\begin{itemize}
\item[(1)]
$
\left|\int_{\cM} f\circ \cF^n d\nu^1 -\int_{\cM} f\circ \cF^n d\nu^2\right|
\leq C\|f\|^-_{C^{\gamma_f}}L(n) n^{1-\alpha_0},$
for any $n\geq 1$, and $C=C(\gamma_0)>0$.
\item[(2)]  If $\supp( f)\subset M$ or $ \supp(g)\subset M$, then
$
\left|\int_{\cM} f\circ \cF^n d\nu^1 -\int_{\cM} f\circ \cF^n d\mu\right|
\leq C\|f\|^-_{C^{\gamma_f}}\mu(R>n).$
\end{itemize}
 \end{lemma}
\begin{proof}
 By Lemma \ref{coupling1}   for  any $n\geq 1$, there exists  a decomposition
$$\nu^i=\sum_{m=1}^{n}\nu_m^i+\bar\nu^i_n,$$  where  $\cF^m_*\nu^1_m$ is coupled with $\cF^m\nu^2_m$ such that for any $m\geq 1$,  for any measurable function $f\in\cH^-(\gamma_f)$ that is constant on any stable manifold $W^s\in\Gamma^s$,
 \beq\label{coupling mapf} \cF^m_*\nu^1_m(f)=\cF^m_*\nu^2_m(f).\eeq

Using the above Lemma \ref{decayinflemma}, we get
\begin{align*}
&\sum_{m=1}^n\biggl|\int_{\hcW_m^1} f\circ \cF^n d\nu_m^1-\int_{\hat\cW_m^2} f\circ \cF^n d\nu_m^2\biggr| \leq C_1\|f\|^-_{C^{\gamma_f}}L(n) n^{-{\alpha_0}},\end{align*}
where $C_1=C_1(\gamma_0)>0$.

Now combining the above facts together with the Coupling Lemma \ref{coupling1},  for both $i=1,2$,
 \beq\label{realbd}\int_{\cM} f\circ \cF^{n} d\nu^1-\int_{\cM} f\circ \cF^{n}\, d\nu^2=\bar\nu^1_n(f\circ \cF^n)-\bar\nu^2_n(f\circ \cF^n)+\sum_{m=1}^n\bigg(\int_{\hcW_m^1} f\circ \cF^n d\nu_m^1-\int_{\hcW_m^2} f\circ \cF^n d\nu_m^2\biggr). \eeq
 Our analysis implies that the second term is of order $L(n)n^{-{\alpha_0}}$, thus the decay rate is essentially dominated by $\bar\nu^1_n(f\circ \cF^n)$ for general observable $f$.
 Note that by Lemma \ref{coupling1}   \textbf{(C2)}, there exists $C>1$ large. such that
 \begin{align}\label{suppfg}
|\bar\nu^i_n(f)|\leq \|f\|_{\infty} \bar\nu^i(\cM)\leq C \|f\|_{\infty} (\mu(R>n)+\mu((\cF^{-n}\supp(f))\cap\supp(g^i)\cap (R\leq n)\cap C_{n,b})).
 \end{align}
Note that $\bar\nu^1(\cM)=\bar\nu^2(\cM)$, thus the above estimation can be improved as
$$|\bar\nu^i_n(f)|\leq C \|f\|_{\infty} (\mu(R>n)+\mu((\cF^{-n}\supp(f))\cap\supp(g^1)\cap\supp(g^2)\cap (R\leq n)\cap C_{n,b})).$$
 Thus we have for $n\geq 1$,
\begin{align*}
&\left|\int_{\cM} f\circ \cF^{n} d\nu^1-\int_{\cM} f\circ \cF^{n}\, d\nu^2\right|
\\&\leq 2\|f\|^-_{C^{\gamma_f}}(\mu(R>n)+\mu((\cF^{-n}\supp(f))\cap\supp(g^1)\cap\supp(g^2)\cap(R\leq n)\cap C_{n,b}))\\
&\leq C\|f\|^-_{C^{\gamma_f}} L(n)n^{1-\alpha_0},\end{align*}
where we have used  Lemma \ref{coupling1}  for the last step, and $C=C(\gamma_0)>0$. This leads to the desired estimate as we have claimed.

If $\supp( f)=M$ , assumption \textbf{(H3)} and its remark imply that: $$\mu((\cF^{-n}\supp(f))\cap\supp(g^1)\cap\supp(g^2)\cap(R\leq n)\cap C_{n,b}))=\mu(\cF^{-n}M\cap(R\leq n)\cap C_{n,b}))\leq C L(n)n^{-\alpha_0}.$$
Thus
\begin{align}\label{McF-nMCnb}
&\left|\int_{\cM} f\circ \cF^{n} d\nu^1-\int_{\cM} f\circ \cF^{n}\, d\nu^2\right|
\leq C\|f\|^-_{C^{\gamma_f}}\mu(R>n).\end{align}
Similarly, if  $ \supp(g)=M$, assumption \textbf{(H3)} implies that: $$\mu((\cF^{-n}\supp(f))\cap\supp(g^1)\cap\supp(g^2)\cap(R\leq n)\cap C_{n,b}))=\mu(M\cap(R\leq n)\cap C_{n,b}))\leq C L(n)n^{-\alpha_0}.$$This also implies (\ref{McF-nMCnb}).
\end{proof}

Note that in the result of Lemma \ref{decayinf}, the decay rates do not depend on the $\|g^i\|_{\gamma_{g^i}}$ norm, because we assumed that $g^i$ defines a proper standard family, which already put a condition on  $\|g^i\|_{\gamma_{g^i}}$, according to (\ref{defnpropercM}).
But this is not true anymore for  the case when we do not have proper standard  families, but only H\"older observables.
\begin{lemma}\label{Lemma91}
For any  piecewise H\"older continuous functions $f\in\cH^-(\gamma_f),  g\in\cH^+(\gamma_g)$,  with $\gamma_f,\gamma_g>\gamma_0$, for any $n\geq 1$,
\begin{align*}
|\Cov&(f\circ \cF^n,g)|:=|\mu(f\circ \cF^n\cdot g)-\mu(f)\mu(g)|\\&
\leq C\|f\|^-_{C^{\gamma_f}}  \|g\|_{C^{\gamma_g}}^+\left( \mu(\supp(g\circ \cF^{-n/2})\cap C_{n/2,b}\cap \cF^{n/2}\supp(f))+n^{-\alpha_0}\right)\\
&\leq C\|f\|^-_{C^{\gamma_f}} \|g\|^+_{C^{\gamma_g}} L(n) n^{1-\alpha_0},\end{align*}
for any $n\geq 1$, and $C=C(\gamma_0)>0$.
\end{lemma}

\begin{proof}
We first consider the case when $g\geq 0$, with $\mu(g)>0$. For simplicity, we denote $\hat g=g/\mu(g)$. Since $\mu$ can be disintegrated on the measurable family of unstable manifolds $\cW^u=\{W_{\alpha},\alpha\in \cA^u\}$, as a standard family $((W_{\alpha},\mu_{\alpha}),\cA^u,\lambda^u)$, such that for any measurable set $A\subset \cM$,
$$\mu(A)=\int_{\alpha\in \cA^u}\mu_{\alpha}(W_{\alpha}\cap A)\,\lambda^u(d\alpha).$$

Since $g\in \cH^+(\gamma_g)$ and  $\cF^{n}\supp(g)\subset \cM$,  there exists $\cA\subset \cA^u$, such that $\cW=\{W_{\alpha}\,:\,\alpha\in \cA\}$ is a measurable foliation   of $\cF^n\supp(g)$ into unstable manifolds,  with factor measure $\lambda^u(\alpha)$, and  for any measurable set $A\subset M$, we define a probability measure $\nu^i$ such that
$$\nu(A):=\int_{\alpha\in \cA}\int_{x\in W_{\alpha}\cap A} g(x)\,d\mu_{\alpha}\,\lambda^u(d\alpha).$$
We denote the standard family $\cG=(\cW,\nu)$.
 For any $n\geq 1$, we obtain
$$\cF^n_*\nu(A)=\int_{\beta\in \cA_n}\int_{y\in W_{\beta}}\bI_A(y) \, d\mu_{\beta}(y)\lambda(d\beta)$$ where $\cA_{n}$ is the index set for $\cF^n \cW$.
We also define for any $x\in \cF^{-n}W_{\beta}\subset W_{\alpha}$ and $\beta\in \cA_n$, the average function
$$\bar g_n(x):=\hat g(x_{\beta}^n):=\mathbb{E}(\hat g|\cF^{-n} W_{\beta})=\int_{\cF^{-n} W_{\beta}} \hat g(y)\,d\cF^{-n}_*\mu_{_{\beta}}(y),$$   here $x_{\beta}^n\in \cF^{-n}W_{\beta}$ exists because $\hat g$ is continuous on $ \cF^{-n}W_{\beta}$. Moreover one can check that $\mu(\bar g_n)=\mu(\hat g)=1$.
By the  H\"older continuity of $\hat g$, we have for any $x\in M$,
$$\sup_{\alpha\in \cA^u}\sup_{x\in W_{\alpha}} |\hat g\circ F^{-n}(x) -\hat g(x_{\beta}^n)|\leq c_M\|\hat g\|_{{\gamma_g}}\vartheta^{-n\gamma_g}.$$

 Note that for any  $x\in W_{\alpha}\cap C_{n,b}^c$, then the forward iteration of $x$ under $\cF$ visit $M$ at least $(b\ln (n))^2$ times.
 By the  H\"older continuity of $\hat g$, we have:
\begin{align*}
&\int_{\beta\in \cA_{n/2}}\left(\int_{y\in W_{\beta}, \cF^{-n/2}W_{\beta}\in C_{n/2,b}^c}f(\cF^{n/2}y)(  \hat g\circ \cF^{-n/2}(y) -\hat g(y_{\beta}^{n/2}) d\mu_{\beta}(y)\right)\, \lambda(d\beta)\\
&\leq C(\gamma_0) \|f\|_{\infty}\|\hat g\|_{{\gamma_g}}\vartheta^{\gamma_g (b\ln n/2)^2}\leq  C(\gamma_0)\|f\|_{\infty} \|\hat g\|_{{\gamma_g}}n^{-\alpha_0},
\end{align*}
where $C(\gamma_0)>0$ is a constant.
Moreover,
\begin{align*}
&|\int_{\beta\in \cA_{n/2}}\int_{y\in W_{\beta}\cap \cF^{n/2} C_{n/2,b}} f(\cF^{n/2} x) ( \hat g(\cF^{-n/2}x) -\hat g(x^{n/2}_{\beta}) )\, d\mu_{\beta}(x)\, \lambda(d\beta)|\\
&\leq C\|f\|_{\infty}\|g\|_{\infty}\mu(C_{n/2,b}\cap \supp(g)\cap\cF^{-n}\supp(f)).
\end{align*}

Thus
\begin{align*}
&\Cov(f\circ \cF^n, \hat g)=\Cov(f\circ \cF^{n},\bar g_{n/2})+\mu\left(f\circ \cF^{n},(\hat g-\bar g_{n/2} )\cdot \bI_{C_{n/2,b}^c})\right)+\mu(f\circ \cF^{n}, (\hat g-\bar g_{n/2} ) \cdot \bI_{C_{n/2,b}})\nonumber\\
&\leq \Cov(f\circ \cF^{n},\bar g_{n/2}
) + C(\gamma_0)\|f\|_{\infty} \|g\|_{{\gamma_g}}n^{-\alpha_0}+C\|f\|^-_{C^{\gamma_f}}  \|g\|_{C^{\gamma_g}} ( \mu(\supp(g)\cap C_{{n/2},b}\cap \cF^{-n}\supp(f)).
\end{align*}

Note that
\begin{align*}
|\int_{\alpha\in \cA}\int_{x\in W_{\alpha}\cap \cF^{-n/2}C_{n/2,b}} &\bar g_{n/2}(x)(f(\cF^{n} x) -\mu(f))\, d\mu_{\alpha}(x)\, \lambda(d\alpha)|\\
&\leq C\|f\|_{\infty}\|g\|_{\infty}\mu(\cF^{-n} C_{n/2,b}\cap \supp(g)\cap\cF^{-n/2}\supp(f)).
\end{align*}

We denote $E_n=\cup_{k=n}^{\infty}\cup_{i=0}^{n-1} \cF^i M_k$. Clearly, $\mu(E_n)\leq C n^{1-\alpha_0}$. We also denote $H_n=E_n^c\cap C_{n/2, b}^c$.
It remains  to estimate the following term:
\begin{align*}
\Cov&(f\circ \cF^{n},\bar g_{n/2}\cdot \bI_{\cF^{-n/2}H_n})\\
&=|\int_{\alpha\in \cA}\int_{x\in W_{\alpha}\cap \cF^{-n/2}H_n} (f\circ \cF^n(x)-\mu(f))\bar g_{n/2}\,d\mu_{\alpha}\,\lambda^u(d\alpha)|\\
&=|\int_{\beta\in \cA_{n/2}}\bar g_{n/2}(\beta)\cdot \mu_{\beta}(W_{\beta}\cap H_n)\left(\int_{y\in W_{\beta}\cap H_n}(f(\cF^{n/2} y) -\mu(f))\, d\bar\mu^{n/2}_{_{\beta}}(y)\right)\, \lambda(d\beta)|,
\end{align*}
where $\bar\mu^{n/2}_{_{\beta}}$ is the probability measure by conditioning $\mu_{_{\beta}}$ on $W_{\beta}\cap H_n$.

\begin{align*}
&|\int_{y\in W_{\beta}\cap H_n}f(\cF^{n/2} y) -\mu(f)\, d\bar\mu^{n/2}_{_{\beta}}(y)|\\
&=|\int_{\cF^{^{N_{\beta}}}(W_{\beta}\cap H_n)}f\circ \cF^{n/2-N_{\beta}} -\mu(f)\, d\cF^{N_{\beta}}_*\bar\mu^{n/2}_{_{\beta}}|\\
&=|\cF^{n}_*\bar\mu^{n/2}_{_{\beta}}(f)-\mu(f)|\leq  C\|f\|^-_{C^{\gamma_f}}  L(n) ({n/2}-N_{\beta})^{1-\alpha_0}\\
&\leq 2\alpha_0 C\|f\|^-_{C^{\gamma_f}} L(n)  n^{1-\alpha_0}.\end{align*}
In particular, if $\supp(f)=M$ or $\supp (g)=M$, then for any $\beta\in \cA_{n/2}$ with  if $|W_{\beta}|\geq \chi_1^{a} (n/3)^{-a}$,
\begin{align*}
&|\int_{y\in W_{\beta}\cap H_n}f(\cF^{n/2} y) -\mu(f)\, d\bar\mu^{n/2}_{_{\beta}}(y)|\\
&=|\cF^{n}_*\bar\mu^{n/2}_{_{\beta}}(f)-\mu(f)|\leq 2\alpha_0 C\|f\|^-_{C^{\gamma_f}}\mu(R>n).
\end{align*}
This implies that
\begin{align*}\Cov&(f\circ \cF^n,\bar g_{{n/2}}\cdot \bI_{\cF^{-{n/2}}(C_{{n/2},b}^c)})\\&\leq |\int_{\beta\in \cA_{{n/2}}, W_{\beta}\subset C_{{n/2},b}^c}\hat g_{\beta}\left(\int_{x\in W_{\beta}}f(\cF^{{n/2}} x) -\mu(f)\, d\mu_{ _{\beta}}(x)\right)\, \lambda(d\beta)|\\
&\leq 2\alpha_0 C\|g\|_{\infty}\|f\|^-_{C^{\gamma_f}}   n^{-\alpha_0}.\end{align*}

Combining the above facts, we get
\begin{align*}
&\Cov(f\circ \cF^n, g)=\Cov(f\circ \cF^{n},\bar g_{{n/2}})+\mu\left(f\circ \cF^{n},g-\bar g_{{n/2}})\right)\nonumber\\
&=\Cov(f\circ \cF^n,\bar g_{{n/2}}\cdot \bI_{\cF^{-{n/2}}C_{{n/2},b}^c})+\Cov(f\circ \cF^n,\bar g_{{n/2}}\cdot \bI_{\cF^{-{n/2}}C_{{n/2},b}})+\mu\left(f\circ \cF^{n},g-\bar g_{{n/2}})\right)\nonumber\\
&\leq 2(\alpha_0+C_q) C(\gamma_0)\|f\|^-_{C^{\gamma_f}}  \|g\|_{C^{\gamma_g}}^+\chi_1^{\alpha_0+1} n^{-\alpha_0}+\|f\|_{\infty}\|g\|_{\infty}\mu(\supp(g\circ \cF^{-{n/2}})\cap C_{{n/2},b}\cap \cF^{{n/2}}\supp(f))\\
&\leq C\|f\|^-_{C^{\gamma_f}}  \|g\|_{C^{\gamma_g}}^+ \left( \mu(\supp(g\circ \cF^{-{n/2}})\cap C_{{n/2},b}\cap \cF^{{n/2}}\supp(f))+\mu(C_{n/2,b}\cap \supp(g)\cap\cF^{-n}\supp(f))+n^{-\alpha_0}\right).\end{align*}

For the general case when $g$ is not nonnegative, we decompose $g=g_+-g_-$ into its positive and negative parts. Since $g_{\pm}\in \cH^+(\gamma_g)$,   a similar statement can be proved.
\end{proof}
Note that this proves Theorem \ref{main2}. Next result immediately follows from  above lemma.

\begin{lemma}\label{Lemma9}
For any  piecewise H\"older continuous functions $f\in\cH^-(\gamma_f),  g^i\in\cH^+(\gamma_g)$, $i=1,2$, with $\mu(g^1)=\mu(g^2)$, $\gamma_g>\gamma_0$, then for any $n\geq 1$,
\begin{align*}
\Delta&(f\circ \cF^n,g^1,g^2):=|\mu(f\circ \cF^n\cdot g^1)-\mu(f\circ \cF^n\cdot g^2)|\\&
 \leq  C_1\|f\|_{_{C^{\gamma_f}}}\|g^1\|_{_{C^{\gamma_g}}}(\mu(R>n)+\mu(\supp(g^1\circ \cF^{-n/2})\cap C_{n/2,b}\cap \cF^{-n/2}\supp(f))\\
 &+C_1\|f\|_{_{C^{\gamma_f}}}\|g^2\|_{_{C^{\gamma_g}}}(\mu(R>n)+\mu(\supp(g^2\circ \cF^{-n/2})\cap C_{n/2,b}\cap \cF^{-n/2}\supp(f))\\
&\leq  C_2\|f\|_{_{C^{\gamma_f}}}\max\{\|g^1\|_{_{C^{\gamma_g}}} ,\|g^2\|_{_{C^{\gamma_g}}} \} n^{1-\alpha_0},\end{align*}
where $C_i=C_i(\gamma_0)>0$, for $i=1,2$.
\end{lemma}

An immediate application of the above Lemma is the following:
\begin{corollary}\label{ABmixing}If $\mu(g)>0$, and $g/\mu(g)$   defines a standard family $\cG^1=(\cW, \nu)$, then
\begin{align*}
\left|\int_{\cM} f\circ \cF^{n} \cdot g d\mu-\mu(f)\mu(g)\right|
&\leq 2\|f\|^-_{C^{\gamma_f}}\mu(g) (\mu(R>n)+\mu((\cF^{-n}\supp(f))\cap B\cap(R\leq n)\cap C_{n,b}))\\
&\leq C\mu(g)\|f\|^-_{C^{\gamma_f}} L(n)n^{1-\alpha_0}.\end{align*}
\end{corollary}


\section{Alpha-mixing property and other limiting properties}

Similar to Subsection \ref{alphamixforF}, we first introduce the following  natural family of $\sigma$-algebras for the system $\cF: \cM\to \cM$.
Recall that $\cS_{\pm n}$ is the singularity set of $\cF^{\pm n}$ for $n\ge 1$.
Let $\xi_0(\cM):=\{\cM\}$ be the trivial partition of $\cM$,
and denote by $\xi_{\pm n}(\cM)$ the partition of $\cM$ into connected components of
$\cM\backslash \cF^{\mp(n-1)} \cS_{\pm 1}$ for $n\ge 1$.
Further, let
$$
\xi_m^n(\cM):=\xi_m(\cM)\vee \dots \vee \xi_(\cM)$$
for all $-\infty\le m\le n\le \infty$.
By Assumption (\textbf{H2}), $\xi_0^\infty(\cM)$ is the partition of $\cM$ into maximal unstable manifolds,
and $\xi_{-\infty}^0(\cM)$ is that into maximal stable manifolds.
Also, $\mu(\partial \xi_m^n(\cM))=0$ by Assumption (\textbf{H4}), where
$\partial \xi_m^n(\cM)$ is the set of boundary curves for components in $\xi^n_m(\cM)$.

Let $\fF_m^n(\cM)$ be the $\sigma$-algebra generated by the partition $\xi_m^n(\cM)$ on $\cM$.
Notice that
$\fF_{-\infty}^\infty(\cM)$ coincides with the Boreal $\sigma$-algebra of  $\cM$.
We denote by $\fF(\cM):=\{\fF_m^n(\cM)\}_{-\infty\le m\le n\le \infty}$ the family of those $\sigma$-algebras on $\cM$.
To make notation simpler, we omit $(\cM)$ in the notation of sigma-algebras  below. We denote the alpha-mixing coefficient as:
$\alpha(n):=\sup_{k\in \IZ}  \alpha(\fF_{-\infty}^k, \fF_{k+n}^\infty)$, for any $n\geq 1$. Since processes generated by $\cF$ are stationary, thus
$$\alpha(n)=\alpha(\fF_{-\infty}^0, \fF_{n}^\infty).$$

\begin{proposition}\label{prop: alpha mixingpoly}
 The family $\fF$ is $\alpha$-mixing with a polynomial rate, i.e.,
there exist $C_0>0$ and $\vartheta_0\in (0, 1)$ (which are the same as in Proposition \ref{exp decay}) such that
\beq\label{1-alpha_0}
\alpha(n)\le C_0 n^{1-\alpha_0},
\eeq
where the definition of $\alpha(\cdot, \cdot)$ is given by \eqref{alphapoly}.
\end{proposition}

\begin{proof} By the fact that $\cF^{-k}\xi_m^n=\xi_{m+k}^{n+k}$ and the invariance of $\mu$, it suffices to show that
\beq\label{alphapoly}
\alpha(\fF_{-\infty}^0, \fF_{n}^\infty)=\sup_{A\in \fF_{-\infty}^0} \sup_{B\in \fF_{n}^\infty}
\left|\mu(A\cap B)-\mu(A)\mu(B) \right| \le C_0L(n)n^{1-\alpha_0}.
\eeq
Since $A\in \fF_{-\infty}^0$ is a union of some maximal stable manifolds, we have that $\b1_A\in \cH^-({\gamma_0})$
such that $\|\b1_A\|_{L^\infty}=1$ and $|\b1_A|_{\gamma}^-=0$ for any $\gamma\in (0, 1)$.
Similarly, $B\in \fF_{n}^\infty$ implies that $\hat B:=F^{-n}(B)\in \fF^\infty_0$
is a union of some maximal unstable manifolds,
and thus $\b1_{\hat B}\in \cH^+({\gamma_0})$ such that $\|\b1_{\hat B}\|_{L^\infty}=1$ and
$|\b1_{\hat B}|_{\gamma_0}^+=0$.
Therefore, by Corollary~\ref{ABmixing}, for any $A\in \fF_{-\infty}^0$ and $B\in \fF_{n}^\infty$,
\beq
\left|\mu(A\cap B)-\mu(A)\mu(B) \right|=
\left| \mu(\b1_{\hat B}\cdot \b1_A\circ \cF^n) -\mu(\hat B) \mu(A) \right|
\le C\mu(B)L(n)n^{1-\alpha_0}.
\eeq
This completes the proof of Proposition~\ref{prop: alpha mixingpoly}.
\end{proof}

Note that the above Proposition implies that the system $(\cF,\cM)$ generates stationary, $\alpha$-mixing processes with rate $\alpha(n)=\cO(n^{1-\alpha})$.  Let $\{X_n\}_{n\in \mathbb{Z}}$ be a stationary, $\alpha$-mixing process on probability space $(\Omega, \mathbb{P})$ with zero mean, and we denote $\alpha_X(n):=\alpha(\sigma(X_0), \sigma(X_n))$, for $n\geq 0$,  be the sequence of alpha mixing coefficients generated by the process $X=\{X_n\}$. Let $\alpha_X(t)=\alpha_X([t])$ be the corresponding alpha mixing  function generated by the process.  Let
$$Q_X(t)=\inf\{s: \mathbb{P}(|X|>s)\leq t\}$$ denote the quantile of $|X|$.  

We denote $S_n:=X_1+\cdots+X_n$, and define the process $\{W_n(t)\,:\, t\in [0,1]\}$, with
$$\sqrt{n} W_n(t):=S_{[nt]}.$$
For each $\omega\in \Omega$, $W_n(t)(\omega)$ is an element of the Skorohod space $D[0,1]$ of all functions on $[0,1]$, with left-hand limit and are continuous from the right. It is equipped with the Skorohod topology (cf. Bilingsly, 1968, Sect 14). Let $W(t)$ denote the standard Brownian Motion on $[0,1]$. If the distribution of $W_n(\cdot)$ convergences weakly in $D[0,1]$ to $\sigma W(t)$, for some $\sigma>0$, $\{X_n\}$ is said to satisfies the functional central limit theorem.
By Theorem 1 and 2 in \cite{Rio1992}, we have
\begin{lemma}\label{boundsigma2}
Let $\{X_n\}_{n\in \mathbb{Z}}$ be a stationary, $\alpha$-mixing process, with
\beq\label{condCX}
C_X:=\int_0^{1} \alpha^{-1}(u) |Q_{X}(u)|^2\, du<\infty,\eeq
here $ \alpha^{-1}(u) $ is the inverse function of $\alpha$.
Then
$$\sigma^2_X:=\sum_{k=-\infty}^{\infty}\Cov(X_0, X_n)\leq 8 C_X.$$
Moreover, $$Var(S_n)\leq 8C_X n,\,\,\,\,\,\,\,\,\lim_{n\to\infty}n^{-1}Var(S_n)=\sigma^2.$$
\end{lemma}

For uniformly bounded observable $X_n$, one can see that $Q_X(0)=\|X\|_{\infty}$, and $Q_X(t)\leq \|X\|_{\infty},$ for $t\in [0,1)$. Moreover,  Proposition \ref{prop: alpha mixingpoly} implies that
$$\int_0^{1} \alpha^{-1}(u) |Q_{X}(u)|^2\, du<\infty$$
which also implies condition (\ref{condCX}). Thus $\{X_n\}$ satisfies the CLT.

Furthermore, the functional CLT was proved in \cite{DMR94}:
\begin{lemma}
Under the above condition in Lemma \ref{boundsigma2}, then the series $\sum_{k=-\infty}^{\infty} \Cov(X_0, X_n)$ is absolutely convergent to $\sigma^2$, and $Z_n$ converges in distribution to $\sigma W$ in the Skorohod space $D[0,1]$, for $\sigma\geq 0$.
\end{lemma}
Note that the above functional CLT holds even for unbounded observables $X_n$, which is new under the current settings.

Let $A\subset \cM$ be a measurable set. We define the hitting time of $A$ by
$\tau_A(x)=\int \{k\geq 1\,:\, \cF^k x\in A\}$, $x\in \cM$.
We are interested in the distribution of the hitting time $\tau_A$ on the probability space $(\cM,\mu)$, and
the return time, defined with the same formula, but on the probability space $(A,\mu(\cdot|A))$ where
$\mu(\cdot|A)$ denotes the conditional measure on $A$.

\begin{lemma}
For any sequence $A_n\in \fF_0^{n-1}$, such that $\mu(A_n)>0$, and
$\mu(\tau_{A_n}\leq n)\to 0$ as $n\to\infty$,
there exists some normalizing constant $C_{A_n}>0$ such that the following hold:\\
 (1) The hitting time of $A_n$, rescaled by $C_{A_n}\mu(A_n)$, converges in distribution to an exponential
distribution. Namely,
$$\sup_{t\geq 0}|\mu(C_{A_n}\mu(A_n)\tau_{A_n}>t)-e^{-t}|\to 0,$$ as $n\to\infty$.
The convergence is uniform on families of sets $A_n$ where the convergence  is uniform.\\
 (2) The distribution of the return time is approximated by a convex combination of a Dirac mass
at zero and an exponential distribution. More precisely,
$$\sup_{t\geq s}
|C_{A_n}\mu(C_{A_n}\mu(A_n)\tau_{A_n}>t|A_n)-e^{-t}|\to 0$$ as $n\to\infty$
for any $s>0$.\\
(3) we have $\lim \sup C_{A_n}\leq 1$.
\end{lemma}

Since the system $(\cF,\cM,\mu)$ is $\alpha$-mixing with polynomial rates (\ref{1-alpha_0}), the above lemma directly follows from  the results by Miguel Abadi and Benoit Saussol \cite{AS10}.
\section{Proof of Theorem \ref{TmMain} and Theorem \ref{MCLXn0}.}
\label{secPMT}
\subsection{Proof of Theorem \ref{TmMain}}

We denote  $$B_n=\bigcup_{m>n}\bigcup_{k=1}^{m-n}\cF^k M_m=(R>n)\setminus M,$$ as the set of points in $\cM\setminus M$ that take at least $n$-iterations under $\cF$  before they come back to $M$. We first prove the following lemma that will be used to prove Theorem \ref{TmMain}.
\begin{lemma}\label{defnBn} For any large $n$, we define
$\mu^n=\frac{\mu|_{B_n^c}}{\mu(B_n^c)}$. Then for any  probability measure $\nu$ with support $M$, we have for any bounded function $f$ on $M$,
\beq\label{decay11}\cF^n_* \nu(f)-\mu(f)-\mu(f)\mu(R>n)=\cF^n_*\nu(f)-\cF^n_* \mu^n(f)+\left(\cO(\mu(R>n)/n)+\cO(\mu(R>n)^2)\right).\eeq

\end{lemma}
\begin{proof}
First note that $$B_n=\bigcup_{m>n}\bigcup_{k=1}^{m-n}\cF^k M_m=(R>n)\setminus (\cup_{m\geq n}M_m),$$ which implies that \beq\label{BnRn}\mu(B_n)=\mu(R>n)+\cO(n^{-1}\mu(R>n)).\eeq
Since the support of the initial measure $\nu$ is contained in $M$, we have that after
$n$ iterations the push forward measure $\cF^n_*\nu$ can never reach the
region $\cF^n B_n$.    Thus we can ignore the measure $\mu$ restricted on $B_n$ within first $n$-iterations. This fact implies that for
each $n\geq 1$ the measure $\mu$ is a linear combination of two
probability measures, $$\mu=\mu(B_n)\cdot\frac{\mu|_{B_n}}{\mu(B_n)}+\mu(B_n^c)\cdot\frac{\mu|_{B_n^c}}{\mu(B_n^c)}.$$
We denote
$\mu^n=\frac{\mu|_{B_n^c}}{\mu(B_n^c)}$. Note that for
any bounded function $f$ supported on $M$, we have $$\mu(f)=\mu^n
(f)\mu(B_n^c)=\mu^n(f)-\mu^n(f)\mu(B_n).$$

  Notice also  that (\ref{BnRn}) implies that $\mu(B_n)= O( L(n) n^{1-{\alpha_0}})$. Thus we have
  \begin{align*}
  \mu^n(f)=\frac{\mu (f\cdot \bI_{B^c_n})}{\mu(B_n^c)}=\frac{\mu(f)}{1-\mu(B_n)}=\mu(f)(1+ \mu(B_n)+O(\mu(B_n)^2)).
   \end{align*}
  Similarly, using the fact that $M\subset \cF^n B_n^c$, we also obtain  $$\cF^n_* \mu^n(f)=\mu(f)(1+ \mu(B_n)+O(\mu(B_n)^2)).$$
In particular, this implies that
$$\mu(\bI_M\cap((R\leq n)\cup M))=\cF^n_*\mu(\bI_M\cap\cF^n(R\leq n))=\mu(M)(1-\mu(R>n)).$$
This also implies that
$$\mu(f)(1+\mu(R>n))=\cF^n_* \mu^n(f)+\mu(f)\mu(R>n)$$
where $\beta_0=\min\{\alpha_0, 2\alpha_0-2\}$.
Here we have used the fact given by (\ref{BnRn}).
 Combining above facts,  we have
  \begin{align*}\cF^n_* \nu(f)-\mu(f)-\mu(f)\mu(R>n)&=\cF^n_*\nu(f)-\cF^n_* \mu^n(f)+\mu(f)\left(\cO(\mu(R>n)/n)+\cO(\mu(R>n)^2)\right).\end{align*}
\end{proof}



To prove Theorem \ref{TmMain}, we first assume the observable $g\in \cH^+(\gamma_0)$ defines a proper standard family $(\cW^u_M,\nu)$, with $d\nu=gd\mu$. By Lemma \ref{defnBn}, we have
\beq\label{decay101}\cF^n_* \nu(f)-\mu(f)-\mu(f)\mu(R>n)=\cF^n_*\nu(f)-\cF^n_* \mu^n(f)+\left(\cO\left(\frac{\mu(R>n)}{n}\right)+\cO(\mu(R>n)^2)\right).\eeq
Our goal is to show that for systems satisfy (\bH3) the correlation $\cF^n_*\nu(f)-\cF^n_* \mu^n(f)=o(\mu(R>n))$.
 We use  assumption \textbf{(H3)}, which states that $$\mu(C_{n,b}\cap \cF^{-n} M)=\cO(\mu(R>n)/n).$$

Next we start to prove Theorem \ref{TmMain}. Let $\gamma_f\geq \gamma_0$. We consider any $f\in \cH^-(\gamma_f)$ supported on $M$, and
 any proper family  $\cG=(\cW^u_M, \nu)$ with $g=d\nu/d\mu\in\cH^+(\gamma_0)$ supported on $M$. For any large $n$, we  define $\cG^1=\cG$ and $\cG^2=(\cW^u,\mu^n)$, then they are both proper families. We denote $\nu^1=\nu$ and $\nu^2=\mu^n$.

 First note that since both measures are essentially supported on $(R\leq n)$, thus we have
 $$\nu^i(R>n)=\cO(\mu_M(R>n))=\cO(\mu(R>n)/n).$$
According to Lemma \ref{decayinflemma} we know that  for any $k=1,\cdots, n$,
\begin{align*}
&\left|\int_{\cM} f\circ \cF^k d\nu^1 -\int_{\cM} f\circ \cF^k d \nu^2\right|\\&=\left|\bar\nu^1_k(f\circ \cF^k)-\bar\nu^2_k(f\circ \cF^k)+\sum_{m=1}^k\bigg(\int_{\cW_m^1} f\circ \cF^k d\nu_m^1-\int_{\cW_m^2} f\circ \cF^k d\nu_m^2\biggr)\right|\\
&\leq \left|\bar\nu^1_k(f\circ \cF^k)-\bar\nu^2_k(f\circ \cF^k)\right|+C(\gamma_0)\|f\|_{\gamma_f}\|g\|_{\infty} \mu(R>n)/n.\end{align*}

 According to the Coupling Lemma \ref{coupling1} and (\ref{suppfg}), we know that  the uncoupled measure $\bar\nu^i_k(\cM)$ is dominated by $C_{k,b}$, while $\bar\nu^i_k(M\cap C_{k,b})=\cO(\mu_M(R>k)).$ Thus
 \begin{align*}
 \bar\nu^i_k(f\circ \cF^k)&=
 \bar\nu^i_k(f\circ \cF^k\cdot \bI_{C_{k,b}})+
 \bar\nu^i_k(f\circ \cF^k\cdot \bI_{R>n})+\bar\nu^i_k(f\circ \cF^k\cdot \bI_{C_{k,b}^c})= \bar\nu^1_k(f\circ \cF^k\cdot \bI_{C_{k,b}})+\cO( \mu_M(R>k))\nonumber\\
& \leq \|f\|_{\infty}\|g\|_{\infty}\left(\mu(C_{k,b}\cap \cF^{-k}\supp(f)\cap \supp(g))\right)+\cO( \mu_M(R>k))\\
&=\|f\|_{\infty}\|g\|_{\infty}\left(\mu(C_{k,b}\cap \cF^{-k}M\cap M)\right)+\cO( \mu_M(R>k))=\cO(\mu_M(R>k)),
 \end{align*}
where we have used Assumption (\textbf{H3}) in the last estimate.
Combining above facts, we take $k=n$, then
$$|\int_{\cM} f\circ \cF^n d\nu^1 -\int_{\cM} f\circ \cF^n d\nu^2|\leq C_1(\gamma_0)\|f\|^-_{C^{\gamma_f}}\|g\|_{\infty}\mu_M(R>k).$$

Combining this with (\ref{decay101}), then for any $n>1$, we get
\begin{align*}
&|\cF^n_* \nu(f)-\mu(f\circ \cF^n)-\mu(f)\mu(B_n)|\\
&\leq |\cF^n_*\nu(f)-\mu^n(f\circ \cF^n)|+ C_1(\gamma_f)\|f\|^-_{C^{\gamma_f}}\|g\|_{\infty}+\left(\cO(\mu_M(R>n))+\cO(\mu(R>n)^2)\right)\\
&\leq C(\gamma_f) \|f\|^-_{C^{\gamma_f}}\|g\|_{\infty}\mu(R>n)+ C_1(\gamma_0)\|f\|^-_{C^{\gamma_f}}\|g\|_{\infty}\left(\cO(\mu(R>n)/n)+\cO(\mu(R>n)^2)\right).
\end{align*}

For general  H\"older observable $g$, we denote $g=g^+-g^-$. It is enough to consider the case when $\mu(g^{\pm})\neq 0$, then we consider $g^+$ and $g^-$ separately as in the proof of Theorem \ref{main2}
, to get  (\ref{mainh2a}) in Theorem \ref{TmMain}.

\subsection{Proof of Theorem \ref{MCLXn0}.}
One interesting application of Theorem \ref{TmMain}
 is that one gets classical central limiting theorem for stochastic processes generated by certain observables for dynamical systems with decay rates of order $\cO(1/n)$.

We consider an observable $f\in H^-(\gamma)\cap H^+(\gamma)$ with $\supp(f)\subset M$, and $\gamma>0$. We assume $f$ is not a coboundary. For any $n\geq 1$, we consider the two partial sums $\tS_n(f):=f+f\circ \cF+\cdots+f\circ \cF^n$ and $S_n(f)=f+f\circ F+\cdots+ f\circ F^n$. We assume $\mu(f)=0$.
For the induced map $(F,M,\mu_M)$, it follows from Theorem 7.52 in \cite{CM} and \cite{CZZ}, that conditions \textbf{(h1)-(h4)} implies that
\beq\label{CLT}\frac{S_n}{\sigma/2}\to N(0,1),\eeq in distribution, where $N(0,1)$ is the standard normal variable, and by the Green-Kubo formula,
\beq\label{sigmaF}\sigma^2=\mu_M(f^2)+2\sum_{n=1}^{\infty}\mu_M(f\circ F^n\cdot f).\eeq
 Note that $$\int f\circ F^n\cdot f\,\,d\mu_M\leq C\|f\|_{C^{\gamma}} \vartheta^n,$$ as the induced map enjoys exponential decay of correlations.

 In \cite{BCD}, the partial sum $S_n$ was associated with the so-called induced observable, $\tilde{S}_n{f}(x)=\sum_{m=0}^{R(x)-1} f(\cF^m x)$. However, since our observable $f$ has support contained in $M$, thus $\tilde{S}_n(f)=S_n(f)$ coincide with the induced observable.
Next we review the relation between CLT of $S_n$ and $\tS_n$, see a detailed proof in \cite{BCD}.

\begin{lemma}\label{MCLXn4} For any $n\geq 1$, if $S_n$ satisfies the CLT (\ref{CLT}), then $\tS_n$ also satisfy a CLT:
\beq\label{ClTZ}\frac{\tS_n}{\tilde{\sigma} n/2}\to N(0,1),\eeq
in distribution.
Here $\tilde{\sigma}^2 =\sigma^2\mu(M)$.
\end{lemma}

Thus (\ref{ClTZ}) implies that $\tS_n$ satisfies the classical CLT with variance $\sigma\sqrt{\mu(M)}$. Moreover, again by the Green-Kubo formula  we get
\begin{align*}
\sigma^2\mu(M)&=\mu(f^2)+2\sum_{n=1}^{\infty}\mu(f\circ \cF^n\cdot f).
\end{align*}
Comparing with  (\ref{sigmaF}), and use the definition $d\mu_M=d\mu/\mu(M)$, we get the interesting relation
$$\sum_{n=1}^{\infty}\mu(f\circ \cF^n\cdot f)=\sum_{n=1}^{\infty}\mu(f\circ F^n\cdot f).$$
This finishes the proof of Theorem \ref{MCLXn0}.\\

\section{Applications  to hyperbolic systems}

To illustrate our method, we  apply it to several classes of dynamical systems.
Since the induced maps for most of these examples were studied in \cite{CZ09},
we only remind some basic facts here.

First we recall standard definitions, see \cite{BSC90,BSC91,C99}. A
2-D flat billiard is a dynamical system where a point moves freely at
unit speed in a domain $Q\subset \mathbb{R}^2$  and bounces off its
boundary $\pQ$ by the laws of elastic reflection. We assume that
$\pQ=\cup_{i}\Gamma_i$ is a finite union of piecewise smooth curves,
such that each smooth component $\Gamma_i\subset\pQ$ is either convex
inward (dispersing), or flat, or convex outward (focusing).
Following Bunimovich, see \cite{Bu74,Bu79} and \cite[Chapter~8]{CM}, we
assume that every focusing component $\Gamma_i$ is an arc of a circle
such that there are no points of $\pQ$ on that circle or inside it,
other than the arc $\Gamma_i$ itself. Under these assumptions the
billiard dynamics is hyperbolic.

Let $ \cM=\pQ\times [-\pi/2,\pi/2]$ be the {\em collision space}, which
is a standard cross-section of the billiard flow. Canonical coordinates
in $\cM$ are $r$ and $\varphi$, where $r$ is the arc length parameter
on $\pQ$ and $\varphi\in [-\pi/2,\pi/2]$ is the angle of reflection.
The collision map $\cF \colon \cM \to \cM$ takes an inward unit vector
at $\partial Q$ to the unit vector after the next collision, and
preserves smooth measure $d\hmu = c\cdot \cos \varphi\, dr\, d\varphi$
on $\cM$, where $c$ is a normalization constant. Furthermore,  $\partial
\cM\cup\cF^{-1}(\partial \cM)$ is the singular set of $\cF$.

For billiards with focusing boundary components, the expansion and
contraction (per collision) may be weak during long series of
successive reflections along certain trajectories. To study the
mixing rates, one needs to find and remove the spots in the phase
space where expansion (contraction) slows down. Such spots come in
several types and are easy to identify, for example, see \cite{CZ05a}
and \cite[Chapter~8]{CM}. Traditionally, we denote
\begin{equation*}
   \partial Q=\partial^0 Q \cup \partial^{\pm} Q,
\end{equation*}
where $\partial Q^0$ is the union of flat boundaries,  $\partial Q^-$
contains focusing boundaries and $\partial Q^+$ contains dispersing
boundaries. The collision space can be naturally divided into
focusing, dispersing and neutral parts:
\begin{equation*}
   \cM_{0}=\{(r, \varphi)\colon r\in \partial^0 Q\},
   \quad\quad
   \cM_{\pm}=\{(r, \varphi)\colon r\in \partial^{\pm} Q\}.
\end{equation*}
Now we define the induced phase space:
\beq \label{M3}
   M= \{x\in \cM_-\colon\, \pi(x)\in \Gamma_i,\,
   \pi(\cF^{-1}x)\in \Gamma_j, \,j\neq i\}\cup \cM_+.
\eeq
Note that $M$ only contains all collisions on dispersing boundaries and
the \emph{first} collisions with each focusing arc (the collisions with
straight lines are skipped altogether). The map $F$ preserves the
measure $\mu$ conditioned on $M$, which we denote by $ \mu =
[\hmu(M)]^{-1} \hmu$.

Furthermore, $F$  has a larger singular set than
the original map.  Let $ S_0=\partial M$, then $ S_1:= S_0\cup
F^{-1} S_0$ is the singular set of $F$.
Let $R\colon \cM\to \mathbb{N}$ be the first return time function, such
that for any $x\in \cM$, $\cF^{R(x)}x$ returns to $M$ for the first time. We define  $M_m=R^{-1}\{m\}\cap M$
as the level set of the return time function.

To be more specific, we consider billiard systems that have been studied
in \cite{CZ05a,CZ05b,CZ07,CZ08,CM07}, which include semi-dispersing
billiards on a rectangle, Bunimovich billiard, billiards with flat points, billiards with cusps. It was proved in
these references that the induced system $(F,M,\mu)$ satisfies the
condition \textbf{(H2)} and \textbf{ (H3)} and enjoys exponential decay of
correlations. It is enough to check (\bH3). We first introduce some new conditions that imply \textbf{(H3)} and which are easier to check.\\

\subsection{Sufficient conditions for (\textbf{H3}).}
In this Subsection, we  introduce one sufficient condition to guarantee assumption (\textbf{H3}).\\


\noindent\textbf{Condition \textbf{(H3$'$)}}. We assume there exist $q \in(0,1)$,  $C>0, p>0$,  and $N>1$ such that for any sufficiently large $n>N$,  and for each $m=1, \ldots, (b\ln n)^2$:

$$\mu\bigl(\{x\in M\colon R(F^m(x))> n^{1-q}\}|M_n\bigr)<C n^{-p}.$$

We will prove the following  lemma.

 \begin{lemma}\label{CnbBnc} Condition \textbf{(H3$'$)} implies  \textbf{(H3)}.
\end{lemma}

\begin{proof}

For  any sufficiently large $n$, any $x \in C_{n,b}$, we define $$k_0=k_0(x)=\sum_{m=0}^{n-1}\bI_{M}(\cF^m(x)),$$ as the number of returns to $M$ within $n$ iterations under $\cF$ along the forward trajectory of $x$. It is enough to consider $C_{n,b}\cap (R\leq n)$. we know that any $x \in C_{n,b}\cap (R\leq n)$ must return to $M$ at least once, with $1\leq k_0<(b\ln n)^2$. Let $x_n\in M$ be the last entrance to $M$ within $n$-iterations, and we define
$$m_{1}(x):=\max_{0\leq k\leq k_0-1}\{R(F^{-k}(x_n))\}$$ to be the largest
return time function value of $R$ along  $n$ iterations of $x$ under $\cF$,  i.e.\ there exists $k_1=k_1(x)\in\{1,\cdots, n\}$, such that $x_1:=\cF^{-k_1}x_n\in M_{m_1}$, and $k_1+m_1\leq n$. Moreover we define indices: $$m_1(x)=R(x_1), m_2(x):=R(F(x_1)), \cdots, m_{k_2}(x):=R(F^{k_2-1}(x_1)),$$ with $F^{k_2-1}(x_1)=F^{-1}x_n$  being the second last return to $M$ along $n$ iterations of $x$.
Without loss of generality, we assume $$\sum_{k=1}^{k_2} m_k\geq n/2,$$ i.e.  the largest index $m_1$ occurs within the first $n/2$ iterations of $x$ under $\cF$. Then by assumption, since $m_1(x)$ is the largest index within $k_0$  returns to $M$ along forward $n$ iterations of $x$, we have
$$
 n/2\leq  m_1+\cdots +m_{k_2}\leq n.
$$

Since points $x\in C_{n,b}\cap (R\leq n)$ return to $M$ at most $\psi$ times during
the first $n$ iterates of $\cF$,  there are $\leq \psi$ number of
intervals between successive returns to $M$, and hence the longest
interval has length $m_1(x)\geq n/\psi$.  Let $c_0<1/2$ be a constant. We split $ C_{n,b}=C_{n,b}'\cup C_{n,b}''$ into two disjoint parts.\\

(I) $C_{n,b}'$ is a `good part' of
$C_{n,b}\cap \cF^{-n} M$ such that for any $y\in C_{n,b}'$,
\beq\label{epsk2}
   m_k(y)<m_1(y)^{1-q},
\eeq
for all
$$
   \frac{c_0n}{m_1(y)} \leq k \leq k_0\leq \psi(n)=(b\ln n)^2.
$$ More precisely, for $y\in C'_{n,b}$, there is a sequence of returns to $M_{m_k}$'s, with index decreasing exponentially in $k$, within $n$-iterations.
For `good' points $y\in C_{n,b}'$ we have
$$
 \frac{n}{2} \leq \sum_{k=1}^{\psi(n)} m_k \leq
  m_1\cdot \frac{c_0n}{m_1} + \sum_{k=\frac{c_0n}{m_1}}^{\psi(n)}m_1^{1-q} \leq c_0 n+m_1\leq c_0 n+m_1.
$$
 Since $c_0<1/2$,
 we conclude $m_1 \geq cn$ for a positive constant $c:=\frac{1}{2}-c_0>0$. This implies that  for points in $C_{n,b}'$, the largest index $m_1$ within $n$ iterations  must be approximately of order $n$.
Accordingly, the measure of good points in $C_{n,b}$ is bounded by
\beq\label{Cnb'H3b}\mu(C'_{n,b}\cap M)\leq \sum_{m_1\geq cn}\mu(M_{m_1})\leq  C n^{-{\alpha_0}}.\eeq
And
\beq\label{Cnb'H3bM}\mu(C'_{n,b})\leq m \sum_{m_1\geq cn}\mu(M_{m_1})\leq  C n^{1-{\alpha_0}}.\eeq

(II) On the other hand $C_{n,b}''$ consists of `bad' points $y\in C_{n,b}$, such that (\ref{epsk2}) fails. i.e.,
\beq\label{epsk3}
   m_k(y)>m_1(y)^{1-q},
\eeq
for some
$$
   \frac{c_0n}{m_1(y)} \leq k \leq \psi(n).
$$

We divide $C_{n,b}''$ according to the maximal index $m_1$ of its points $C_{n,b}''=\bigcup_{m_1>\frac{n}{\psi}} C_{n,b,m_1}$
such that $C_{n,b,m_1}$ contains all points in $C_{n,b}''$ with the largest return time $m_1$ within $n$ iterations.
The contribution of $M_{m_1}$ to these `bad' points in $C_{n,b}''$ will be estimated according to  (\bH3$'$) as following:
$$
 \mu(C_{n,b, m_1}) \leq \mu(M_{m_1})\sum_{k=\tfrac{c_0n}{m_1}}^{\psi}
 m_1^{-p}.
$$

By assumption $m_1\geq n/\psi$, so the total measure of
$C_{n,b}''$ can be estimated as in (\ref{Cnb'H3b}): $
  \mu(C_{n,b}''\cap M)\leq C_1 n^{-\alpha_0}.
$ This implies $\mu(C_{n,b}\cap M)= O(n^{-\alpha_0})
$.

\end{proof}

\subsection{Billiards with property \textbf{(H3$'$)}.}

Assume that each cell $M_n$ has dimension $\sim n^{-a}$ in the unstable direction, dimension $\sim n^{-\beta}$ in the stable direction, and measure $\mu(M_n)\sim n^{-d}$, with $d\geq a+\beta>2$.  We first foliate $M_n$ with unstable curves $W_{{\alpha}}\subset M_n$ (where $\alpha$ runs through an index set $\cA$). These curves have length $|W_{\alpha}|\sim n^{-\beta}$. Let $\nu_n:=\frac{1}{\mu(M_n)}\mu|_{M_n}$ be the conditional measure of $\mu$ restricted on $M_n$. Let $\cW=\cup_{\alpha\in \cA} W_{\alpha}$ be the collection of all unstable curves, which foliate the cell $M_n$. Then we can disintegrate the measure $\nu$ along the leaves $W_{\alpha}$. More precisely, in this way we can obtain a standard family $\cG_n=(\cW,\nu_n)$, such that
for any measurable set $A\subset M_n$,
$$\nu_n(A)=\int_{\cA}\nu_{\alpha}(W_{\alpha}\cap A) \, d\lambda(\alpha),$$
where $(W_{\alpha},\nu_{\alpha})$ is a standard pair, and $\lambda$ is the probability factor measure on $\cA$. For some $k\leq n$, let $\cR_k=\bigcup_{i>k}M_i$, which contains all the cells with index greater than $k$.  For each unstable curve $W_{\alpha}\in\cW$, if $F^n W_{\alpha}$  crosses $\cR_k$, then $F^m W_{\alpha}$ is cut into pieces by the boundary of cells in $\cR_k$. Moreover, the largest length of these pieces is $\sim k^{-\beta}$.  According to the growth lemma (\ref{firstgrowth}), there exists $\theta_0\in(0,1)$, such that we have
\beq\label{growthest}F^m_*\nu_n(\cR_k)\leq c\vartheta_0^m  F_*\nu_n(\cR_k)+C_z k^{-\beta q_0}.\eeq

\bigskip

 \noindent{\textbf{Case I. Billiards with cusps.}}\\

This class of billiards were first studied by Machta \cite{Ma83}. It is
known that the billiard maps on these
tables are hyperbolic and ergodic. However, the hyperbolicity is
non-uniform. As a result, correlations
decay   with order $ \cO(n^{-1})$, see \cite{CZ05a,CZ08, CM07}. Moreover, it was showed that it  satisfies the One-Step Expansion \textbf{(h4)} with $q_0=1$.

In \cite{CM07} Chernov and Markarian showed that
the induced map $F$ on a subset $M\subset \cM$ has exponential decay
of correlations.
Dynamics of $F$ on billiards with cusps are remarkably different than those
on a stadium when it comes to points travelling
between $m$-cells: if $x \in M_m$ and $F x \in M_k$, then $k\in
B_m=[a_m,b_m]$, with $a_m\asymp \sqrt{m}, b_m\sim m^2$.
And the transition probability from the $m$-cell to the $k$-cell is
$$\mu_M(F x \in M_k | x \in M_m):=
\frac{\mu(\{x\in M_m: F x \in M_k\})}{\mu(M_m)}
\sim\frac{m^{2/3}}{k^{7/3}},$$ with $k\in [\sqrt{m}, m^2]$.
 Each cell $M_m$
 has length approximately  $ m^{-7/3}$
in the unstable direction and length
approximately $ m^{-2/3}$
in the stable direction. Its measure is $\mu(M_m)\sim m^{-3}$.

Moreover, it was checked in \cite{CZ08} at the end of section 5 that this class of billiards satisfies for any small enough $e\in (0,1/4)$:
$$\mu(R(F(x))>n^{\frac{1}{2}+e} |R(x)=n)\leq C n^{-\frac{1}{2e}},$$
for some uniform constant $C>0$.
Since each cell $M_n$ has length approximately  $ n^{-7/3}$
in the unstable direction, we take $\beta=7/3$, and $k=n^{\frac{1}{2}+e}$. Then we have
$$F_*\nu_n(\cR_{n^{\frac{1}{2}+e}})\leq C n^{-\frac{1}{2e}}.$$

Now we apply (\ref{growthest}) to get for any $i=1,\cdots, (b\ln n)^2$,
\begin{align*}\nu_n(R(F^i(x))&>n^{\frac{1}{2}+e})=F^i_*\nu_n(\cR_k)\\
&\leq c\vartheta_0^i  F_*\nu_n(\cR_k)+C_z k^{-\beta}\\
&\leq C\vartheta_0^i  n^{-\frac{1}{2e}}+C_z n^{-\frac{7}{3}(\frac{1}{2}+e)}.\end{align*}
This verifies \textbf{(H3$'$)} with $q=1/2-e$, $p=\frac{1}{2e}$.\\

\noindent{\textbf{Case II. Semi-dispersing billiards.}}
Billiards in a square with a finite number of fixed, disjoint circular obstacles removed are
known as semi-dispersing billiards.
Chernov and Zhang proved \cite{CZ08} that this system has a decay of
correlations bounded by
$\text{const} \cdot n^{-1}$. Here the reduced phase space $M$ is made up
only of collisions with the
circular obstacles. The induced map $F: M \rightarrow M$ is then equivalent
to the well studied Lorentz gas
billiard map without horizon \cite{CZ05a}, which is known to have
exponential decay of correlations
(see \cite{CM07}, for instance). The structure of the $m$-cells $M_m =
\{x\in M : R(x) = m\}$ is
examined thoroughly in \cite{BSC90,BSC91,CM07}. We will use some of the
facts presented in
those references.
Many properties of the $m$-cells and of the induced billiard map in the
semi-dispersing case
are quite similar to those in billiards with cusps. In particular, the
measure of each $m$-cell is again
$\mu_M(M_m) \asymp m^{-3}$, with $u$-dimension approximately $m^{-2}$. Thus we take $\beta=2$. Moreover it satisfies the One-Step Expansion Estimate \textbf{(h4)} with $q_0=1$.
 It is also know that for a point $x\in M_m$,  $Fx\in M_k$
 where $k\in B_m=[a_m,b_m]$, with $$a_m\asymp \sqrt{m}, b_m\asymp
m^2,$$
as in billiards with cusps. One major change is the transition
probabilities between cells. For semi-dispersing
billiards, we have (for admissible $k$) that
$$\mu_M(F x \in M_k | x \in M_m) \asymp \frac{m+k}{k^3}.$$

Moreover, it was checked in \cite{CZ08} at  Section 5 that this class of billiards satisfies for any small enough $e\in (0,1/4)$:
$$\mu(R(F(x))>n^{\frac{1}{2}+e} |R(x)=n)\leq C n^{-\frac{1}{2e}},$$
for some uniform constant $C>0$.
We take  $k=n^{\frac{1}{2}+e}$. Then we have
$$F_*\nu_n(\cR_{n^{\frac{1}{2}+e}})\leq C n^{-\frac{1}{2e}}.$$

Now we apply (\ref{growthest}) to get for any $i=1,\cdots, (b\ln n)^2$,
\begin{align*}\nu_n(R(F^i(x))&>n^{\frac{1}{2}+e})=F^i_*\nu_n(\cR_k)\\
&\leq c F_*\nu_n(\cR_k)+C_z k^{-\beta}\\
&\leq C n^{-\frac{1}{2e}}+C_z n^{-2(\frac{1}{2}+e)}.\end{align*}
This verifies \textbf{(H3)}(b) with $q=1/2-e$, $p=\frac{1}{2e}$.\\

This implies that the semi-dispersing billiards on a rectangle and dispersing billiards with cusps   have optimal bounds of decay rates of correlations  given by Theorem
\ref{TmMain}.

\medskip

\noindent\textbf{Acknowledgement}.
  SV was supported by the ANR-
Project Perturbations, by
the  PICS05968 with U. Houston, and with a CNRS support to the Centro de Modelamiento Matem\`atico,
UMI2807, in Santiago de Chile. S.V. would like to thank UMass Amherst for the kind hospitality during the completion of this work. H.-K.Z.\  was partially supported by NSF grant DMS-1151762, by   a grant from the Simons
Foundation (337646, HZ); and by the French CNRS with a {\em poste d'accueil} position  at the Center of Theoretical Physics in Luminy and by the University of Toulon.

\end{document}